\documentclass[a4paper, authoryear, 12pt]{elsarticle}
\usepackage{lipsum}
\makeatletter
\def\ps@pprintTitle{%
		\let\@oddhead\@empty
		\let\@evenhead\@empty
		\def\@oddfoot{}%
		\let\@evenfoot\@oddfoot}
\makeatother

\usepackage[english]{babel}
\usepackage{babel}
\usepackage{empheq}
\usepackage{fontenc}
\usepackage[table]{xcolor}
\usepackage{boldline} 
\usepackage{graphicx}
\usepackage{caption}
\usepackage{subcaption}
\usepackage{booktabs}
\usepackage{amsmath,amsthm,amssymb}
\usepackage{natbib}
\usepackage[colorlinks=true,citecolor=blue]{hyperref}
\usepackage{color}
\usepackage{graphicx}
\usepackage{color}
\usepackage{newlfont}
\usepackage{multirow}
\usepackage{longtable} %Allows for long tables that stretch across many pages
\usepackage{ifthen}
\usepackage{alltt}
\usepackage{enumerate}
\usepackage[text={7in, 9in},centering]{geometry}
\usepackage{tikz}
\usepackage{epstopdf}
\usepackage{float}
\usepackage{arydshln}
\usepackage{multicol}

\usepackage{tikz}
\usetikzlibrary{shapes,arrows}
\numberwithin{equation}{section}
\newtheorem{theorem}{Theorem}[section]
\newtheorem{lemma}{Lemma}[section]
\newtheorem{proposition}{\bf Proposition}[section]
\newtheorem{remark}{Remark}[section]
\newtheorem{definition}{\hskip\parindent\bf Definition}[section]

\input epsf.sty
\begin{document}

	\newtheorem{mydef}{Definition}
	\newtheorem{thm}{Theorem}
	\newtheorem{prop}{Proposition}
	\newtheorem{corol}{Corollary}
	\newtheorem{rem}{Remark}
	
	\title{Global stability and optimal control in a dengue model with fractional order transmission and recovery process}
	\author[JU]{Tahajuddin Sk}
        \author[IITK]{Kaushik Bal}
        \author[JU]{Santosh Biswas}
	\author[DAC]{Tridip Sardar \footnote{Corresponding author. Email: tridipiitk@gmail.com, Phone: +91-9674243998}}
	\address[JU]{Department of Mathematics,	Jadavpur University, Kolkata, India.}
    \address[IITK]{Department of Mathematics, Indian Institute of Technology-Kanpur, India.}
	\address[DAC]{Department of Mathematics, Dinabandhu Andrews College, Kolkata, India.}

\begin{abstract}
The current manuscript introduce a single strain dengue model developed from stochastic processes incorporating fractional order transmission and recovery. The fractional derivative has been introduced within the context of transmission and recovery process, displaying characteristics similar to tempered fractional ($TF$) derivatives. It has been established that under certain condition, a function's $TF$ derivatives are proportional to the function itself. Applying the following observation, we examined stability of several steady-state solutions, such as disease-free and endemic states, in light of this newly formulated model, using the reproduction number ($R_{0}$). In addition, the precise range of epidemiological parameters for the fractional order model was determined by calibrating weekly registered dengue incidence in the San Juan municipality of Puerto Rico, from April 9, 2010, to April 2, 2011. We performed a global sensitivity analysis method to measure the influence of key model parameters (along with the fractional-order coefficient) on total dengue cases and the basic reproduction number ($R_{0}$) using a Monte Carlo-based partial rank correlation coefficient (PRCC). Moreover, we formulated a fractional-order model with fractional control to asses the effectiveness of different interventions, such as reduction the recruitment rate of mosquito breeding, controlling adult vector, and providing individual protection. Also, we established the existence of a solution for the fractional-order optimal control problem. Finally, the numerical experiment illustrates that, policymakers should place importance on the fractional order transmission and recovery parameters that capture the underline mechanisms of disease along with reducing the spread of dengue cases, carried out through the implementation of two vector controls.

%In this ongoing article, we have introduced a model designed for a single-strain of dengue. The mention model was developed using a stochastic process that involves derivatives of non-integer order. The fractional derivatives has been incorporated into the transmission and recovery process and displays characterstic similar to tempered fractional ($TF$) derivatives.  	
\end{abstract} 
\begin{keyword}
	Fractional order dengue model; Power-law transmission process; Power-law recovery process; Global stability; Optimal control 
\end{keyword}
\maketitle

%%AMS subject classifications. 34A08, 33B20, 65R20, 45G10, 45G15 

\section{Introduction}
The viral infection dengue transmitted to human due to the bite by female Aedes aegypti mosquito carrying any one among the four viral serotypes namely DENV-1, DENV-2, DENV-3, and DENV-4~\citep{WHO-DEN09, gubler1995dengue}. An infected vector stays infected and never fully recovers from the infection because mosquitoes have a short lifespan~\citep{sardar2015mathematical}. The past few decades have seen an increase in dengue fever cases worldwide as a result of rising temperatures~\cite{ong2021adult}. According to some estimates, approximately around 390 million cases of dengue are reported each year in the world, among them 96 million experience clinical symptoms~\citep{bhatt2013global}. The dengue virus is known as acute febrile viral illness that frequently causes headaches, rash, pain in the muscles and bones, and flu-like symptoms~\citep{htun2021clinical}. However, dengue hemorrhagic fever (DHF) or dengue shock syndrome (DSS) is a potentially life threatening version of illness that can cause bleeding, vomiting, and an abrupt drop in blood pressure~\citep{WHO-DEN09, htun2021clinical}. Cases of the severe form of dengue hemorrhagic fever (DHF) rose sharply each year throughout the world~\citep{WHO-DEN09}. The most affected areas by dengue fever are Latin America, the Eastern Mediterranean, Africa, the Western Pacific region, South-East Asia, and the subtropical regions of the world~\citep{WHO-DEN09, gubler2002epidemic,pinheiro1997global}. Controlling the vector population and adopting personal protective measures are two potential intervention strategies for dengue fever, as there is currently no known treatment or vaccine~\citep{WHO-DEN09}. Therefore, in order to offer a more insightful interpretation of the mechanism of dengue transmission and effect of different interventions, a suitable mathematical model is necessary~\citep{andraud2012dynamic, aguiar2022mathematical}.

Over the past ten years, a variety of models based on statistics and mathematics have been formulated to study dengue fever, which yield valuable understanding of the spread of the disease and its management~\citep{gubler2002epidemic,andraud2012dynamic, aguiar2022mathematical, pinho2010modelling}. However, most of these models based on the assumption that the transmission and recovery processes follow an exponential distribution leading to a coupled ordinary differential equation system~\citep{gubler2002epidemic,andraud2012dynamic, aguiar2022mathematical, pinho2010modelling}. The solution generated from such a system in general follows a Markovian process~\citep{sardar2015mathematical,agarwal2013existence, podlubny1999fractional,du2013measuring, dokoumetzidis2009fractional, dokoumetzidis2010fractional}. However, disease transmission typically occurs through a non-markovian mechanism~\citep{sardar2015mathematical,saeedian2017memory, starnini2017equivalence, angstmann2016fractional,sardar2017mathematical}. As a result, a mathematical model of dengue fever developed considering the assumption of non-Markovian transmission may provide some useful information on the disease's spread and control~\citep{angstmann2016fractional, sardar2017mathematical, angstmann2017fractional}. 

A few mathematical models of infectious diseases assuming non-Markovian transmission can be found in the literature~\citep{sardar2015mathematical, angstmann2016fractional, sardar2017mathematical, angstmann2017fractional, angstmann2021general, wu2023global}. However, we could only identify two publications that develop mathematical models with non-Markovian transmission about dengue fever or vector-borne diseases~\citep{sardar2015mathematical, sardar2017mathematical}. In~\cite{sardar2015mathematical}, authors first convert a simple deterministic vector-borne disease model to a system of integral equations. Then they introduced some time-dependent integral kernels in two forces of infection. Furthermore, the authors showed that when time-dependent kernels follow some power-law form, the deterministic dengue system converts to a fractional-order system. Moreover, by studying local stability properties related to different equilibrium solutions of the fractional order model, the authors proved that conventional epidemiological results do not hold for such systems, by lowering the basic reproduction ratio ($R_{0}$) below one may not be sufficient for the eradication of dengue fever~\citep{sardar2015mathematical}. In~\citep{sardar2017mathematical}, authors follow the approach introduced by Angstmann et al.~\citep{angstmann2016fractional,angstmann2017fractional} to develop a fractional-order dengue compartmental system directly from a stochastic process. In their model~\citep{sardar2017mathematical}, fractional-order derivative appeared in the force of infection due to the assumption that the disease spread govern by power-law (non-Markovian) distribution. Furthermore, the author investigate the local stability of different equilibrium solutions and calibrate the fractional-order dengue model to some real-life data~\citep{sardar2017mathematical}.  

Some recent epidemiological findings indicate that both the transmission and recovery processes may exhibit a non-Markovian mechanism~\citep{starnini2017equivalence,lin2020non,sherborne2018mean}. However, in the context of dengue fever or any other vector-borne disease, a mathematical model with a non-Markovian transmission and recovery mechanism seems rare. Furthermore, we found only one work that mathematically analyzed the global stability of different fixed points of a $SIS$ epidemic model equipped with fractional order transmission~\citep{wu2023global}. However, in~\citep{wu2023global}, authors considered only linear fractional-order infection terms, which is different from other models developed in this direction~\citep{angstmann2016fractional,angstmann2017fractional,angstmann2021general}. Furthermore, as far as we are aware, there is not much research done on optimal control problems for these types of mathematical models, where transmission and recovery processes follow some kind of fractional-order dynamics.

The following are the three primary objective of this study:
\begin{enumerate}
	\item Development of a new mathematical model for dengue fever with fractional-order transmission and recovery processes.
	\item Provide mathematical proof of local and global dynamics of different steady-state solution to this new fractional-order dengue model.
	\item We attempt to develop and analyze an optimal intervention problem associated to this new fractional-order dengue system to analyzed the the impact of various dengue interventions.
	\item Test the model to real dengue data set and estimate key model parameters.
\end{enumerate}
The remaining parts of this work is organized in the following manner: Within Section~\ref{sec:Preliminaries}, we made available various definition and proofs related to some results that will be used throughout the manuscript. A new mathematical model of dengue incorporating fractional-order transmission and recovery has been developed in Section~\ref{sec:Model}. A few mathematical aspects such as positive invariance, and boundedness in relation to the new fractional-order dengue system are studied in Part~\ref{sec:Mathematical-properties}. In Section~\ref{Stability of equilibrium point}, investigate local and global stability involving different equilibrium state of the newly developed system including fractional-order transmission, and recovery processes. The development of an optimal control problem in connection with the fractional-order dengue system is provided in Section~\ref{Optimal Control}. In addition, within Section~\ref{estimation} an estimation of the crucial parameter relevant to our model~\ref{EQ:Fractional-order-vector-borne-disease-model} was carried out and subsequent to this, a comprehensive sensitivity analysis was undertaken in Section~\ref{sensitivity}. A numerical investigation of this new fractional-order dengue model along with the optimal control problem is provided in Section~\ref{Result}. Finally, the paper conclude with a brief discussion and summery of the key results in Section~\ref{conclusion}.

\section{Preliminary definition}\label{sec:Preliminaries}

Let's prapare the definitions, lemmas, and theorems that will form the foundation of this manuscript. 
The $\alpha$ order Riemann-Liouville (R-L) and Caputo fractional derivatives and integral of a smooth real function $\mathcal{G}$(x) (see~\citep{podlubny1999fractional}), where, $0<\alpha<1$, are formally defined as follows:
\begin{definition}~\label{Definition-0} {(R-L fractional integral)}
	The fractional-order Riemann-Liouville integral is described as:
	\begin{eqnarray}
		\displaystyle	\ _{RL}I_{t}^{\alpha} \left[\mathcal{G}(t) \right] &=&  \displaystyle\frac{1}{\Gamma(\alpha)}  \int_{0}^{t}  (t-u)^{\alpha-1} \mathcal{G}(s)ds.
		\label{EQ:R-L-integral}
	\end{eqnarray}		
\end{definition}
\begin{definition}~\label{Definition-1} {(R-L fractional derivative)}
	The Riemann-Liouville fractional-order derivative formulated as:
	\begin{eqnarray}\label{EQ:R-L-derivative}
		\displaystyle  \ _{RL}D_{t}^{\alpha} \left[\mathcal{G}(t) \right] &= &\displaystyle \frac{1}{\Gamma(1-\alpha)} \frac{d}{dt} \int_{0}^{t}  (t-u)^{-\alpha} \mathcal{G}(u) du.
	\end{eqnarray}
\end{definition}
\begin{definition}~\label{Definition-2}{(Caputo fractional derivative)} Let's consider the Caputo fractional-order derivative as follows:
	\begin{eqnarray}\label{EQ:Caputo-derivative}
		\displaystyle \ _{C}D_{t}^{\alpha} \left[\mathcal{G}(t) \right]  &= &\displaystyle \frac{1}{\Gamma(1-\alpha)}  \int_{0}^{t}  (t-u)^{-\alpha} \mathcal{G}'(u) du.
	\end{eqnarray}	
\end{definition}
Relation between the R-L and the Caputo fractional derivatives and their Laplace transformations (see~\citep{podlubny1999fractional}) are provided below:
\begin{eqnarray}\label{EQ:Relation R-L and Caputo}
	\displaystyle  \ _{RL}D_{t}^{\alpha} \left[\mathcal{G}(t) \right]  &=& \displaystyle \ _{C}D_{t}^{\alpha} \left[\mathcal{G}(t) \right] + \frac{\mathcal{G}(0) t^{-\alpha}}{\Gamma(1-\alpha)}.
\end{eqnarray}
and 
\begin{eqnarray}\label{EQ:Laplace R-L}
	\displaystyle  L\left[ \ _{RL}D_{t}^{\alpha} \left( \mathcal{G}(t)\right);s\right]  &= &\displaystyle s^{\alpha}L[\mathcal{G}(t)].
\end{eqnarray}
\begin{definition}~\label{Definition-3}
	Let the function $\gamma(t,y)$, with $y > 0$ and $t \geq 0$, be described as a lower incomplete gamma function that is explicitly expressed by:
	\begin{eqnarray}\label{EQ:Lower-Incomplete Gamma function}
		\displaystyle	\gamma (t,y):=\int_{0}^{t} u^{y -1} e^{-u} du.
	\end{eqnarray}
\end{definition}
\begin{remark}~\label{Definition-4}
	The following integral relation holds:
	\begin{eqnarray}\label{EQ:LIG With fractional derivative}
		\displaystyle	\int_{t_{1}}^{t_{2}} (t-s)^{y-1} e^{-l (t-s)} ds  &=& \displaystyle \frac{1}{l^{y}} \Bigg[\gamma\Big(l(t-t_{1}),y\Big) - \gamma\Big(l(t-t_{2}),y\Big)\Bigg]
	\end{eqnarray}where, $t_{1}, t_{2} >0$.    
\end{remark}
\begin{theorem}~\label{Theorem-1}
	Let $F \in \mathbb{C}[0, T]$, $F(t) \geq 0$, $\forall t \geq 0$ and $\beta > 0$, then we have, $e^{-\beta t}\ _{RL}D_{t}^{y}\left[e^{\beta t} F(t)\right] \approx \beta^{y} F(t)$.
\end{theorem}
\begin{proof} Note that,
	\begin{equation}
		e^{{-\beta} t}\ _{RL}I_{t}^{y} \left[e^{\beta t}F(t) \right] = \frac{1}{\Gamma(y)} \int_{0}^{t}  (t-s)^{y-1} e^{-\beta\left(t-s\right)} F(s) ds
	\end{equation} 
Given that $F \in \mathbb{C}[0, t]$, and $g(t) = t^{y-1} e^{-\beta t}$ is an integrable function which does not change sign in $[0, t]$, applying the Mean Value Theorem of definite integrals as a result
	\begin{eqnarray}
		\begin{array}{llll}
			\displaystyle	e^{-\beta t}\ _{RL}I_{t}^{y} \left[e^{\beta t}F(t) \right] &=& \displaystyle \frac{1}{\Gamma(y)} F(\xi) \int_{0}^{t}  (t-s)^{y-1} e^{-\beta\left(t-s\right)} ds,~~\text{where}~{} \xi \in(0,t) \\ &=& \displaystyle \frac{1}{\Gamma(y)} F(\xi) \frac{\gamma (\beta t,y)}{\beta^{y}},~\hspace{1cm}[\text{using Remark}~ \eqref{Definition-4}],\\ \displaystyle \implies \ _{RL}I_{t}^{y} \left[e^{\beta t}F(t) \right] &=&\displaystyle \frac{e^{\beta t}}{\Gamma(y)} F(\xi) \frac{\gamma (\beta t,y)}{\beta^{y}}, \\ \displaystyle \implies e^{\beta t}F(t) &=& \displaystyle\ _{RL}D_{t}^{y}\left[\frac{e^{\beta t}}{\Gamma(y)} F(\xi) \frac{\gamma (\beta t,y)}{\beta^{y}}\right],\\ \displaystyle \implies \beta^{y} F(t) &=&\displaystyle \frac{e^{-\beta t}}{\Gamma(y)} \ _{RL}D_{t}^{y}\left[ e^{\beta t}F(\xi) \gamma (\beta t,y)\right],\\ &=&\displaystyle \frac{e^{-\beta t}}{\Gamma(y)} \frac{1}{\Gamma(1-y)} \frac{d}{dt} \int_{0}^{t}  (t-s)^{-y} e^{\beta s}F(s) \gamma (\beta s,y) ds.
		\end{array}
	\end{eqnarray}

	Again, the function $\gamma (\beta t,y)$ is a continuous and since, $F(t) \geq 0$, $\forall t \geq 0$, then, $(t-s)^{-y} e^{\beta s} F(s)$, where, $0 < s < t$, is integrable and does not changes sign in $[0, t]$. Consequently, by utilizing the Mean Value Theorem of definite integrals yields,
	\begin{eqnarray}~\label{EQ:Final-results}
		\begin{array}{llll}
			\displaystyle	\beta^{y} F(t) &=&\displaystyle \frac{e^{-\beta t}}{\Gamma(y)} \frac{1}{\Gamma(1-y)} \frac{d}{dt}\left[\gamma (\beta \xi_{1},y) \int_{0}^{t}  (t-s)^{-y} e^{\beta s}F(s) ds\right],~~ \text{where}~~{} \xi_{1} \in(0,t),\\ &=&\displaystyle\frac{\gamma (\beta \xi_{1},y)}{\Gamma(y)} e^{-\beta t}\ _{RL}D_{t}^{y} \left[e^{\beta t}F(t) \right]+ \frac{e^{-\beta \xi_{1} } \left(\beta \xi_{1}\right)^{y-1} \gamma (\beta t,1-y)}{\beta^{y}\Gamma(y)\Gamma(1-y)} F(\xi_{2})~\hspace{0.1cm}\text{where,}~{} \xi_{1}, \xi_{2} \in(0,t).
		\end{array}
	\end{eqnarray}
	
	Now, $\frac{\gamma (\beta t,1-y)}{\Gamma(1-y)} \to 1$ as $t \to \infty$ and $e^{-\beta \xi_{1} } \left(\beta \xi_{1}\right)^{y-1} \to 0$ as $t \to \infty$. Hence~\eqref{EQ:Final-results}
	can be written as
	
	\begin{eqnarray}
		\begin{array}{llll}
			\displaystyle	e^{-\beta t}\ _{RL}D_{t}^{y} \left[e^{\beta t}F(t) \right]&\approx & \displaystyle \beta^{y} F(t)
		\end{array}
	\end{eqnarray}
	\vspace{0.1cm}
\end{proof}

\section{{Model formulation}}\label{sec:Model}
The following assumption are made to develop a fractional-order model for vector-borne disease, involving transmission and recovery process:

\begin{enumerate}
	\item [$\bullet$] Total human population $N_{H}$ is constant.
	\item [$\bullet$] Average biting rate is $\frac{b}{N_{H}}$.
	\item [$\bullet$] $\displaystyle \rho_{V}(t-s)$ is the intrinsic infection age in an infected vector. This expression depending on the time since infection $t-s$ as well as present time $t$.
	\item [$\bullet$] $\displaystyle \rho_{H}(t-s)$ is the intrinsic infection age in an infected human which depend on current time $t$, and time since infection occurred $t-s$.
	\item [$\bullet$] $\beta_{V}^{H}$ is probability of transmission from vector to human.
	\item [$\bullet$] $\beta_{H}^{V}$ is the rate of transmission probability from human to vector is .
	\item [$\bullet$] The human population is partitioned among three mutually exclusive categories: susceptible human ($S_{H}$), infected human ($I_{H}$), and recovered human ($R_{H}$).
	\item [$\bullet$] $(N_{V})$ the vector populations is subdivided in two mutually exclusive groups: susceptible vector ($S_{V}$), and infected vector ($I_{V}$).
	\item [$\bullet$] Considering short life span of the vector, we assumed no recovery for those who have been infected.
\end{enumerate}

Now, the average incidence of newly infected humans generated within the time domain $[t, t+\Delta t]$ by a single vector infected at earlier in time $s < t$ is defined as
\begin{eqnarray}\label{EQ:Average-infected-human}
	\begin{array}{llll}
		\displaystyle \frac{b}{N_{H}} \beta_{V}^{H} \rho_{V}(t-s) S_{H}(t) \Delta t + o(\Delta t).
	\end{array}
\end{eqnarray}
Following~\cite{sardar2017mathematical}, we consider that the infection first occurred at $t = 0$. Furthermore, the survival probability $\Phi_{I_{V}}(t,s)$ of a vector individual once infected at time $t$, provided it enters this compartment during the time $s < t$. Therefore, at $t$, the average number of new infections is:

\begin{eqnarray}\label{EQ:Newflux-infected-human}
	\begin{array}{llll}
		\displaystyle Q(I_{H},t) &=& \displaystyle\frac{b}{N_{H}} \beta_{V}^{H} S_{H}(t) \int_{0}^{t} \rho_{V}(t-s) \Phi_{I_{V}}(t,s) Q(I_{V},s) ds,
	\end{array}
\end{eqnarray}
where, average number of infected vector population is $Q(I_{V},t)$ at time $t$. 

An infected vector may only transition from its respective compartment through natural death. We assume that the death process of vectors follows an exponential distribution. Therefore,

\begin{eqnarray}\label{EQ:death-process-vector}
	\begin{array}{llll}
		\displaystyle \Phi_{I_{V}}(t,s) &= &\displaystyle e^{-\int_{s}^{t} \mu_{v} ds} = \displaystyle e^{- \mu_{v} (t-s)},
	\end{array}
\end{eqnarray}
where, the natural mortality rate of the vectors is $\mu_{v}$. 

The number of individual with infections at $t$ is given by

\begin{eqnarray}\label{EQ:Infected-human-at-t}
	\begin{array}{llll}
		\displaystyle I_{H}(t) &=& \displaystyle \int_{0}^{t} \Phi_{I_{H}}(t,s) Q(I_{H},s) ds,
	\end{array}
\end{eqnarray}
where, the survival probability of an infected individual who moved into the specified state at time $s$ is $\Phi_{I_{H}}(t,s)$, and will continue to stay in this compartment until time $t$ $(s < t)$. Furthermore, $Q(I_{H},t)$ is provided in equation~\eqref{EQ:Newflux-infected-human}.

From the hypothesis, an individua can move the infected human compartment in two ways, either become achieve recovery from the infection or by natural death. We assume these two processes are independent. Thus probability of survival of an individual in the infected human state is provided below:

\begin{eqnarray}\label{EQ:Survival-probablity-infected-human}
	\begin{array}{llll}
		\displaystyle \Phi_{I_{H}}(t,s) &= &\displaystyle \phi_{h}^{R}(t-s) \theta_{h}^{D}(t,s),
	\end{array}
\end{eqnarray}
where, $\phi_{h}^{R}(t-s)$, and $\theta_{h}^{D}(t,s)$ are the probabilities that an entity has not recovered and deceased at time $t$, respectively, while those entered in this compartment ($I_{H}$) at time $s < t$. Furthermore, it is assume that the death process of humans is governed by the exponential distribution provided below:

\begin{eqnarray}\label{EQ:Death-survivalprobability-infectedhuman}
	\begin{array}{llll}
		\displaystyle \theta_{h}^{D}(t,s) &=& \displaystyle e^{-\int_{s}^{t} \mu_{H} ds} = \displaystyle e^{- \mu_{H} (t-s)},
	\end{array}
\end{eqnarray}
where, $\mu_{H}$ represent the natural mortality rate of the human populations.   

Therefore, using~\eqref{EQ:Survival-probablity-infected-human}, we have from~\eqref{EQ:Infected-human-at-t}:

\begin{eqnarray}\label{EQ:Infected-human-compartment-1}
	\begin{array}{llll}
		\displaystyle \frac{dI_{H}}{dt} &= &\displaystyle Q(I_{H},t) - \int_{0}^{t} \Psi(t-s) \theta_{h}^{D}(t,s) Q(I_{H},s) ds -\mu_{H} I_{H},
	\end{array}
\end{eqnarray}
where, $\displaystyle \Psi(t) =\displaystyle - \frac{d \phi_{h}^{R}(t)}{dt}$.
Hence from ~\eqref{EQ:Newflux-infected-human} we have:

\begin{eqnarray}\label{EQ:Infected-human-compartment-2}
	\begin{array}{llll}
		\displaystyle \frac{dI_{H}}{dt} &= &\displaystyle \frac{b}{N_{H}} \beta_{V}^{H} S_{H}(t) \int_{0}^{t} \rho_{V}(t-s) \theta_{i_{v}}^{D}(t,s) Q(I_{V}, s) d\tau \\ &-& \displaystyle\int_{0}^{t} \Psi(t-s) \theta_{h}^{D}(t,s) Q(I_{H},s) ds -\mu_{H} I_{H},
	\end{array}
\end{eqnarray}
where, $\theta_{i_{v}}^{D}(t,s) = e^{-\mu_{v} (t-s)}$.

Similarly, the average incidence of newly infected vector at time $t$ is:  

\begin{eqnarray}\label{EQ:Newflux-infected-vector}
	\displaystyle Q(I_{V},t) &= &\displaystyle \frac{b}{N_{H}} \beta_{H}^{V} S_{V}(t) \int_{0}^{t} \rho_{H}(t-s) \Phi_{I_{H}}(t,s) Q(I_{H},s) ds,
\end{eqnarray}
where, $Q(I_{H},t)$ is defined in equation~\eqref{EQ:Newflux-infected-human}. Using~\eqref{EQ:Survival-probablity-infected-human}, we have:

\begin{eqnarray}\label{EQ:Newflux-infected-vector-1}
	\displaystyle Q(I_{V},t) &=& \displaystyle\frac{b}{N_{H}} \beta_{H}^{V} S_{V}(t) \int_{0}^{t} \rho_{H}(t-s) \phi_{h}^{R}(t-s) \theta_{h}^{D}(t, s) Q(I_{H},s) ds.
\end{eqnarray}

Therefore, the infected number of vectors at $t$ time is as follows:

\begin{eqnarray}\label{EQ:Infected-vector-at-t}
	\displaystyle I_{V}(t) &= &\displaystyle \int_{0}^{t} \Phi_{I_{V}}(t,s) Q(I_{V},s) ds.
\end{eqnarray}
Taking derivative in equation~\eqref{EQ:Infected-vector-at-t}, and using the relation in~\eqref{EQ:death-process-vector} we have:
\begin{eqnarray}\label{EQ:Infected-vector-compartment-1}
	\displaystyle \frac{dI_{V}}{dt} &=& \displaystyle Q(I_{V},t) -\mu_{V} I_{V}.
\end{eqnarray} Using relation~\eqref{EQ:Newflux-infected-vector-1}, we have:

\begin{eqnarray}\label{EQ:Infected-vector-compartment-2}
	\displaystyle \frac{dI_{V}}{dt}&=&\displaystyle \frac{b}{N_{H}} \beta_{H}^{V} S_{V}(t) \int_{0}^{t} \rho_{H}(t-s) \phi_{h}^{R}(t-s) \theta_{h}^{D}(t,s) Q(I_{H}, s) -\mu_{V} I_{V}.
\end{eqnarray}

Now, both $\theta_{h}^{D}(t,s)$, and $\theta_{i_{v}}^{D}(t,s)$ satisfy the semi-group property \textit{i.e.}
\begin{eqnarray}\label{EQ:Semi-group-property}
	\begin{array}{llll}
		\displaystyle \theta_{h}^{D}(t,s)& = &\displaystyle \theta_{h}^{D}(t,u) \theta_{h}^{D}(u,s),\\ \displaystyle \theta_{i_{v}}^{D}(t,s)& = &\displaystyle \theta_{i_{v}}^{D}(t,u) \theta_{i_{v}}^{D}(u,s),\hspace{0.4cm} \text{where} \hspace{0.4cm} s < u < t.
	\end{array}
\end{eqnarray} Using semi-group property the systems~\eqref{EQ:Infected-human-compartment-2} and~\eqref{EQ:Infected-vector-compartment-2} become:

\begin{eqnarray}
	\begin{array}{llll}
		\displaystyle \frac{d I_{H}}{dt} &=&\displaystyle \frac{b}{N_{H}} \beta_{V}^{H} S_{H}(t) \theta_{i_{v}}^{D}(t,0) \int_{0}^{t} \rho_{V}(t-s) \frac{Q(I_{V}, s)}{\theta_{i_{v}}^{D}(s,0)} ds\\ &-&\displaystyle \theta_{h}^{D}(t,0) \int_{0}^{t} \Psi(t-s) \frac{Q(I_{H}, s)}{\theta_{h}^{D}(s,0)} ds -\mu_{H} I_{H}
	\end{array}
	\label{EQ:Infected-human-compartment-3}
\end{eqnarray} and

\begin{eqnarray}
	\begin{array}{llll}
		\displaystyle \frac{d I_{V}}{dt} &=& \displaystyle \frac{b}{N_{H}} \beta_{H}^{V} S_{V}(t) \theta_{h}^{D}(t,0) \int_{0}^{t} \rho_{H}(t-s) \phi_{h}^{R}(t-s) \frac{Q(I_{H}, s)}{\theta_{h}^{D}(s,0)} ds -\mu_{V} I_{V}.
	\end{array}
	\label{EQ:Infected-vector-compartment-3}
\end{eqnarray}
Now again using semi-group property~\eqref{EQ:Semi-group-property}, we have from~\eqref{EQ:Infected-vector-at-t}:

\begin{eqnarray}
	\begin{array}{llll}
		\displaystyle \frac{I_{V}(t)}{\theta_{i_{v}}^{D}(t,0)} &=& \displaystyle \int_{0}^{t} \frac{Q(I_{V},s)}{\theta_{i_{v}}^{D}(s,0)} ds.
	\end{array}
	\label{EQ:Infected-vector-at-time-t1}
\end{eqnarray}
Taking Laplace transform w.r.t. $t$ in the both side of~\eqref{EQ:Infected-vector-at-time-t1}, we have:
\begin{eqnarray}~\label{EQ:Infected-vector-at-time-t2}
	\displaystyle L\left[\frac{I_{V}(t)}{\theta_{i_{v}}^{D}(t,0)}\right] &= &\displaystyle \frac{1}{q} L\left[\frac{Q(I_{V}, t)}{\theta_{i_{v}}^{D}(t,0)}\right].
\end{eqnarray}
As a result of equation~\eqref{EQ:Infected-human-compartment-3}, the first integrand on the right-hand side can be converted as:

\begin{eqnarray}\label{EQ:Integral-in-Infected-human-compartment-3}
	\begin{array}{llll}
		\displaystyle L\left[\int_{0}^{t} \rho_{V}(t-s) \frac{Q(I_{V}, s)}{\theta_{i_{v}}^{D}(s,0)} \right] &= &\displaystyle L\left[ \kappa_{V}(t)\right] L\left[\frac{I_{V}(t)}{\theta_{i_{v}}^{D}(t,0)}\right],\\ &=&\displaystyle \int_{0}^{t} \kappa_{V}(t-s) \frac{I_{V}(s)}{\theta_{i_{v}}^{D}(s,0)} ds,
	\end{array}
\end{eqnarray}
where, 

\begin{eqnarray}\label{EQ:kappaV}
	\displaystyle \kappa_{V}(t) &= &\displaystyle L^{-1}\left\{q L\left[\rho_{V}(t)\right]  \right\}. 	
\end{eqnarray}
Again from~\eqref{EQ:Infected-human-at-t}, we have:

\begin{eqnarray}\label{EQ:Infected-human-at-time-t2}
	\displaystyle L\left[\frac{I_{H}(t)}{\theta_{h}^{D}(t,0)}\right] &=& \displaystyle L\left[\phi_{h}^{R}(t)\right] L\left[\frac{Q(I_{H}, t)}{\theta_{h}^{D}(t,0)}\right].
\end{eqnarray}
Therefore, the right-hand side of the second integrand in the equation~\eqref{EQ:Infected-human-compartment-3}, can be transformed as:
\begin{eqnarray}\label{EQ:Second-Integral-in-Infected-human-compartment-3}
	\begin{array}{llll}
		\displaystyle L\left[\int_{0}^{t} \Psi(t-s) \frac{Q(I_{H}, s)}{\theta_{h}^{D}(s,0)} \right] &=& \displaystyle L\left[ \kappa_{R}(t)\right] L\left[\frac{I_{H}(t)}{\theta_{h}^{D}(t,0)}\right],\\ &=& \displaystyle \int_{0}^{t} \kappa_{R}(t-s) \frac{I_{H}(s)}{\theta_{h}^{D}(s,0)} ds,
	\end{array}	
\end{eqnarray}
where,

\begin{eqnarray}\label{EQ:kappaR}
	\displaystyle \kappa_{R}(t) &= &\displaystyle L^{-1}\left\{\frac{L\left[\Psi(t)\right]}{L\left[\phi_{h}^{R}(t)\right]}  \right\}. 	
\end{eqnarray}
Therefore, using~\eqref{EQ:Integral-in-Infected-human-compartment-3}, and~\eqref{EQ:Second-Integral-in-Infected-human-compartment-3}, we have the general integro-differential equation of the infected human compartment:

\begin{eqnarray}\label{EQ:Infected-human-compartment-4}
	\begin{array}{llll}
		\displaystyle \frac{d I_{H}}{dt} &=& \displaystyle \frac{b}{N_{H}} \beta_{V}^{H} S_{H}(t) e^{-\mu_{V} t} \int_{0}^{t} \kappa_{V}(t-s) I_{V}(s) e^{\mu_{V} s} ds\\
		&-&  \displaystyle e^{-\mu_{H} t} \int_{0}^{t} \kappa_{R}(t-s) I_{H}(s) e^{\mu_{H} s} ds -\mu_{H} I_{H}.
	\end{array}	
\end{eqnarray}

Proceeding similarly, we have the general integro-differential equation of the infected vector compartment provided below:
\begin{eqnarray}\label{EQ:Infected-vector-compartment-4}
	\displaystyle \frac{d I_{V}}{dt} &=& \displaystyle \frac{b}{N_{H}} \beta_{H}^{V} S_{V}(t) e^{-\mu_{H} t} \int_{0}^{t} \kappa_{H}(t-s) I_{H}(s) e^{\mu_{H} s} ds -\mu_{V} I_{V},
\end{eqnarray}where 

\begin{eqnarray}\label{EQ:kappah}
	\displaystyle \kappa_{H}(t) &=& \displaystyle L^{-1}\left\{\frac{L\left[ \rho_{H}(t) \phi_{h}^{R}(t)\right]}{L\left[\phi_{h}^{R}(t)\right]} \right\},
\end{eqnarray}

Therefore, a general integro-differential equation representing the interaction between humans and vectors is provided below:
\begin{eqnarray}\label{EQ:inegro-differential-vector-borne}
	\begin{array}{llll}
		\displaystyle \frac{dS_{H}}{dt} &=&\displaystyle \mu_{H} N_{H} - \frac{b}{N_{H}} \beta_{V}^{H} S_{H}(t) e^{-\mu_{V} t} \int_{0}^{t} \kappa_{V}(t-s) I_{V}(s) e^{\mu_{V} s} ds -\mu_{H} S_{H},\\
		\displaystyle \frac{dI_{H}}{dt}&=& \displaystyle \frac{b}{N_{H}} \beta_{V}^{H} S_{H}(t) e^{-\mu_{V} t} \int_{0}^{t} \kappa_{V}(t-s) I_{V}(s) e^{\mu_{V} s} ds- e^{-\mu_{H} t} \int_{0}^{t} \kappa_{R}(t-s) I_{H}(s) e^{\mu_{H} s} ds\\ & -& \displaystyle  \mu_{H} I_{H},\\
		\displaystyle \frac{dR_{H}}{dt}&=& \displaystyle e^{-\mu_{H} t} \int_{0}^{t} \kappa_{R}(t-s) I_{H}(s) e^{\mu_{H} s} ds - \mu_{H} R_{H},\\ 	\displaystyle \frac{dS_{V}}{dt}&= &\displaystyle \Pi_{V}- \frac{b}{N_{H}} \beta_{H}^{V} S_{V}(t) e^{-\mu_{H} t} \int_{0}^{t} \kappa_{H}(t-s) I_{H}(s) e^{\mu_{H} s} ds -\mu_{V} S_{V},\\
		\displaystyle  \frac{dI_{V}}{dt}&= &\displaystyle \frac{b}{N_{H}} \beta_{H}^{V} S_{V}(t) e^{-\mu_{H} t} \int_{0}^{t} \kappa_{H}(t-s) I_{H}(s) e^{\mu_{H} s} ds -\mu_{V} I_{V},
	\end{array}
\end{eqnarray}
where, $\kappa_{V}(t)$, $\kappa_{R}(t)$, and $\kappa_{H}(t)$ are provided in equations~\eqref{EQ:kappaV},~\eqref{EQ:kappaR}, and~\eqref{EQ:kappah}, respectively.  

Following~\cite{ angstmann2016fractional,sardar2017mathematical, angstmann2016fractionalA}, the fractional derivatives can be incorporated in the model~\eqref{EQ:inegro-differential-vector-borne} by choosing the intrinsic infection age of the vector function $\rho_{V}(t)$ in power-law form as follows: 

\begin{eqnarray}\label{EQ:power-law-intrinsic-infection-age-vector}
	\displaystyle \rho_{V}(t) &=& \displaystyle \frac{t^{\alpha-1}}{\Gamma(\alpha)},\hspace{0.4cm} 0<\alpha \leq 1.
\end{eqnarray}

Using~\eqref{EQ:power-law-intrinsic-infection-age-vector}, and following~\cite{angstmann2016fractional,sardar2017mathematical,angstmann2016fractionalA} we have:

\begin{eqnarray}\label{EQ:first-integral-susceptible-human}
	\displaystyle \int_{0}^{t} \kappa_{V}(t-s) I_{V}(s) e^{\mu_{V} s} ds &= &\displaystyle _{RL}{D_{t}}^{1-\alpha}\left[I_{V}(t) e^{\mu_{V} t} \right],
\end{eqnarray}

Furthermore, following~\cite{ angstmann2017fractional,hilfer1995fractional}, $\Psi(t)$ the probability density function can be considered in Mittag-Leffler form as follows:

\begin{equation}\label{EQ:psit-Mittag-leffler}
	\displaystyle \Psi(t) = \displaystyle \frac{t^{\beta-1}}{C^{\beta}} E_{\beta,\beta} \left(-\left(\frac{t}{C}\right)^{\beta} \right),\hspace{0.4cm} 0< \beta \leq 1,	
\end{equation}
where, the two parameter Mittag-Leffler function is $\displaystyle E_{r, l} = \sum_{k=0}^{\infty} \frac{z^{k}}{\Gamma(r k + l)}$, and $C$ is a scaling parameter. Density function for sufficiently large $t$ provided in~\eqref{EQ:psit-Mittag-leffler} has a power-law tail with $\displaystyle \Psi(t) \sim  t^{-1-\beta}$. From~\eqref{EQ:Infected-human-compartment-1}, the survival function $\phi_{h}^{R}(t)$ become:

\begin{equation}\label{EQ:Survival-probablity}
	\displaystyle \phi_{h}^{R}(t) = \displaystyle E_{\beta,1} \left(-\left(\frac{t}{C}\right)^{\beta} \right),\hspace{0.4cm} 0< \beta \leq 1. 
\end{equation} 
Following~~\cite{angstmann2017fractional}, we have:

\begin{eqnarray}\label{EQ:kappaR-power-law}
	\begin{array}{llll}
		\displaystyle \kappa_{R}(t) &=& \displaystyle L^{-1} \left[\frac{L\left\{ \Psi(t)\right\}}{L\left\{\phi_{h}^{R}(t)\right\}}\right]\\ &=&\displaystyle q^{1-\beta} C^{-\beta}
	\end{array}	
\end{eqnarray} Using~\eqref{EQ:kappaR-power-law}, we have:

\begin{eqnarray}\label{EQ:second-integral-infected-human}
	\displaystyle \int_{0}^{t} \kappa_{R}(t-s) I_{H}(\tau) e^{\mu_{H} s} ds &=& \displaystyle \frac{_{RL}D_{t}^{1-\beta}\left[I_{H}(t) e^{\mu_{H} t} \right]}{C^{\beta}}.
\end{eqnarray}Following~~\cite{ angstmann2016fractional,angstmann2017fractional}, we choose the intrinsic infection age of an infected human as follows:

\begin{eqnarray}\label{EQ:power-law-intrinsic-infection-age-human}
	\begin{array}{llll}
		\displaystyle \rho_{H}(t) &=& \displaystyle \frac{t^{p-1}}{\phi_{h}^{R}(t) C^{p}} E_{\beta, p}\left(-\left(\frac{t}{C}\right)^{\beta} \right),\hspace{0.4cm} 0<\beta, p \leq 1.
	\end{array}	
\end{eqnarray}

As the age of infection must be a non-negative function, therefore, the condition for $\rho_{H} (t) \geq 0$ is the fractional coefficients $\beta$, and $p$ must satisfy $0 < \beta \leq p \leq 1$. Following~~\cite{angstmann2017fractional}, we have:

\begin{eqnarray}\label{EQ:kappah-power-law}
	\displaystyle \kappa_{H}(t) &= &\displaystyle q^{1-p} C^{-p}.  	
\end{eqnarray}
 Using~\eqref{EQ:kappah-power-law}, we have:
\begin{eqnarray}\label{EQ:third-integral-infected-vector}
	\displaystyle \int_{0}^{t} \kappa_{H}(t-s) I_{H}(s) e^{\mu_{H} s} ds &=& \displaystyle \frac{_{RL}D_{t}^{1-p}\left[I_{H}(t) e^{\mu_{H} t} \right]}{C^{p}},\hspace{0.4cm} 0 <\beta \leq p \leq 1.
\end{eqnarray}

Therefore, we obtain the vector-borne disease model equipped with fractional-order transmission and recovery process after substituting~\eqref{EQ:first-integral-susceptible-human},~\eqref{EQ:second-integral-infected-human}, and~\eqref{EQ:third-integral-infected-vector} in system~\eqref{EQ:inegro-differential-vector-borne}:

\begin{eqnarray}\label{EQ:Fractional-order-vector-borne-disease-model}
	\begin{array}{llll}
		\displaystyle \frac{dS_{H}}{dt} &=& \displaystyle \mu_{H} N_{H} - \frac{b^{\alpha}}{N_{H}} \beta_{V}^{H} S_{H}(t) e^{-\mu_{V} t}~ _{RL}{D_{t}}^{1-\alpha}\left[I_{V}(t) e^{\mu_{V} t} \right] -\mu_{H} S_{H},\\
		\displaystyle  \frac{dI_{H}}{dt}&=& \displaystyle \frac{b^{\alpha}}{N_{H}} \beta_{V}^{H} S_{H}(t) e^{-\mu_{V} t}~ _{RL}{D_{t}}^{1-\alpha}\left[I_{V}(t) e^{\mu_{V} t} \right]- \frac{e^{-\mu_{H} t}}{C^{\beta}}~_{RL}D_{t}^{1-\beta}\left[I_{H}(t) e^{\mu_{H} t} \right] -\mu_{H} I_{H},\\
		\displaystyle \frac{dR_{H}}{dt}&= &\displaystyle \frac{e^{-\mu_{H} t}}{C^{\beta}}~_{RL}D_{t}^{1-\beta}\left[I_{H}(t) e^{\mu_{H} t} \right] - \mu_{H} R_{H},\\
		\displaystyle  \frac{dS_{V}}{dt}&=& \displaystyle \Pi_{V}- \frac{b^{p}}{N_{H} C^{p}} \beta_{H}^{V} S_{V}(t) e^{-\mu_{H} t}~_{RL}D_{t}^{1-p}\left[I_{H}(t) e^{\mu_{H} t}\right]  -\mu_{V} S_{V},\\
		\displaystyle  \frac{dI_{V}}{dt}&= &\displaystyle \frac{b^{p}}{N_{H} C^{p}} \beta_{H}^{V} S_{V}(t) e^{-\mu_{H} t} ~_{RL}D_{t}^{1-p}\left[I_{H}(t) e^{\mu_{H} t}\right] -\mu_{V} I_{V},
	\end{array}
\end{eqnarray} 
where, $0 < \alpha \leq 1$, and $0 < \beta \leq p \leq 1$, $S_{H}(0)=S_{0}^{H}\geq 0$, $I_{H}(0)=I_{0}^{H}\geq 0$, $R_{H}(0)=R_{0}^{H}\geq 0$, $S_{V}(0)=S_{0}^{V}\geq 0$, and $I_{V}(0)=I_{0}^{V}\geq 0$ and $\Pi_{V}=\delta \times N_{H}$. Furthermore, we assumed that all model~\eqref{EQ:Fractional-order-vector-borne-disease-model} parameters are positive.

\begin{table}[H]
	\tabcolsep 2 pt
	\centering
	\caption{\bf{ In the context of the Dengue model~\eqref{EQ:Fractional-order-vector-borne-disease-model}, all the parameters are positive. The ranges and biological significance of these parameters are provided }}
	\begin{tabular}{p{2cm} p{8cm} p{2cm} p{2cm}}
		\footnotesize{{Parameter}} &\footnotesize{{Epidemiological meaning}} & \footnotesize{{Value}} &\footnotesize{{Reference}}\\  \hline\\
		\footnotesize{$N_{H}$} & \footnotesize{Total Human Population}  &
		\footnotesize{$\textup{2347833}$\; $\textup{}$ } & \footnotesize{\cite{Puerto}} \\\\
		\footnotesize{$1/\mu_{H}$}  & \footnotesize{The mean life expectancy of individuals at birth} & \footnotesize{$79.25$}
		\; $\textup{years}$ & \footnotesize{\cite{sardar2017mathematical}}\\\\  
		\footnotesize{$\alpha $} & \footnotesize{The power law exponent governs the force of infection transmitted by an infected vector to susceptible human}  &
		\footnotesize{$(0 - 1)$
			\; $\textup{}$ } & \footnotesize{\cite{sardar2017mathematical}} \\\\
		\footnotesize{$\beta$} & \footnotesize{The order of fractional derivative in recovery term}  &
		\footnotesize{$(0 - 1)$
			\; $\textup{}$ } & \footnotesize{\cite{sardar2017mathematical}} \\\\
		\footnotesize{$p$} & \footnotesize{The fractional-order derivative denotes the force of infection spread by infected human to susceptible vector}  &
		\footnotesize{$(0 - 1)$
			\; $\textup{}$ } & \footnotesize{\cite{andraud2012dynamic,sardar2017mathematical}}\\\\
		\footnotesize{$\beta_{V}^{H}$} & \footnotesize{Probability of disease transmission from infected vector to susceptible human}  & 
		\footnotesize{$(0 - 1)$
			\; $\textup{}$ } & \footnotesize{\cite{andraud2012dynamic,sardar2017mathematical}}\\\\
		\footnotesize{$\beta_{H}^{V}$} & \footnotesize{Transmission probability of susceptible vector from infected human}  &
		\footnotesize{$(0 - 1)$
			\; $\textup{}$ } & \footnotesize{\cite{sardar2017mathematical}} \\\\
		\footnotesize{$b$} & \footnotesize{Average number of daily biting rate per mosquito}&
		\footnotesize{$(1.7 - 7)$
			\; $\textup{day}^{-1}$ } &  \footnotesize{\cite{andraud2012dynamic,pinho2010modelling,sardar2017mathematical}} \\\\
		\footnotesize{$\mu_{V}$} & \footnotesize{Mosquito's mortality rate}  &
		\footnotesize{$(0.14 - 1.75)$
			\; $\textup{day}^{-1}$ }  & \footnotesize{\cite{andraud2012dynamic,sardar2017mathematical}}\\\\
		\footnotesize{$C$} & \footnotesize{Scaling constant in the probability density function $\Psi(t)$ defined in~\eqref{EQ:psit-Mittag-leffler}}  &
		\footnotesize{$1.0$
			\; $\textup{}$ }  & \footnotesize{Assumed}\\\\
		\footnotesize{$\delta$} & \footnotesize{Ratio between total vector and human population} &
		\footnotesize{$(1 - 5)$
			\; $\textup{}$ }  & \footnotesize{\cite{sardar2017mathematical}}\\\\
		\footnotesize{$\Pi_{V}$} & \footnotesize{Recruitment rate of the vector population} &
		\footnotesize{$\delta \times N_{H}$
			\; $\textup{}$ }  & \footnotesize{\cite{sardar2017mathematical}}\\\\      [0.2ex]
		\hline\\
	\end{tabular}
	\label{Table:Model-parameters}
\end{table}
\section{Key mathematical properties associated with the Model~\eqref{EQ:Fractional-order-vector-borne-disease-model}}\label{sec:Mathematical-properties}
To begin this section, we introduce the Lemma given as follows:

\textbf{Lemma}[Lemma~2 in~\citep{yang1996permanence}]~\label{lemma-4}:
Let us consider the open set $\Omega \subset \mathbb{R} \times \mathbb{C}^n$, and $\mathcal{F}_{i} \in C(\Omega, \mathbb{R}), i=1,2,...,n$. If $\mathcal{F}_{i}|_{x_{i}(t)=0,\mathcal{X}_{t} \in \mathbb{C}_{+0}^n}\geq 0$, $\mathcal{X}_t=\left[x_{1}(t),x_{2}(t),.....,x_{n}(t)\right]^T, i=1,2,....,n$, then the invariant domain $\mathbb{C}_{+0}^n =\lbrace \phi=(\phi_1,.....,\phi_n):\phi \in \mathbb{C}([-\tau,0],\mathbb{R}_{+0}^n)\rbrace$ defines for the the following equations
\begin{eqnarray}
	\begin{array}{llll}
		\frac{dx_i(t)}{dt}=\mathcal{F}_i(t,\mathcal{X}_t),\hspace{0.3cm} t\geq \Hat{t},\hspace{0.2cm} i=1,2,...,n.\\\nonumber
	\end{array}
\end{eqnarray} in which $\mathbb{R}_{+0}^{n}=\lbrace (x_{1},....x_{n}): x_i\geq 0,\hspace{0.2cm} i=1,2....,n \rbrace$
\begin{proposition}~\label{prop-1}
	$\mathbb{R}_{+0}^5$ is the invariant domain for the dengue model~\eqref{EQ:Fractional-order-vector-borne-disease-model}.
\end{proposition}

\begin{proof}
	Now, from the equation~\eqref{EQ:Fractional-order-vector-borne-disease-model}, we conclude that:
	\begin{eqnarray}~\label{EQ:Model-positivity}
		\begin{array}{llll}
			\displaystyle	\frac{d S_{H}}{d t}\left|_{S_{H}=0,\hspace{0.2cm} \mathcal{X} \in \mathbb{R}_{+0}^5 }\right.& =&  \displaystyle \mu_{H} N_{H} > 0,\\
			\displaystyle	\frac{d S_{V}}{d t}\left|_{S_{V}=0,\hspace{0.2cm} \mathcal{X} \in \mathbb{R}_{+0}^5}\right. &=&\displaystyle \Pi_{V}>0
		\end{array}
	\end{eqnarray}
	Our objective is to show that, $\displaystyle \frac{dI_{H}}{dt}\left|_{I_H}=0,\mathcal{X}\in \mathbb{R}_{+0}^{5}\right.\geq0$. By utilizing \eqref{EQ:Relation R-L and Caputo} and the Riemann-Liouville and Caputo fractional derivative relationship in \cite{podlubny1999fractional}, we obtain
	\begin{eqnarray}
		\begin{array}{llll}
			\displaystyle	\frac{d I_{H}}{d t}&=&  \displaystyle \frac{b^{\alpha}}{N_{H}} \beta_{V}^{H} S_{H}
			e^{-\mu_{V} t}~ _{RL}{D_{t}}^{1-\alpha}\left[I_{V}(t) e^{\mu_{V}t}\right]\\&=&  \displaystyle \frac{b^{\alpha}}{N_{H}} \beta_{V}^{H} S_{H}\int_{0}^{t}  e^{-\mu_{V}( t-s)} (t-s)^{\alpha-1} \left[ \mu_{V}I_{V}(s)+I'_{V}(s)\right]  ds\\&+&  \displaystyle  \frac{I_{V}(0) b^{\alpha}}{N_{H} \Gamma(\alpha)} e^{-\mu_{V}t} t^{\alpha-1} \beta_{V}^{H} S_{H}\\&=& \displaystyle \frac{b^{\alpha}}{N_{H}} \beta_{V}^{H} S_{H}\int_{0}^{t}   e^{-\mu_{V}( t-s)} (t-s)^{\alpha-1} \left[ \frac{b^{p}}{N_{H} C^{p}} \beta_{H}^{V} S_{V}(s) e^{-\mu_{H} s} ~_{RL}D_{s}^{1-p}\left[I_{H}(s) e^{\mu_{H} s}\right]\right] ds\\&+&   \displaystyle \frac{I_{V}(0) b^{\alpha}}{N_{H} \Gamma(\alpha)} e^{-\mu_{V}t} t^{\alpha-1} \beta_{V}^{H} S_{H}.
		\end{array}
	\end{eqnarray}
	Hence,
	\begin{eqnarray}
		\begin{array}{llll}
			\displaystyle	\frac{d I_{H}}{d t}\left|_{I_{H}=0,\hspace{0.2cm} \mathcal{X} \in \mathbb{R}_{+0}^5}\right.&=&  \displaystyle \frac{I_{V}(0) b^{\alpha}}{N_{H} \Gamma(\alpha)} e^{-\mu_{V}t} t^{\alpha-1} \beta_{V}^{H} S_{H},\\
			\Rightarrow \frac{d I_{H}}{d t}\left|_{I_{H}=0,\hspace{0.2cm} \mathcal{X} \in \mathbb{R}_{+0}^5}\right. &\geq&  0.
		\end{array}
		\label{EQ:Model-positivity1}
	\end{eqnarray}
	Similarly,
	\begin{eqnarray}
		\begin{array}{llll}
			\displaystyle	\frac{d R_{H}}{d t}\left|_{R_{H}=0,\hspace{0.2cm} \mathcal{X} \in \mathbb{R}_{+0}^5}\right. &\geq& 0,\\
			\displaystyle	\frac{d I_{V}}{d t}\left|_{I_{V}=0,\hspace{0.2cm} \mathcal{X} \in \mathbb{R}_{+0}^5}\right. &\geq & 0.\\
		\end{array}
		\label{EQ:Model-positivity2}
	\end{eqnarray}
	Thus, following Lemma~\eqref{lemma-4}, $\mathbb{R}_{+0}^5$ act as an invariant domain with respect to the system~\eqref{EQ:Fractional-order-vector-borne-disease-model}.
\end{proof}

In the real world, all the solution associated with a disease process remain bounded. Consequently, we impose further limitations on the invariant domain $\mathbb{R}_{+0}^5$ for the model~\eqref{EQ:Fractional-order-vector-borne-disease-model} to a bounded region as detailed below:

\begin{equation}
	\displaystyle \mathbb{D}=\left\lbrace (S_{H}, I_{H},R_{H},S_{V},I_{V})\in \mathbb{R}_{+0}^5: S_{H}+I_{H}+R_{H}= N_{H}, S_{V}+I_{V}\leq \frac{\Pi_{V}}{\mu_{V}}\right\rbrace.
	\label{EQ:invariant-bounded-domain}
\end{equation}

\begin{proposition}~\label{prop-2}
	The domain $\mathbb{D}$ which is closed and bounded and contained in $\mathbb{R}_{+0}^5$ act as a positively invariant and global attracting set for the system~\eqref{EQ:Fractional-order-vector-borne-disease-model}.
\end{proposition}
\begin{proof}
	From ~\eqref{EQ:Fractional-order-vector-borne-disease-model}, we have:
	\begin{eqnarray}
		\displaystyle \frac{d\bigg(S_{H}+I_{H}+R_{H}\bigg)}{dt} &=& \displaystyle 0,\\\nonumber \implies S_{H}(t)+I_{H}(t)+R_{H}(t) &=& N_{H}\;\; \forall t,\\\nonumber  \implies S_{H}(t) &\leq& N_{H},~{} I_{H}(t) \leq N_{H},~{} R_{H}(t) \leq N_{H},\;\;\forall t.  
	\end{eqnarray}
	Thus $S_{H}$, $I_{H}$, and $R_{H}$ are bounded above by some $N_{H} > 0$ for any $t>0$.
	Again, adding the last two equation from~\eqref{EQ:Fractional-order-vector-borne-disease-model}, we get
	\begin{eqnarray}~\label{EQ:invariant-bounded-domain-1}
		\begin{array}{llll}
			\displaystyle	\frac{d N_V}{dt} &= &  \displaystyle \Pi_{V} -\mu_{V} N_{V}.\\
			\displaystyle	\implies N_{V} &= &  \displaystyle \frac{\Pi_{V}}{\mu_{V}} +\left(N_{V}(0) - \frac{\Pi_{V}}{\mu_{V}}\right)e^{-\mu_{V} t}.
		\end{array}
	\end{eqnarray}
	As a result of~\eqref{EQ:invariant-bounded-domain-1}, we have~$N_{V}(0) \leq \displaystyle\frac{\Pi_{V}}{\mu_{V}}$ $\implies$ $N_{V}(t) \leq \displaystyle\frac{\Pi_{V}}{\mu_{V}}$. Accordingly, for the system~\eqref{EQ:Fractional-order-vector-borne-disease-model} the domain $\mathbb{D}$ is positively invariant. Furthermore, $N_{V}(0) > \displaystyle\frac{\Pi}{\mu_{V}}$ $\implies$ $\displaystyle \lim\limits_{t\rightarrow \infty}\sup N_{V}(t)\leq \displaystyle\frac{\Pi_{V}}{\mu_{V}}$, which necessarily conclude that the domain $\mathbb{D} \subset \mathbb{R}_{+0}^5$ serve as a global attracting set~\cite{zhao2003dynamical} with respect to the model~\eqref{EQ:Fractional-order-vector-borne-disease-model}.
\end{proof}
Since, $S_{H}+I_{H}+R_{H} =N_{H}$, and $\lim_{t\to\infty}(S_{V}+I_{V})(t) = \frac{\Pi_{V}}{\mu_{V}}$, therefore, model~\eqref{EQ:Fractional-order-vector-borne-disease-model} can be reduced as follows:
\begin{eqnarray}~\label{EQ:Fractional-order-vector-borne-disease-converted_model}
	\displaystyle	\frac{dS_{H}}{dt} &= &  \displaystyle \mu_{H} N_{H} - \frac{b^{\alpha}}{N_{H}} \beta_{V}^{H} S_{H}(t) e^{-\mu_{V} t}~ _{RL}{D_{t}}^{1-\alpha}\left[I_{V}(t) e^{\mu_{V} t} \right] -\mu_{H} S_{H},\\\nonumber 
	\displaystyle \frac{dI_{H}}{dt}&= &\displaystyle \frac{b^{\alpha}}{N_{H}} \beta_{V}^{H} S_{H}(t) e^{-\mu_{V} t}~ _{RL}{D_{t}}^{1-\alpha}\left[I_{V}(t) e^{\mu_{V} t} \right]- \frac{e^{-\mu_{H} t}}{C^{\beta}}~_{RL}D_{t}^{1-\beta}\left[I_{H}(t) e^{\mu_{H} t} \right]-\mu_{H} I_{H},\\\nonumber 
	\displaystyle \frac{dI_{V}}{dt}&=&\displaystyle \frac{b^{p}}{N_{H} C^{p}} \beta_{H}^{V} \left(\frac{\Pi_{V}}{\mu_{V}}-I_{V}\right) e^{-\mu_{H} t} ~_{RL}D_{t}^{1-p}\left[I_{H}(t) e^{\mu_{H} t}\right] -\mu_{V} I_{V},
\end{eqnarray}where $0 < \alpha \leq 1$, and $0 < \beta \leq p \leq 1$, $S_{H}(0)=S_{0}^{H}\geq 0$, $I_{H}(0)=I_{0}^{H}\geq 0$, and $I_{V}(0)=I_{0}^{V}\geq 0$. Furthermore, we assumed that all model~\eqref{EQ:Fractional-order-vector-borne-disease-converted_model} parameters are positive.
Henceforth, we can further reduce our invariant region $\mathbb{R}_{+0}^5$ for the dengue system~\eqref{EQ:Fractional-order-vector-borne-disease-converted_model} to a bounded region as follows:
\begin{eqnarray}~\label{EQ:invariant-bounded-domain-2}
	\mathbb{T}=\left\lbrace (S_{H}, I_{H},I_{V})\in \mathbb{R}_{+0}^3 :  S_{H}+I_{H}\leq N_{H}, I_{V}\leq \frac{\Pi_{V}}{\mu_{V}}\right\rbrace,\\\nonumber
	\mathbb{\mathring{T}}=\left\lbrace (S_{H}, I_{H},I_{V})\in \mathbb{R}_{+0}^3 : S_{H}+I_{H}<N_{H}, I_{V}< \frac{\Pi_{V}}{\mu_{V}}\right\rbrace,
\end{eqnarray} where, the interior of $\mathbb{T}$ is $\mathbb{\mathring{T}}$.
\begin{proposition}~\label{prop-3}
$\mathbb{T}$ a set that is both closed and bounded in the octant $ \mathbb{R}_{+0}^3$ forms a positively invariant and globally attracting set of the dynamical system denoted by the equation~\eqref{EQ:Fractional-order-vector-borne-disease-converted_model}.
\end{proposition}
\begin{proof}
	Proof is similar to~Proposition~\ref{prop-2}.
\end{proof}
Now using the Theorem~\eqref{Theorem-1} we have,
\begin{eqnarray}
	\centering
	\begin{array}{llll}
		\displaystyle	e^{-\mu_{V} t}~_{RL}{D_{t}}^{1-\alpha}\left[I_{V}(t) e^{\mu_{V} t}\right] &=&\displaystyle \mu_{V}^{1-\alpha} I_{V}(t)\frac{\Gamma(1-\alpha)}{\gamma(\mu_{V}\xi_{1},1-\alpha)}-\frac{e^{-\mu_{V}\xi_{1}} (\mu_{V}\xi_{1})^{-\alpha}\gamma(\mu_{V}t,\alpha)}{\mu_{V}^{1-\alpha}\gamma(\mu_{V}\xi_{1},1-\alpha)\Gamma(\alpha)}I_{V}(\xi_{2}),\\
		\displaystyle e^{-\mu_{H} t}~_{RL}{D_{t}}^{1-\beta}\left[I_{H}(t) e^{\mu_{H} t}\right] &=&  \displaystyle \mu_{H}^{1-\beta} I_{H}(t) \frac{\Gamma(1-\beta)}{\gamma(\mu_{H}\xi_{1},1-\beta)}-\frac{e^{-\mu_{H}\xi_{1}} (\mu_{H}\xi_{1})^{-\beta}\gamma(\mu_{H}t,\alpha)}{\mu_{H}^{1-\beta}\gamma(\mu_{H}\xi_{1},1-\beta)\Gamma(\beta)}I_{H}(\xi_{2}),\\ 
		\displaystyle e^{-\mu_{H} t}~_{RL}{D_{t}}^{1-p}\left[I_{H}(t) e^{\mu_{H} t}\right] &=& \displaystyle\mu_{H}^{1-p} I_{H}(t)\frac{\Gamma(1-p)}{\gamma(\mu_{H}\xi_{1},1-p)}-\frac{e^{-\mu_{H}\xi_{1}} (\mu_{H}\xi_{1})^{-p}\gamma(\mu_{H}t,p)}{\mu_{H}^{1-p}\gamma(\mu_{H}\xi_{1},1-p)\Gamma(p)}I_{H}(\xi_{2}),
	\end{array}
	\hspace{0.4cm}\label{EQ:Relation-fractional_derivative-2}
\end{eqnarray}
where, $0<\xi_{1},\xi_{2}<t$. 
Hence the equation~\eqref{EQ:Fractional-order-vector-borne-disease-converted_model} converted to
\begin{eqnarray}
	\small
	\begin{array}{llll}
		\displaystyle	\frac{dS_{H}}{dt} &=& \displaystyle \mu_{H} N_{H} - \frac{b^{\alpha}\beta_{V}^{H} \mu_{V}^{1-\alpha}}{N_{H}}  \frac{\Gamma(1-\alpha)}{\gamma(\mu_{V}\xi_{1},1-\alpha)}S_{H}  I_{V}\\ &+&\displaystyle\frac{b^{\alpha}\beta_{V}^{H} \mu_{V}^{1-\alpha}}{N_{H}}\frac{e^{-\mu_{V}\xi_{1}} (\mu_{V}\xi_{1})^{-\alpha}\gamma(\mu_{V}t,\alpha)}{\gamma(\mu_{V}\xi_{1},1-\alpha)\Gamma(\alpha)}S_{H}  I_{V}(\xi_{2})   -\mu_{H} S_{H},\\ 
		\displaystyle \frac{dI_{H}}{dt}&=&\displaystyle \frac{b^{\alpha}\beta_{V}^{H} \mu_{V}^{1-\alpha}}{N_{H}}  \frac{\Gamma(1-\alpha)}{\gamma(\mu_{V}\xi_{1},1-\alpha)}S_{H}  I_{V}-\frac{b^{\alpha}\beta_{V}^{H} \mu_{V}^{1-\alpha}}{N_{H}}\frac{e^{-\mu_{V}\xi_{1}} (\mu_{V}\xi_{1})^{-\alpha}\gamma(\mu_{V}t,\alpha)}{\gamma(\mu_{V}\xi_{1},1-\alpha)\Gamma(\alpha)}S_{H}  I_{V}(\xi_{2}) \\ &-&\displaystyle \frac{\Gamma(1-\beta)}{\gamma(\mu_{H}\xi_{1},1-\beta)}\frac{\mu_{H}^{1-\beta}}{C^{\beta}}~ I_{H}+\frac{e^{-\mu_{H}\xi_{1}} (\mu_{H}\xi_{1})^{-\beta}\gamma(\mu_{H}t,\alpha)}{\gamma(\mu_{H}\xi_{1},1-\beta)\Gamma(\beta)}\frac{\mu_{H}^{1-\beta}}{C^{\beta}}~ I_{H}(\xi_{2}) -\mu_{H} I_{H},\\ 
		\displaystyle \frac{dI_{V}}{dt}&= &\displaystyle \frac{b^{p}\beta_{H}^{V} \mu_{H}^{1-p}}{N_{H} C^{p}}\frac{\Gamma(1-p)}{\gamma(\mu_{H}\xi_{1},1-p)} \left(\frac{\Pi_{V}}{\mu_{V}}-I_{V}\right) I_{H}-\frac{b^{p}\beta_{H}^{V} \mu_{H}^{1-p}}{N_{H} C^{p}} \frac{e^{-\mu_{H}\xi_{1}} (\mu_{H}\xi_{1})^{-p}\gamma(\mu_{H}t,p)}{\gamma(\mu_{H}\xi_{1},1-p)\Gamma(p)}\left(\frac{\Pi_{V}}{\mu_{V}}-I_{V}\right) I_{H}(\xi_{2})\\ &- &\displaystyle\mu_{V} I_{V}
	\end{array}
	\label{EQ:Fractional-order-vector-borne-disease-converted_model-1}
\end{eqnarray}
where, $0 < \alpha \leq 1$, $0 <\beta \leq p \leq 1$, $S_{H}(0) = S^{H}_{0}$, $I_{H}(0) = I^{H}_{0}$, $R_{H}(0) = R^{H}_{0}$, $S_{V}(0) = S^{V}_{0}$, $I_{V}(0) = I^{V}_{ 0}$.

\subsection{The basic reproductive index}
The basic reproductive indexed $R_{0}$, known as basic reproduction ratio is an epidemiological constant which reflect the infection potential or transmissibility related to a vira illness. The above signifies an average number of new infections resulting from a primary caase within a group being entirely susceptible to the infection. We used the next generation matrix method~\cite{van2002reproduction} to estimate the basic reproduction number $(R_{0})$ for the fractional-order dengue model~\eqref{EQ:Fractional-order-vector-borne-disease-converted_model-1}.
Let us define the infected compartment of the equation~\eqref{EQ:Fractional-order-vector-borne-disease-converted_model-1} as follows:

$\begin{pmatrix}
	\frac{dI_{H}}{dt}\\
	\frac{dI_{V}}{dt}
\end{pmatrix}= \begin{pmatrix}
	F(t)-V(t)\\
\end{pmatrix}\begin{pmatrix}
	I_{H}\\
	I_{V}
\end{pmatrix}$,\\
where, $F(t)$ and $V(t)$ can be written as follows:
\begin{eqnarray}
	\centering
	\footnotesize
	\begin{array}{llll}
		F(t) = \left(
		\begin{array}{c c}
			0& \frac{b^{\alpha} \beta_{V}^{H}\mu_{V}^{1-\alpha}\Gamma(1-\alpha)}{ N_{H}\gamma(\mu_{V}\xi_{1},1-\alpha)}S_{H}^* -\frac{e^{-\mu_{V}\xi_{1}} (\mu_{V}\xi_{1})^{-\alpha}\gamma(\mu_{V}t,\alpha)}{\mu_{V}^{1-\alpha}\gamma(\mu_{V}\xi_{1},1-\alpha)\Gamma(\alpha)}S_{H}^* \\
			\frac{b^{p} \beta_{H}^{V}\Pi_{V}  }{N_{H} C^{p} \mu_{V}} \left[\frac{\mu_{H}^{1-p} \Gamma(1-p)}{\gamma(\mu_{H}\xi_{1},1-p)}-\frac{e^{-\mu_{H}\xi_{1}} (\mu_{H}\xi_{1})^{-p}\gamma(\mu_{H}t,p)}{\mu_{H}^{1-p}\gamma(\mu_{H}\xi_{1},1-p)\Gamma(p)}\right] &0\\
		\end{array}
		\right),
	\end{array}
	\hspace{0.6cm}\label{EQ:Basic reproduction number-1}
\end{eqnarray}

and 
\begin{eqnarray}
	\centering
	\small
	\begin{array}{llll}
		\displaystyle	V(t)&=&  \displaystyle \left(
		\begin{array}{c c}
			\mu_{H}+\frac{\mu_{H}^{1-\beta}}{C^{\beta}}\frac{\Gamma(1-\beta)}{\gamma(\mu_{H}\xi_{1},1-\beta)}-\frac{e^{-\mu_{H}\xi_{1}} (\mu_{H}\xi_{1})^{-\beta}\gamma(\mu_{H}t,\alpha)}{\mu_{H}^{1-\beta}\gamma(\mu_{H}\xi_{1},1-\beta)\Gamma(\beta)}  & 0 \\
			0 &  \mu_{V}\\
		\end{array}
		\right).\\
	\end{array}
	\label{EQ:Basic reproduction number-2}
\end{eqnarray}
Then $\lim_{t \rightarrow \infty} F(t)=F$, and $\lim_{t \rightarrow \infty} V(t)=V$.\\
where, $F$ and $V$ is defined as,
\begin{eqnarray}
	\centering
	\small
	\begin{array}{llll}
		\displaystyle	F &=&   \displaystyle \left(
		\begin{array}{c c}
			0& \frac{b^{\alpha} \beta_{V}^{H}\mu_{V}^{1-\alpha}}{ N_{H}}S_{H}^*  \\
			\frac{b^{p} \beta_{H}^{V}\Pi_{V} \mu_{H}^{1-p} }{N_{H} C^{p} \mu_{V}}&0  \\
		\end{array}
		\right),\\
	\end{array}
	\label{EQ:Basic reproduction number-3}
\end{eqnarray}
and 
\begin{eqnarray}
	\centering
	\small
	\begin{array}{llll}
		V &=& \left(
		\begin{array}{c c}
			\mu_{H}+\frac{\mu_{H}^{1-\beta}}{C^{\beta}} & 0 \\
			0 &  \mu_{V}\\
		\end{array}
		\right).\\
	\end{array}
	\label{EQ:Basic reproduction number-4}
\end{eqnarray}
Then \begin{eqnarray}
	\centering
	\small
	\begin{array}{llll}
		FV^{-1} &=& \left(
		\begin{array}{c c}
			0& \frac{b^{\alpha} \beta_{V}^{H}\mu_{V}^{1-\alpha}}{\mu_{V} }  \\
			\frac{b^{p} \beta_{H}^{V}\Pi_{V} \mu_{H}^{1-p}C^{\beta}\mu_{H}^{\beta} }{N_{H} C^{p}\mu_{V} \left(\mu_{H}^{\beta}C^{\beta}+1\right)}&0  \\
		\end{array}
		\right),\\
	\end{array}
	\label{EQ:Basic reproduction number-5}
\end{eqnarray}
so, $ FV^{-1}$ has eigen values are -$\lambda$ and $R_{0}$, where -$\lambda$ and $R_{0}$ are defined as follows:
\begin{eqnarray}
	\lambda& = -\sqrt{ \frac{\Pi_{V}~b^{\alpha+p}~ \beta_{V}^{H}~\beta_{H}^{V}~C^{\beta}}{N_{H}~C^{p}~ \mu_{V}^{1+\alpha}~\mu_{H}^{p-\beta}~ \left(\mu_{H}^{\beta}C^{\beta}+1\right)}},\\\nonumber
\end{eqnarray}
\begin{eqnarray}~\label{Basic reproduction number}
	R_{0}& = \sqrt{ \frac{\Pi_{V}~b^{\alpha+p}~ \beta_{V}^{H}~\beta_{H}^{V}~C^{\beta}}{N_{H}~C^{p}~ \mu_{V}^{1+\alpha}~\mu_{H}^{p-\beta}~ \left(\mu_{H}^{\beta}C^{\beta}+1\right)}}.\\\nonumber
\end{eqnarray}
The representation of basic reproduction number (see equation~\eqref{Basic reproduction number}) clearly indicate that $R_{0}$ depends on the three fractional order derivatives ($\alpha$, $\beta$, and $p$) appeared in the fractional order dengue system~\eqref{EQ:Fractional-order-vector-borne-disease-model}.
\section{Stability of equilibrium point}~\label{Stability of equilibrium point}
\subsection{Disease free equilibrium}
Let $\Psi_{0}$ represent the unique disease-free steady-state solution associated with the system~\eqref{EQ:Fractional-order-vector-borne-disease-converted_model} identified as follows:
\begin{eqnarray}~\label{EQ:Disease-free-equilibrium}  
	\Psi_{0}:~\left(S_{H}^{*}, I_{H}^{*}, I_{V}^{*} \right) =  \left( N_{H}, 0, 0\right).
\end{eqnarray}

\subsection{Local stability properties of $\Psi_{0}$}
If $R_{0}<1$, the disease-free equilibrium point $\Psi_{0}$ for the model~\eqref{EQ:Fractional-order-vector-borne-disease-model}, is stable unless it become unstable.
\begin{proposition}~\label{prop-4}
Asymptotic stability is observed locally for the model~\eqref{EQ:Fractional-order-vector-borne-disease-model} at the disease-free equilibrium point $\Psi_{0}$ if $R_{0} < 1$ and unstable if $R_{0}>1$.	
\end{proposition}
\begin{proof}
	Since, $\lim_{t \rightarrow \infty} \frac{\Gamma(1-\alpha)}{\gamma(\mu_{V}\xi_{1},1-\alpha)}=1$. Then for $\epsilon>0$  we have 
	\begin{eqnarray}~~\label{EQ:Epsilon_Relation}
		\frac{\Gamma(1-\alpha)}{\gamma(\mu_{V}\xi_{1},1-\alpha)}<1+\epsilon~\hspace{2cm} \forall t\geq t_{1}\\\nonumber
		e^{-\mu_{V}\xi_{1}} (\mu_{V}\xi_{1})^{-\alpha}<\epsilon~\hspace{2cm} \forall t\geq t_{2} \\\nonumber
		\frac{\gamma(\mu_{V} t,\alpha)}{\Gamma(\alpha)}<1+\epsilon~\hspace{2cm} \forall t\geq t_{3} \\\nonumber
		-\frac{\Gamma(1-\alpha)}{\gamma(\mu_{V}\xi_{1},1-\alpha)}<(\epsilon-1)~\hspace{2cm} \forall t\geq t_{4}.\\\nonumber
	\end{eqnarray}
	Then using~\eqref{EQ:Epsilon_Relation} and taking $T_{1}= max \left\lbrace t_{1},t_{2},t_{3},t_{4} \right\rbrace$ we have,
	\begin{align}~\label{EQ:Epsilon_Relation-1}
		e^{-\mu_{V} t}~_{RL}{D_{t}}^{1-\alpha}\left[I_{V}(t) e^{\mu_{V} t}\right] \leq& \mu_{V}^{1-\alpha} I_{V}(t) +\epsilon \left( \mu_{V}^{1-\alpha}I_{V}(t)-\frac{I_{V}(\xi_{2})}{\mu_{V}^{1-\alpha}\Gamma(1-\alpha)} \right) ~\hspace{2cm} \forall t\geq T_{1},\\\nonumber
	\end{align}
	similarly  we have, 
	\begin{eqnarray}~\label{EQ:Epsilon_Relation-2}
		\begin{array}{llll}
			\displaystyle	e^{-\mu_{H} t}~_{RL}{D_{t}}^{1-\beta}\left[I_{H}(t) e^{\mu_{H} t}\right] &\leq & \displaystyle \mu_{H}^{1-\beta} I_{H}(t) +\epsilon \left( \mu_{H}^{1-\beta}I_{H}(t)-\frac{I_{H}(\xi_{2})}{\mu_{H}^{1-\beta}\Gamma(1-\beta)} \right)~\hspace{2cm} \forall t\geq T_{2} ,\\  e^{-\mu_{H} t}~_{RL}{D_{t}}^{1-p}\left[I_{H}(t) e^{\mu_{H} t}\right] &\leq & \displaystyle \mu_{H}^{1-p} I_{H}(t)+\epsilon \left( \mu_{H}^{1-p}I_{H}(t)-\frac{I_{H}(\xi_{2})}{\mu_{H}^{1-p}\Gamma(1-p)} \right)~\hspace{2cm} \forall t\geq T_{3}.
		\end{array}
	\end{eqnarray}
	Then for $T= max \left\lbrace T_{1},T_{2},T_{3} \right\rbrace$ and $t\geq T$ and using the above relation in the equation~\eqref{EQ:Fractional-order-vector-borne-disease-converted_model-1} we have,
	\begin{eqnarray}
		\centering
		\begin{array}{llll}
			\displaystyle	\frac{dI_{H}}{dt}&\leq & \frac{b^{\alpha} \beta_{V}^{H}}{N_{H}} S_{H}(t)\left[  \mu_{V}^{1-\alpha}   I_{V}(t)+\epsilon \left( \mu_{V}^{1-\alpha}I_{V}(t)-\frac{I_{V}(\xi_{2})}{\mu_{V}^{1-\alpha}\Gamma(1-\alpha)}\right)\right]- \frac{\mu_{H}^{1-\beta} I_{H}(t)}{C^{\beta}}\\&+&\displaystyle \frac{\epsilon}{C^{\beta}} \left( \mu_{H}^{1-\beta}I_{H}(t)-\frac{I_{H}(\xi_{2})}{\mu_{H}^{1-\beta}\Gamma(1-\beta)} \right) -\mu_{H} I_{H},\\ 
			\displaystyle \frac{dI_{V}}{dt}&\leq& \displaystyle  \frac{b^{p}\beta_{H}^{V}}{N_{H} C^{p}} S_{V} \left[ \mu_{H}^{1-p} I_{H}(t)+\epsilon \left( \mu_{H}^{1-p}I_{H}(t)-\frac{1}{\mu_{H}^{1-p}\Gamma(1-p)} I_{H}(\xi_{2})\right) \right]  -\mu_{V} I_{V}.\\
		\end{array}
		\label{EQ:Modified_Fractional-order-model-1}
	\end{eqnarray}
	
during the analysis of local stability regarding the equilibrium without infection~$\Psi_{0}$, equation~(\eqref{EQ:Fractional-order-vector-borne-disease-converted_model-1}) is linearized around $\Psi_{0}$. Which gives,
	\begin{eqnarray}
		\begin{array}{llll}
			\displaystyle I_{H}(t)&\leq &  \displaystyle b^{\alpha} \beta_{V}^{H} \int_{0}^{t} e^{-K(t-s)}\left[  \mu_{V}^{1-\alpha}(s) I_{V}(s)+\epsilon \left( \mu_{V}^{1-\alpha}-\frac{1}{\mu_{V}^{1-\alpha}\Gamma(1-\alpha)} \right) I_{V}(s)\right]ds\\ &+& \displaystyle \frac{\epsilon}{C^{\beta}}\int_{0}^{t} e^{-K(t-s)}\left( \mu_{H}^{1-\beta}-\frac{1}{\mu_{H}^{1-\beta}\Gamma(1-\beta)} \right) I_{H}(s) ds,\\\\
			\displaystyle	I_{V}(t)&\leq& \displaystyle e^{-\mu_{V}t}I_{V}(0) +\frac{b^{p}\beta_{H}^{V}\Pi_{V}}{N_{H} C^{p}\mu_{V}}  \int_{0}^{t} e^{-\mu_{V}}(t-s)\left[ \mu_{H}^{1-p}I_{H}(s) +\epsilon \left( \mu_{H}^{1-p}-\frac{1}{\mu_{H}^{1-p}\Gamma(1-p)} \right)I_{H}(s) \right]ds .
		\end{array}
		\label{EQ:Modified_Fractional-order-model-2}
	\end{eqnarray} where,
	\begin{eqnarray}\label{EQ:Modified_Fractional-order-model-3}
		K=\mu_{H}+\frac{\mu_{H}^{1-\beta}}{C^{\beta}}.
	\end{eqnarray}
	Then using $I_{V}$ in the equation~(\ref{EQ:Modified_Fractional-order-model-2}) then we have,
	\begin{eqnarray}
		\centering
		\begin{array}{llll}
			\displaystyle I_{H }(t) & \leq & 	\displaystyle I_{H }(0)e^{-{K t}} + \epsilon b^{\alpha } \beta_{V}^{H}\left( \mu_{V}^{1-\alpha}-\frac{1}{\mu_{V}^{1-\alpha}\Gamma(1-\alpha)} \right)  I_{V}(0)\int_{0}^{t} e^{-K(t-s)} e^{-\mu_{V}s} ds \\&+& \displaystyle  \frac{b^{\alpha } \beta_{V}^{H}}{ \mu_{V}^{\alpha-1}  }I_{V}(0)\int_{0}^{t}      e^{-K(t-s)} e^{-\mu_{V}s} ds+\frac{b^{\alpha+p } \beta_{V}^{H}\beta_{H}^{V} \Pi_{V} }{N_{H} C^{p} \mu_{V}^{\alpha} \mu_{H}^{p-1} }\int_{0}^{t} e^{-K(t-s)} \int_{0}^{s} e^{-\mu_{V} ({s}-u)}  I_{H}(u)du \ d{s} \\&+&\displaystyle \epsilon \frac{b^{\alpha+p } \beta_{V}^{H}\beta_{H}^{V}  \Pi_{V} }{N_{H} C^{p} \mu_{V}^{\alpha} }\left( \mu_{H}^{1-p}-\frac{1}{\mu_{H}^{1-p}\Gamma(1-p)} \right) \int_{0}^{t} e^{-K(t-s)} \int_{0}^{s} e^{-\mu_{V} ({s}-u)}  I_{H}(u)du \ d{s}\\&+&\displaystyle \epsilon \frac{b^{\alpha+p } \beta_{V}^{H}\beta_{H}^{V}  \Pi_{V} }{N_{H} C^{p} \mu_{V}\mu_{H}^{p-1} }\left( \mu_{V}^{1-\alpha}-\frac{1}{\mu_{V}^{1-\alpha}\Gamma(1-\alpha)} \right) \int_{0}^{t} e^{-K(t-s)} \int_{0}^{s} e^{-\mu_{V} ({s}-u)}  I_{H}(u)du \ d{s}+o(\epsilon).
		\end{array}
		\label{EQ:Disease-free-1}
	\end{eqnarray}
	Following~\cite{hethcote1995sis,wu2023global} we assume that at the beginning of the disease process $I_{H}(t)\propto e^{zt}$, then we have,
	\begin{eqnarray}
		\centering
		\begin{array}{llll}
			\displaystyle 1 & \leq & \displaystyle I_{H }(0)e^{-{(K+z) t}} + \epsilon b^{\alpha } \beta_{V}^{H}\left( \mu_{V}^{1-\alpha}-\frac{1}{\mu_{V}^{1-\alpha}\Gamma(1-\alpha)} \right) e^{-{z t}} I_{V}(0)\int_{0}^{t} e^{-K(t-s)} e^{-\mu_{V}s} ds \\&+& \displaystyle  \frac{b^{\alpha } \beta_{V}^{H}}{ \mu_{V}^{\alpha-1}  }I_{V}(0)  e^{-{z t}}\int_{0}^{t}      e^{-K(t-s)} e^{-\mu_{V}s} ds+\frac{b^{\alpha+p } \beta_{V}^{H}\beta_{H}^{V}  \Pi_{V}}{N_{H} C^{p} \mu_{V}^{\alpha} \mu_{H}^{p-1} } e^{-{z t}}\int_{0}^{t} e^{-K(t-s)} \int_{0}^{s} e^{-\mu_{V} ({s}-u)}  e^{{z u}}du \ d{s} \\&+&\displaystyle \epsilon \frac{b^{\alpha+p } \beta_{V}^{H}\beta_{H}^{V} \Pi_{V}}{N_{H} C^{p} \mu_{V}^{\alpha} }\left( \mu_{H}^{1-p}-\frac{1}{\mu_{H}^{1-p}\Gamma(1-p)} \right) e^{-{z t}} \int_{0}^{t} e^{-K(t-s)} \int_{0}^{s} e^{-\mu_{V} ({s}-u)}   e^{{zu}}du \ d{s}\\&+&\displaystyle \epsilon \frac{b^{\alpha+p } \beta_{V}^{H}\beta_{H}^{V} \Pi_{V}}{N_{H} C^{p} \mu_{V}\mu_{H}^{p-1} }\left( \mu_{V}^{1-\alpha}-\frac{1}{\mu_{V}^{1-\alpha}\Gamma(1-\alpha)} \right) e^{-{z t}} \int_{0}^{t} e^{-K(t-s)} \int_{0}^{s} e^{-\mu_{V} ({s}-u)}  e^{{z u}} du \ d{s}+o(\epsilon).
		\end{array}
		\label{EQ:Disease-free-2}
	\end{eqnarray}
	Let us consider a scenario in which equation~\eqref{EQ:Disease-free-1} is supposed to have complex root $z=x+i y$, with the condition $x\geq 0$. Now, settings, ${s}-t = z_{1} ,$ $u- {s} = z_{2} $. Then we have,
	\begin{eqnarray}
		\centering
		\footnotesize
		\begin{array}{llll}
			\displaystyle 1 & \leq & 	\displaystyle I_{H }(0) e^{-{(K+x+iy) t}} + \epsilon b^{\alpha } \beta_{V}^{H}\left( \mu_{V}^{1-\alpha}-\frac{1}{\mu_{V}^{1-\alpha}\Gamma(1-\alpha)} \right) e^{-{(x+i y) t}} I_{V}(0)\int_{-t}^{0} e^{Kz_{1}} e^{\mu_{V}(t+z_{1})} dz_{1}\\&+& 
			\displaystyle  \frac{b^{\alpha } \beta_{V}^{H}}{ \mu_{V}^{\alpha-1}  }I_{V}(0)  e^{-{(x+i y) t}}\int_{-t}^{0}  e^{Kz_{1}} e^{\mu_{V}(t+z_{1})} dz_{1}\\&+&  \displaystyle \frac{b^{\alpha+p } \beta_{V}^{H}\beta_{H}^{V} \Pi_{V}}{N_{H} C^{p} \mu_{V}^{\alpha} \mu_{H}^{p-1} } e^{-{(x+i y) t}}\int_{-t}^{0} e^{Kz_{1}} \int_{-(t+z_{1})}^{0} e^{\mu_{V} z_{2}}  e^{{(x+i y) (t+z_{1}+z_{2})}}dz_{2} \ dz_{1} \\&+&\displaystyle \epsilon \frac{b^{\alpha+p } \beta_{V}^{H}\beta_{H}^{V} \Pi_{V}}{N_{H} C^{p} \mu_{V}^{\alpha} }\left( \mu_{H}^{1-p}-\frac{1}{\mu_{H}^{1-p}\Gamma(1-p)} \right) e^{-{z t}} \int_{-t}^{0} e^{Kz_{1}} \int_{-(t+z_{1})}^{0} e^{\mu_{V} z_{2}}  e^{{(x+iy) (t+z_{1}+z_{2})}}dz_{2} \ dz_{1} \\&+&\displaystyle \epsilon \frac{b^{\alpha+p } \beta_{V}^{H}\beta_{H}^{V} \Pi_{V}}{N_{H} C^{p} \mu_{V}\mu_{H}^{p-1} }\left( \mu_{V}^{1-\alpha}-\frac{1}{\mu_{V}^{1-\alpha}\Gamma(1-\alpha)} \right) e^{-{(x+i y) t}} \int_{-t}^{0} e^{Kz_{1}} \int_{-(t+z_{1})}^{0} e^{\mu_{V} z_{2}}  e^{{(x+i y) (t+z_{1}+z_{2})}}dz_{2} \ dz_{1} +o(\epsilon)
		\end{array}
		\label{EQ:Disease-free-3}
	\end{eqnarray}
Preassigned $\epsilon>0$, we assumed a sufficiently large value of $t>T$ such that $I_{H }(0)e^{-{(K+x+iy) t}}<\epsilon$, as, $e^{Kz_{1}}\leq1$ for $z_{1} \in [-t, 0]$, $I_{V}(0)e^{-{(x+iy) t}}<\epsilon$, and $e^{{(x+i y) (z_{1}+z_{2})}}<1$, since, $z_{2} \in [-(t+z_{1}), 0]$ and $z_{1}\in [-t, 0]$.
	Hence, the integral take the form as,
	\begin{eqnarray}
		\begin{array}{llll}
			\displaystyle	\int_{-t}^{0} e^{Kz_{1}} e^{\mu_{V}(t+z_{1})} dz_{1} \leq \frac{1}{\mu_{V}},\\
			\displaystyle	\int_{-t}^{0} e^{Kz_{1}}\int_{-(t+z_{1})}^{0} e^{\mu_{V} z_{2}}  e^{{(x+i y) (z_{1}+z_{2})}} dz_{2} dz_{1}\leq \frac{1}{K\mu_{V}}.
		\end{array}
		\label{EQ:Disease-free-4}
	\end{eqnarray}
	Then using~\eqref{EQ:Disease-free-3} in the equation~\eqref{EQ:Disease-free-4} the we have,
	\begin{eqnarray}
		\centering
		\begin{array}{llll}
			\displaystyle 1 & < & 
			\displaystyle  \epsilon + \frac{b^{\alpha } \beta_{V}^{H}}{ \mu_{V}^{\alpha}  }I_{V}(0) \epsilon +\frac{b^{\alpha+p } \beta_{V}^{H}\beta_{H}^{V}  \Pi_{V}}{N_{H} C^{p} \mu_{V}^{\alpha+1} \mu_{H}^{p-1}K } + \epsilon \frac{b^{\alpha+p } \beta_{V}^{H}\beta_{H}^{V} \Pi_{V}}{N_{H} C^{p} \mu_{V}^{\alpha+1} K}\left( \mu_{H}^{1-p}-\frac{1}{\mu_{H}^{1-p}\Gamma(1-p)} \right) \\&+&\displaystyle \epsilon \frac{b^{\alpha+p } \beta_{V}^{H}\beta_{H}^{V} \Pi_{V}}{N_{H} C^{p} \mu_{V}^{2}\mu_{H}^{p-1}K }\left( \mu_{V}^{1-\alpha}-\frac{1}{\mu_{V}^{1-\alpha}\Gamma(1-\alpha)} \right) +o(\epsilon)\\&<& \displaystyle \left[   1+\frac{b^{\alpha } \beta_{V}^{H}}{ \mu_{V}^{\alpha-1}  }I_{V}(0)+2\frac{b^{\alpha+p } \beta_{V}^{H}\beta_{H}^{V} \Pi_{V}}{N_{H} C^{p} \mu_{V}^{\alpha+1}\mu_{H}^{p-1} K}\right]  \epsilon+\frac{b^{\alpha+p } \beta_{V}^{H}\beta_{H}^{V} \Pi_{V}}{N_{H} C^{p} \mu_{V}^{\alpha+1} \mu_{H}^{p-1}K}\\&<& \displaystyle  \left[   1+\frac{b^{\alpha } \beta_{V}^{H}}{ \mu_{V}^{\alpha-1}  }I_{V}(0)+2R_{0}^{2}\right]  \epsilon+ R_{0}^{2}<1,
		\end{array}
		\label{EQ:Disease-free-5}
	\end{eqnarray} this produce a contradiction. This implies that each of the roots $z=x+i y$ of the equation~\eqref{EQ:Fractional-order-vector-borne-disease-model} have negative real parts i.e, $x<0$. Hence, the proof verifies local asymptotic stability of $\Psi_{0}$.
\end{proof}

\subsection{Global stability analysis of disease-free equilibrium point}
NFinally, we have demonstrated the model's~\eqref{EQ:Modified_Fractional-order-vector-borne-disease-model} disease-free equilibrium point is globally stable.

\begin{proposition}~\label{disease_free_global_prop-4}
If $R_{0} < 1$  $\Psi_{0}$ the equilibrium state of the dengue model~\eqref{EQ:Fractional-order-vector-borne-disease-model} exhibit globally asymptotically stable. Otherwise it is unstable.
\end{proposition}
Let use the equation~\eqref{EQ:Disease-free-1} we have,
\begin{eqnarray}
	\centering
	\begin{array}{llll}
		\displaystyle I_{H }(t) & \leq & 
		\displaystyle I_{H }(0)e^{-{K t}} + \epsilon b^{\alpha } \beta_{V}^{H}\left( \mu_{V}^{1-\alpha}-\frac{1}{\mu_{V}^{1-\alpha}\Gamma(1-\alpha)} \right)  I_{V}(0)\int_{0}^{t} e^{-K(t-t_{1})} e^{-\mu_{V}t_{1}} dt_{1} \\&+& \displaystyle  \frac{b^{\alpha } \beta_{V}^{H}}{ \mu_{V}^{\alpha-1}  }I_{V}(0)\int_{0}^{t}   e^{-K(t-t_{1})} e^{-\mu_{V}t_{1}} dt_{1}+\frac{b^{\alpha+p } \beta_{V}^{H}\beta_{H}^{V} \Pi_{V}}{N_{H} C^{p} \mu_{V}^{\alpha} \mu_{H}^{p-1} }\int_{0}^{t} e^{-K(t-t_{1})} \int_{0}^{t_{1}} e^{-\mu_{V} ({t_{1}}-s)}  I_{H}(s)ds \ d{t_{1}} \\&+&\displaystyle \epsilon \frac{b^{\alpha+p } \beta_{V}^{H}\beta_{H}^{V} \Pi_{V} }{N_{H} C^{p} \mu_{V}^{\alpha} }\left( \mu_{H}^{1-p}-\frac{1}{\mu_{H}^{1-p}\Gamma(1-p)} \right) \int_{0}^{t} e^{-K(t-t_{1})} \int_{0}^{t_{1}} e^{-\mu_{V} ({t_{1}}-s)}  I_{H}(s)ds \ d{t_{1}}\\&+& \displaystyle \epsilon \frac{b^{\alpha+p } \beta_{V}^{H}\beta_{H}^{V} \Pi_{V}}{N_{H} C^{p} \mu_{V}\mu_{H}^{p-1} }\left( \mu_{V}^{1-\alpha}-\frac{1}{\mu_{V}^{1-\alpha}\Gamma(1-\alpha)} \right) \int_{0}^{t} e^{-K(t-t_{1})} \int_{0}^{t_{1}} e^{-\mu_{V} (t_{1}-s)}  I_{H}(s)ds \ d{t_{1}}.
	\end{array}
	\label{EQ:eq1}
\end{eqnarray}

Let us define,
\begin{eqnarray}
	\small
	\centering
	\begin{array}{llll}
		\displaystyle B_{1} = \int_{0}^{t} e^{-K(t-t_{1})} e^{-\mu_{V}t_{1}} d{t_{1}}.
	\end{array}
	\label{EQ:eq2}
\end{eqnarray} and
\begin{eqnarray}
	\small
	\centering
	\begin{array}{llll}
		\displaystyle B_{2}  =  \int_{0}^{t} e^{-K(t-t_{1})} \int_{0}^{t_{1}} e^{-\mu_{V} ({t_{1}}-s)}  I_{H}(s)ds \ d{t_{1}}.
	\end{array}
	\label{EQ:eq3}
\end{eqnarray} 
Now let us assume that    $\lim_{t\to\infty} sup  I_{H}(t) = i_{0}$.

Now from equation~(\ref{EQ:eq2}) we have
\begin{eqnarray}
	\small
	\centering
	\begin{array}{llll}
		\displaystyle B_{1}  &\leq &  \displaystyle \epsilon \int_{0}^{t} e^{-K(t-t_{1})}  d{t_{1}} \\\\ &\leq & \displaystyle\frac{ \epsilon}{K}.
	\end{array}
\end{eqnarray}
According to the definition of $lim sup$, preassigned $\epsilon > 0$ $\exists u $ such that $ I_{H}(t) < i_{0}  + \epsilon$, $ \forall t >u$.
Let us proceed with the transformation,
$ t - {t_{1}} = z_{1} $ $ {t_{1}}-s = z_{2} $ we have
$\exists t_{2}$ such that $\lim_{t\to\infty}$ $ \int_{t_{2}}^{t}  e^{-K z_{1}} \int_{t^*-z_{1}}^{t-z_{1}} e^{-\mu_{V} z_{2}} \ dz_{2} \ dz_{1} < $ $\epsilon$,$\forall t> t_{2}$ and $\forall$$\epsilon$ $> 0 $.
Suppose we take any $ t^{*} > t_{2}$ and now let us evaluate the integral we have, 

\begin{eqnarray}
	\small
	\centering
	\begin{array}{llll}
		\displaystyle B_{2} & \leq &  \displaystyle \int_{0}^{t^{*}} e^{-K z_{1}} \int_{0}^{t^*-z_{1}}e^{-\mu_{V} z_{2}} \hspace{0.1cm} I_{H}(t-z_{1}-z_{2}) \ dz_{2} \ dz_{1}, \\\\& + &  \displaystyle  \int_{t_{*}}^{t} e^{-K z_{1}} \int_{t^*-z_{1}}^{t-z_{1}} e^{-\mu_{V} z_{2}}  I_{H}(t-z_{1}-z_{2}) \ dz_{2} \ dz_{1}.
	\end{array}
	\label{EQ:eq4}
\end{eqnarray} As $t$ becomes sufficiently, we obtain
$ t-z_{1}-z_{2} \geq t-t^{*} > u $. 
Thus, $ I_{H}(t-z_{1}-z_{2}) < i_{0}+\epsilon$.
Therefore from equation~\eqref{EQ:eq4} we have 
\begin{eqnarray}
	\small
	\centering
	\begin{array}{llll}
		\displaystyle  B_{2} & \leq &  \displaystyle \int_{0}^{t^{*}} e^{-K z_{1}} \int_{0}^{t^*-z_{1}}e^{-\mu_{V} z_{2}} \hspace{0.1cm} \ (i_{0}+\epsilon) \ dz_{2} \ dz_{1} \\& + &  \displaystyle N_{H} \int_{t_{*}}^{t} e^{-K z_{1}} \int_{t^*-z_{1}}^{t-z_{1}} e^{-\mu_{V} z_{2}} \ dz_{2} \ dz_{1},~~\text{since}~{} I_{H}\leq N_{H} \\& \leq & \displaystyle \frac{(i_0 +\epsilon)}{\mu_{V} K} + N_{H} \epsilon ,\\ & = & \displaystyle  \frac{i_0 }{\mu_{V} K} + \epsilon \left( N_{H}+ \frac{1}{\mu_{V} K}\right).
	\end{array}
	\label{EQ:B1}
\end{eqnarray}\\

Then form equation~\eqref{EQ:eq1} we have 
\begin{eqnarray}
	\small
	\centering
	\begin{array}{llll}
		\displaystyle I_{H}(t) &\leq & \displaystyle I_{H}(0)e^{-{K t}} + \frac{b^{\alpha } \beta_{V}^{H}}{K\mu_{V}^{\alpha-1}  }I_{V}(0) \epsilon+\frac{b^{\alpha+p } \beta_{V}^{H}\beta_{H}^{V} \Pi_{V}}{N_{H} C^{p} \mu_{V}^{\alpha+1} \mu_{H}^{p-1}K }i_0 + \epsilon \frac{b^{\alpha+p } \beta_{V}^{H}\beta_{H}^{V} \Pi_{V}}{N_{H} C^{p} \mu_{V}^{\alpha} \mu_{H}^{p-1} }\left( N_{H}+ \frac{1}{\mu_{V} K}\right)\\&+&\displaystyle \epsilon \frac{b^{\alpha+p } \beta_{V}^{H}\beta_{H}^{V} \Pi_{V}}{N_{H} C^{p} \mu_{V}^{\alpha} }\left( \mu_{H}^{1-p}-\frac{1}{\mu_{H}^{1-p}\Gamma(1-p)} \right) \frac{i_0 }{\mu_{V} K} \\&+&\displaystyle \epsilon \frac{b^{\alpha+p } \beta_{V}^{H}\beta_{H}^{V} \Pi_{V}}{N_{H} C^{p} \mu_{V}\mu_{H}^{p-1} }\left( \mu_{V}^{1-\alpha}-\frac{1}{\mu_{V}^{1-\alpha}\Gamma(1-\alpha)} \right)\frac{i_0 }{\mu_{V} K} +o(\epsilon),\\ & \leq &  \displaystyle \left[I_{H}(0)+ \frac{b^{\alpha } \beta_{V}^{H}}{K\mu_{V}^{\alpha-1}  }I_{V}(0)+\frac{b^{\alpha+p } \beta_{V}^{H}\beta_{H}^{V}\Pi_{V}}{ C^{p} \mu_{V}^{\alpha} \mu_{H}^{p-1} }+R_{0}^{2}(i_{0}+1)\right] \epsilon +R_{0}^{2}i_{0}.
	\end{array}
\end{eqnarray} If $ R_{0} < 1 $ then we select the value of $\epsilon$ so that $I_{H}(t)$ $<$ $ i_{0} $ , $\forall t>0 $ .
Therefore $\lim_{t \rightarrow \infty}$ $sup I_{H}(t)<i_{0} $,
but   $\lim_{t \rightarrow \infty}$ $supI_{H}(t) =i_{0}$. Hence, $\lim_{t \rightarrow \infty}$$I_{H}(t)=0$.
Similarly, $S_{H}(t) =0$ and
$ I_{V}(t)=0$ as $t \rightarrow \infty,$ since $I_{H}(t) = 0 $ as $t \rightarrow \infty $. $ \Box$

\subsection{Endemic Equilibrium point and it's stability}
Let $\Psi_{*}=\left(S_{H}^{*}, I_{H}^{*}, I_{V}^{*}\right)$ denotes the equilibrium point regarding the modified fractional-order dengue model~\eqref{EQ:Fractional-order-vector-borne-disease-converted_model}. Now, with the aid of Theorem-\eqref{Theorem-1} the fractional-order terms in~\eqref{EQ:Fractional-order-vector-borne-disease-converted_model} are converted to,
\begin{eqnarray}
	\centering
	\begin{array}{llll}
		\displaystyle	\lim_{t\to\infty}\left\lbrace e^{-\mu_{V} t}~ _{RL}{D_{t}}^{1-\alpha}\left[I_{V}(t) e^{\mu_{V} t} \right]\right\rbrace = \mu_{V}^{1-\alpha} I_{V}^{*},\\ 
		\displaystyle	\lim_{t\to\infty}\left\lbrace e^{-\mu_{H} t} ~_{RL}D_{t}^{1-\beta}\left[I_{H}(t) e^{\mu_{H} t}\right]\right\rbrace =\mu_{H}^{1-\beta} I_{H}^{*},\\
		\displaystyle	\lim_{t\to\infty}\left\lbrace e^{-\mu_{H} t} ~_{RL}D_{t}^{1-p}\left[I_{H}(t) e^{\mu_{H} t}\right]\right\rbrace= \mu_{H}^{1-p} I_{H}^{*}.
	\end{array}
	\label{EQ:Endemic_point-1}
\end{eqnarray}
Thus the system~\eqref{EQ:Fractional-order-vector-borne-disease-converted_model}  becomes,
\begin{eqnarray}
	\centering
	\begin{array}{llll}
		\displaystyle	\mu_{H} N_{H} - \frac{b^{\alpha}}{N_{H}} \beta_{V}^{H} S_{H}^{*}(t) \mu_{V}^{1-\alpha} I_{V}^{*} -\mu_{H} S_{H}^{*}=0,\\ 
		\displaystyle \frac{b^{\alpha}}{N_{H}} \beta_{V}^{H} S_{H}^{*}(t) \mu_{V}^{1-\alpha} I_{V}^{*} -  \frac{1}{C^{\beta}}\mu_{H}^{1-\beta} I_{H}^{*}  -\mu_{H} I_{H}^{*}=0,\\ 
		\displaystyle  \frac{b^{p}}{N_{H} C^{p}} \beta_{H}^{V} \left( \frac{\Pi_{V}}{\mu_{V}}-I_{V}\right) \mu_{H}^{1-p} I_{H}^{*}  -\mu_{V} I_{V}^{*}=0.  
	\end{array}
	\label{EQ:Endemic_point-2}
\end{eqnarray}
Now from~\eqref{EQ:Endemic_point-2} we have,
\begin{eqnarray}
	\centering
	\begin{array}{llll}
		\displaystyle	S_{H}^{*} &= &  \displaystyle\frac{\mu_{H} N_{H}}{\frac{b^{\alpha}\mu_{V}^{1-\alpha}\beta_{V}^{H}}{N_{H}}I_{V}^{*}+\mu_{H}} \\
		\displaystyle	I_{H}^{*} &= &  \displaystyle \frac{ b^{\alpha}\beta_{V}^{H}\mu_{V}^{1-\alpha}\mu_{H}^{\beta}C^{\beta}} {\mu_{H}(\mu_{H}^{\beta}C^{\beta}+1)N_{H}}S_{H}^{*} I_{V}^{*}\\ &=& \displaystyle \frac{ b^{\alpha}\beta_{V}^{H} \mu_{V}^{1-\alpha}\mu_{H}^{\beta}C^{\beta}} {\left( \frac{b^{\alpha}\mu_{V}^{1-\alpha}\beta_{V}^{H}}{N_{H}}I_{V}^{*}+\mu_{H}\right)\left(\mu_{H}^{\beta}C^{\beta}+1\right) }I_{V}^{*}\\
		\displaystyle	I_{V}^{*} &=& \displaystyle  \frac{ b^{\alpha+p}\beta_{V}^{H} \beta_{H}^{V}\mu_{V}^{1-\alpha} \mu_{H}^{\beta}C^{\beta}} {N_{H}\mu_{V} C^{p} \left(  \frac{b^{\alpha}\mu_{V}^{1-\alpha}\beta_{V}^{H}}{N_{H}}I_{V}^{*}+\mu_{H}\right) \left( \mu_{H}^{\beta}C^{\beta}+1\right)\mu_{H}^{p-1}}(\frac{\Pi_{V}}{\mu_{V}}-I_{V}^{*})I_{V}^{*}.
	\end{array}
	\label{EQ:Endemic_point-3}
\end{eqnarray}
Now, one solution is $I_{V}^{*}= 0$, which provides the diseases free equilibrium $(N_{H},0,0)$. In case $I_{V}^{*}\neq 0$, then we have,
\begin{eqnarray}\label{EQ:Endemic_point-4}
	I_{V}^{*} &= & \displaystyle\frac{\mu_{H}}{A} \left( R_{0}^{2}-1\right),
\end{eqnarray}
where, 
\begin{eqnarray}
	\centering
	\begin{array}{llll}
		\displaystyle	A &=&  \displaystyle \frac{\mu_{V}^{1-\alpha} b^{\alpha} \beta_{V}^{H}}{N_{H}}+\frac{b^{\alpha+p} \beta_{V}^{H} \beta_{H}^{V} \mu_{H}^{\beta}C^{\beta}}{ N_{H}C^{p} \mu_{V}^{\alpha} \mu_{H}^{p-1} \left(\mu_{H}^{\beta}C^{\beta}+1\right)}.\\
	\end{array}
\end{eqnarray} and $R_{0}$ defined in~\eqref{Basic reproduction number}.
Thus, $I_{V}^{*}> 0$, $S_{H}^{*}>  0$, $I_{H}^{*}>  0$ holds if and only if $R_{0}>1$. In the light of this, the system admits a unique positive endemic equilibrium point if and only if the threshold $R_{0}>1$, also known as the basic reproduction number.
\begin{proposition}~\label{prop-5}
As long as $R_{0}> 1$, the endemic equilibrium point $\Psi_{*}$ of model~\eqref{EQ:Fractional-order-vector-borne-disease-converted_model} remains stable in an asymptotic sense.	
\end{proposition}
\begin{eqnarray}
	\centering
	\begin{array}{llll}
		\displaystyle \frac{dS_{H}}{dt} &=& \displaystyle  \mu_{H} N_{H} - \frac{b^{\alpha}}{N_{H}} \beta_{V}^{H} S_{H}(t) X(t) -\mu_{H} S_{H},\\
		\displaystyle	\frac{dI_{H}}{dt}&= & \displaystyle \frac{b^{\alpha}}{N_{H}} \beta_{V}^{H} S_{H}(t) X(t)- \frac{1}{C^{\beta}}Y(t) -\mu_{H} I_{H},\\  
		\displaystyle \frac{dI_{V}}{dt} &= & \displaystyle \frac{b^{p}}{N_{H} C^{p}} \beta_{H}^{V} \left(\frac{\Pi_{V}}{\mu_{V}}-I_{V}\right) Z(t) -\mu_{V} I_{V}.
	\end{array}
	\label{EQ:Fractional-order_borne-Converted-model}
\end{eqnarray} Where,
\begin{eqnarray}
		\centering
	\begin{array}{llll}
\mathbb{X}(t)=e^{-\mu_{V} t}~_{RL}{D_{t}}^{1-\alpha}\left[I_{V}(t) e^{\mu_{V} t}\right],\\
\mathbb{Y}(t) =e^{-\mu_{H} t}~_{RL}{D_{t}}^{1-\beta}\left[I_{H}(t) e^{\mu_{H} t}\right],\\
\mathbb{Z}(t)= e^{-\mu_{H} t}~_{RL}{D_{t}}^{1-p}\left[I_{H}(t) e^{\mu_{H} t}\right].
\end{array}
\label{EQ:Relation-fractional_derivative}
\end{eqnarray}
As a result, we infer that,
\begin{eqnarray}
	\centering
	\begin{array}{llll}
		\displaystyle	S_{H}(t)&=& \displaystyle S_{H}(0)e^{- \mu_{H} t}+ N_{H}(1-e^{-\mu_{H} t}) - \frac{b^{\alpha} \beta_{V}^{H} }{N_{H} }\int_{0}^{t} e^{-\mu_{H}(t-s)} S_{H}(s) \mathbb{X}(s) ds,\\
		\displaystyle	I_{H}(t) &=&  \displaystyle I_{H}(0)e^{-K t} +\frac{b^{\alpha} \beta_{V}^{H}}{N_{H}} \int_{0}^{t}  e^{-K (t-s)}S_{H}(s) \mathbb{X}(s) ds -\frac{\mu_{H}^{1-\beta}}{C^{\beta} } \int_{0}^{t}e^{-K (t-s)}\mathbb{Y}(s) ds \\&+& \displaystyle \int_{0}^{t}e^{-K (t-s)}I_{H}(s) ds,\\
		\displaystyle	I_{V}(t)&=& \displaystyle  I_{V}(0)e^{- \mu_{V} t}+\frac{b^{p} \beta_{H}^{V} }{N_{H}C^{p} }\int_{0}^{t} e^{-\mu_{V}(t-s)} \left(\frac{\Pi_{V}}{\mu_{V}}-I_{V}(s)\right) \mathbb{Z}(s) ds,\\
	\end{array}
	\label{EQ:Volterra-integral_1}
\end{eqnarray} where, $K=\mu_{H}+\frac{\mu_{H}^{1-\beta}}{C^{\beta} }$.

Now, we use following perturbation, $S_{H}(t)=S_1(t) + S_{H}^*$, $I_{H}(t)=I_1(t) + I_{H}^*$, $I_{V}(t)=I_2(t) + I_{V}^*$, $\mathbb{X}(t)=X_{1}(t)+X^* $,  $\mathbb{Y}(t)=Y_{1}(t)+Y^* $,  $\mathbb{Z}(t)=Z_{1}(t)+Z^* $, where,  
$\displaystyle X^{*} = \displaystyle \mu_{V}^{1-\alpha} I_{V}^{*}$. $Y^{*}=\mu_{H}^{1-\beta} I_{H}^{*}$
and  $Z^{*}=\mu_{H}^{1-p} I_{H}^{*}$
then the system~(\ref{EQ:Volterra-integral_1}) can be written as:
\begin{eqnarray}
	\centering
	\begin{array}{llll}
		\displaystyle	S_{1}(t)&=& \displaystyle  S_{H}(0)e^{- \mu_{H} t}+ N_{H}(1-e^{-\mu_{H} t}) - \frac{b^{\alpha} \beta_{V}^{H} }{N_{H} }\int_{0}^{t} e^{-\mu_{H}(t-s)}\left(S_{1}(s) + S_{H}^*\right) \left(X_{1}(s)+X^*\right) ds-S_{H}^{*},\\
		\displaystyle	I_{1}(t)&=& \displaystyle I_{H}(0)e^{-K t} +\frac{b^{\alpha} \beta_{V}^{H}}{N_{H}} \int_{0}^{t}  e^{-K (t-s)}\left(S_{1}(s) + S_{H}^*\right) \left(X_{1}(s)+X^*\right) ds \\ &-& \displaystyle \frac{\mu_{H}^{1-\beta}}{C^{\beta} } \int_{0}^{t}e^{-K (t-s)} \left(Y_{1}(s)+Y^*\right) ds+ \int_{0}^{t}e^{-K (t-s)}\left(I_{1}(s) + I_{H}^*\right)ds-I_{H}^{*},\\
		
		\displaystyle	I_{2}(t)&=& \displaystyle  I_{V}(0)e^{- \mu_{V} t}+\frac{b^{p} \beta_{H}^{V} }{N_{H}C^{p} }\int_{0}^{t} e^{-\mu_{V}(t-s)} \left(\frac{\Pi_{V}}{\mu_{V}}-\left(I_{2}(s) + I_{V}^*\right)\right) \left(Z_{1}(s)+Z^*\right) ds-I_{V}^{*},\\
		
		\displaystyle	X_{1}(t)&=& \displaystyle -X^{*}+\frac{\Gamma(1-\alpha)}{\gamma(\mu_{V}\xi_{1},1-\alpha)}\mu_{V}^{1-\alpha} (I_2(t) + I_{V}^*)- \frac{e^{-\mu_{V}\xi_{1}}(\mu_{V}\xi_{1})^{-\alpha}}{\mu_{V}^{1-\alpha}} \frac{\gamma(\mu_{V}t,\alpha)}{\gamma(\mu_{V}\xi_{1},\alpha) \Gamma(\alpha)} (I_2(t) + I_{V}^*),\\
		
		\displaystyle	Y_{1}(t)&=& \displaystyle -Y^{*}+\frac{\Gamma(1-\beta)}{\gamma(\mu_{H}\xi_{1},1-\beta)}\mu_{H}^{1-\beta} (I_1(t) + I_{H}^*)- \frac{e^{-\mu_{H}\xi_{1}}(\mu_{H}\xi_{1})^{-\beta}}{\mu_{H}^{1-\beta}} \frac{\gamma(\mu_{H}t,\beta)}{\gamma(\mu_{H}\xi_{1},\beta) \Gamma(\beta)} (I_1(t) + I_{H}^*),\\
		
		\displaystyle	Z_{1}(t)&=&\displaystyle -Z^{*}+\frac{\Gamma(1-p)}{\gamma(\mu_{H}\xi_{1},1-p)}\mu_{H}^{1-p} (I_1(t) + I_{H}^*)- \frac{e^{-\mu_{H}\xi_{1}}(\mu_{H}\xi_{1})^{-p}}{\mu_{H}^{1-p}} \frac{\gamma(\mu_{H}t,p)}{\gamma(\mu_{H}\xi_{1},p) \Gamma(p)} (I_1(t) + I_{H}^*).\\
	\end{array}
	\label{EQ:Volterra-integral_2}
\end{eqnarray}
Since, $\lim_{t \rightarrow \infty} \frac{\Gamma(1-\alpha)}{\gamma(\mu_{V}\xi_{1},1-\alpha)}=1$, and $\lim_{t \rightarrow \infty}$ $e^{-\mu_{V}\xi_{1}} (\mu_{V}\xi_{1})^{-\alpha}=0$. Hence, $\lim_{t\to\infty}(X_{1}(t)+X^*)=\mu_{V}^{1-\alpha}(I_2(t) + I_{V}^*)+o(\epsilon)$, similarly, $\lim_{t\to\infty}(Y_{1}(t)+Y^*)=\mu_{H}^{1-\beta}(I_1(t) + I_{H}^*)+o(\epsilon)$, and $\lim_{t\to\infty}(Z_{1}(t)+Z^*)=\mu_{H}^{1-p}(I_1(t) + I_{H}^*)+o(\epsilon)$ for $\epsilon>0$,
thus, the system~\eqref{EQ:Volterra-integral_2} converted to:
\begin{eqnarray}
	\begin{array}{llll}
		\displaystyle \mathbb{Q}(t) &=& \displaystyle \mathbb{F}(t) + \int_0^t A(t-\tau) \mathbb{G}(\mathbb{Q}(s))ds
	\end{array}
	\label{EQ:Volterra-1}
\end{eqnarray}where,

\begin{eqnarray}
	% \nonumber to remove numbering (before each equation)
	\displaystyle \mathbb{Q}(t)&=& \displaystyle \left(
	\begin{array}{c}
		S_1(t) \\
		I_1(t) \\
		I_2(t) \\
	\end{array}
	\right),
\end{eqnarray}
and, 
\begin{align}~\label{EQ:Volterra-2}
	\footnotesize
	\mathbb{F}(t)=\left(
	\begin{array}{c}
		S_{H}(0)e^{- \mu_{H} t}+ N_{H}(1-e^{-\mu_{H} t}) - \frac{b^{\alpha} \beta_{V}^{H}\mu_{V}^{1-\alpha} S_{H}^*
			I_{V}^* }{N_{H} \mu_{H}}  (1- e^{-\mu_{H} t})-\frac{b^{\alpha} \beta_{V}^{H} S_{H}^* }{N_{H}\mu_{H}} (1- e^{-\mu_{H} t})
		o(\epsilon)-S_{H}^*\\\\
		I_{H}(0)e^{-K t} +\frac{b^{\alpha} \beta_{V}^{H}\mu_{V}^{1-\alpha} S_{H}^*
			I_{V}^* }{K}  (1- e^{-K t})+\frac{b^{\alpha} \beta_{V}^{H}}{N_{H} K} (1- e^{-K t})	o(\epsilon) -\frac{o(\epsilon)}{C^{\beta} K}(1- e^{-K t})-I_{H}^* \\\\
		I_{V}(0)e^{-\mu_{H} t} -\frac{b^{p} \beta_{H}^{V}\mu_{H}^{1-p} I_{H}^*	I_{V}^* }{N_{H}C^{p} \mu_{V}}  (1- e^{-\mu_{V} t})+\frac{b^{p} \beta_{H}^{V}\mu_{H}^{1-p}\Pi_{V} I_{H}^* }{N_{H} C^{p} \mu_{V}^{2}} (1- e^{-\mu_{V} t})+\frac{b^{p} \beta_{H}^{V} }{N_{H} C^{p} \mu_{V}} o(\epsilon)(1- e^{-\mu_{V}t})-I_{H}^*
	\end{array}\right)\\\nonumber
\end{align}

and,

\begin{eqnarray}
	\centering
	\small
	\begin{array}{llll}
		\mathbb{G}(t)& =& \displaystyle  \left(
		\begin{array}{c}
			I_2(t)S_1(t)+I_{V}^*S_1(t)+S_{H}^*I_{2}\\ 
			I_{1} \\
			I_2(t)I_1(t)+I_{V}^*I_1(t)+I_{H}^*I_{2}(t) \\ 
		\end{array}
		\right),
	\end{array}
	\label{EQ:Volterra-3}
\end{eqnarray}

The matrix $A(t) = \left[A_{ij}(t) \right]$ defined in the equation~\eqref{EQ:Volterra-1} has following non-zero components:

\begin{eqnarray}
	\centering
	\begin{array}{llll}
		\displaystyle  A_{11}(t) &= & \displaystyle -\frac{b^{\alpha} \beta_{V}^{H} }{N_{H}}  \mu_{V}^{1-\alpha} \exp(-\mu_{H} t ),\\\\
		
		\displaystyle A_{21}(t)&=& \displaystyle \frac{b^{\alpha} \beta_{V}^{H} }{N_{H}}  \mu_{V}^{1-\alpha} \exp(-\mu_{H} t ),\\\\
		\displaystyle A_{22}(t)&=& \displaystyle \frac{ \mu_{H}^{1-\beta} }{C^{\beta}}  \exp(-K t ),\\\\
		\displaystyle  A_{32}(t)& =& \displaystyle \frac{b^{p}\beta_{H}^{V} \Pi_{V}}{N_{H} C^{p} \mu_{V}}  \mu_{H}^{1-p}\exp(-\mu_{V} t ) , \\\\
		
		\displaystyle A_{33}(t)& =& \displaystyle \frac{b^{p}\beta_{H}^{V} }{N_{H} C^{p} }  \mu_{H}^{1-p}\exp(-\mu_{V} t ).
	\end{array}
	\label{EQ:Volterra-4}
\end{eqnarray}

The linear approximation of the transformed system~(\ref{EQ:Volterra-1}) includes the characteristic equation below

\begin{equation}
	det\left( I - \int_{0}^{\infty} e^{-\lambda t} A(t) J dt\right)  = 0,\\
	\label{EQ:characteristic-equation}
\end{equation}wherein, $I$, $det$, $J$ is denoted as identity matrix, determinant and the Jacobian matrix respectively, of $G$ computed at equilibrium point which is written as follows:
\begin{eqnarray}
	\centering
	\small
	\begin{array}{llll}
		J=DG(0) &=& \left(
		\begin{array}{c c c}
			I_{V}^{*} & 0& S_{H}^{*}\\
			0&1&0\\
			0& I_{V}^{*}& I_{H}^{*} \\ 
		\end{array}
		\right),
	\end{array}
	\label{EQ:Volterra-5}
\end{eqnarray} 
To determine asymptotic stability of the EEP $\Psi_{*}$ locally we follow the lemma given below

\begin{lemma}(\cite{hethcote1980integral}; \cite{miller1967linearization})~\label{lemma-3}:
If the system~\eqref{EQ:Volterra-1} has solution that exist on the interval $[0, \infty)$ and remain bounded with following condition
\begin{enumerate}
\item [$\bullet$] $\mathbb{F}(t) \in C[0, \infty)$  and $\mathbb{F}(t) \rightarrow 0$ as $t \rightarrow \infty$.
\item [$\bullet$] $A(t) \in L^{1}[0, \infty)$,
\item [$\bullet$]$\mathbb{G}(L) \in C^{1}(R^{2})$ such that $\mathbb{G}(0) = 0$
\item [$\bullet$] $det(J)\neq 0$ and (\ref{EQ:characteristic-equation}) have characteristic roots with negative real parts
\end{enumerate}
 This indicates, origin exhibits local asymptotic stablility for the disease model~\eqref{EQ:Volterra-1}. As a result, popint $E_{*}$ of the system~\eqref{EQ:Fractional-order-vector-borne-disease-converted_model} is locally asymptotic stable.
\end{lemma}
\begin{proof}
	First, we prove that $\lim_{t \rightarrow \infty}\mathbb{F}(t) = 0$. Using the endemic equilibrium point of the system~\eqref{EQ:Fractional-order-vector-borne-disease-model} and the expression of $\mathbb{F}(t)$ in~\eqref{EQ:Volterra-2} leads to the following form:
	
	\begin{eqnarray}
		\centering
		\begin{array}{llll}
			\displaystyle F_{1}(t) &= & \displaystyle S_{H}(0)e^{- \mu_{H} t}+ N_{H}(1-e^{-\mu_{H} t}) - \frac{b^{\alpha} \beta_{V}^{H}\mu_{V}^{1-\alpha} S_{H}^*
				I_{V}^* }{N_{H} \mu_{H}}  (1- e^{-\mu_{H} t})\\&-& \displaystyle
			\frac{b^{\alpha} \beta_{V}^{H} S_{H}^* }{ N_{H}\mu_{H}} (1- e^{-\mu_{H} t})
			o(\epsilon)-S_{H}^{*}.
		\end{array}
		\label{EQ:F7}
	\end{eqnarray}
	Now we have 
	\begin{eqnarray}
		\centering
		\begin{array}{llll}
			\displaystyle	\lim_{t \rightarrow \infty}F_{1}(t)& =& \displaystyle \lim_{t \rightarrow \infty}\left\lbrace  S_{H}(0)e^{- \mu_{H} t}+ N_{H}(1-e^{-\mu_{H} t}) - \frac{b^{\alpha} \beta_{V}^{H}\mu_{V}^{1-\alpha} S_{H}^*
				I_{V}^* }{N_{H} \mu_{H}}  (1- e^{-\mu_{H} t})\right.\\ &-& \displaystyle \left.\quad\frac{b^{\alpha} \beta_{V}^{H} S_{H}^* }{ N_{H}\mu_{H}} (1- e^{-\mu_{H} t})	o(\epsilon)-S_{H}^*\right\rbrace ,\\
			&= & \displaystyle  N_{H} -S_{H}^*- \frac{b^{\alpha} \beta_{V}^{H}\mu_{V}^{1-\alpha} S_{H}^*
				I_{V}^* }{N_{H} \mu_{H}}, \\  & =& \displaystyle N_{H} -\left[ 1+\frac{b^{\alpha} \beta_{V}^{H}\mu_{V}^{1-\alpha}
				I_{V}^* }{N_{H} \mu_{H}}\right] S_{H}^{*}  ,\\ &= & \displaystyle  N_{H} -\left[ 1+\frac{b^{\alpha} \beta_{V}^{H}\mu_{V}^{1-\alpha}
				I_{V}^* }{N_{H} \mu_{H}}\right]  \frac{\mu_{H} N_{H}}{\frac{b^{\alpha}\mu_{V}^{1-\alpha}\beta_{V}^{H}}{N_{H}}I_{V}^{*}+\mu_{H}} ,\\ & = & \displaystyle 0 .\\
		\end{array}
	\end{eqnarray}
	As above, we can easily show that  $\displaystyle \lim_{t \rightarrow \infty} F_{2}(t) = 0$, $\displaystyle \lim_{t \rightarrow \infty} F_{3}(t) = 0$. \\
	Therefore, $\displaystyle \lim_{t \rightarrow \infty} \mathbb{F}(t) = 0$. 
	
	Expanding the characteristic equation~(\ref{EQ:characteristic-equation}), we have $det L = 0$. Where, $ L= (I - \int_{0}^{\infty} e^{-\lambda t} A(t) J dt)$. The matrix $L$ is given as follows:
	
	\begin{eqnarray}
		\begin{array}{llll}
			\displaystyle L_{11} &=& \displaystyle 1+ \frac{b^{\alpha} \beta_{V}^{H}\mu_{V}^{1-\alpha} I_{V}^*}{N_{H}(\lambda+\mu_{H})},\\
			\displaystyle L_{12}&=& 0,\\
			\displaystyle L_{13} &=& \displaystyle -\frac{b^{\alpha} \beta_{V}^{H}\mu_{V}^{1-\alpha} S_{H}^*}{(\lambda+\mu_{H})}, \\	  
			\displaystyle L_{21}  &= &\displaystyle \frac{b^{\alpha}\beta_{V}^{H}\mu_{V}^{1-\alpha} I_{V}^*}{(\lambda+\mu_{H})} ,\\
			\displaystyle L_{22} &=& \displaystyle1-\frac{\mu_{H}^{1-\beta}}{C^{\beta}(\lambda+K)} \\ 
			\displaystyle L_{23}  &=& \displaystyle L_{13} \\
			\displaystyle L_{31} &= &\displaystyle 0, \\
			\displaystyle L_{32} &=& \displaystyle - \frac{b^{p}\beta_{H}^{V} \mu_{H}^{1-p}\Pi_{V}}{N_{H} C^{p} \mu_{V}(\mu_{V}+\lambda)} +\frac{b^{p}\beta_{H}^{V}\mu_{H}^{1-p} }{N_{H} C^{p} (\mu_{V}+\lambda)} I_{V}^{*},\\
			\displaystyle L_{33} &=& \displaystyle 1+\frac{b^{p}\beta_{H}^{V}\mu_{H}^{1-p} }{N_{H} C^{p} (\mu_{V}+\lambda)} I_{H}^{*},\\
		\end{array}		
		\label{EQ:Chacteristic-Expanding}
	\end{eqnarray}
	Expanding the determinant $det L = 0$, the characteristic equation~(\ref{EQ:characteristic-equation}) becomes:
	\begin{eqnarray}
		\lambda^{3} + A_{1} \lambda^{2} + B_{1} \lambda + C_{1}= 0,
	\end{eqnarray} where,
	\begin{eqnarray}
		\centering
		\begin{array}{llll}
			\displaystyle	A_{1} &=& \displaystyle \left(\mu_{V}+ \frac{b^{p} \beta_{H}^{V} \mu_{H}^{1-p}}{N_{H} C^{p}} I_{H}^{*}\right)+\left(\mu_{H}+\frac{\mu_{H}^{1-\alpha}}{C^{\beta}}\right)+ \left(\mu_{H} + \frac{b^{\alpha}\beta_{V}^{H} \mu_{V}^{1-\alpha}}{ N_{H}} I_{V}^{*}\right)>0,\\
			\displaystyle	B_{1} & = & \displaystyle \left(\mu_{H}+\frac{\mu_{H}^{1-\alpha}}{C^{\beta}}\right) \left(\mu_{H} + \frac{b^{\alpha}\beta_{V}^{H} \mu_{V}^{1-\alpha}}{ N_{H}} I_{V}^{*}+ \frac{b^{p} \beta_{H}^{V} \mu_{H}^{1-p}}{N_{H} C^{p}} I_{H}^{*}\right)\\ &+& \displaystyle \left(\mu_{H} + \frac{b^{\alpha}\beta_{V}^{H} \mu_{V}^{1-\alpha}}{ N_{H}} I_{V}^{*}\right)\left(\mu_{V}+ \frac{b^{p} \beta_{H}^{V} \mu_{H}^{1-p}}{N_{H} C^{p}} I_{H}^{*}\right)>0, \\
			\displaystyle	C_{1} & = & \displaystyle \left(\mu_{H}+\frac{\mu_{H}^{1-\alpha}}{C^{\beta}}\right) \left[\mu_{V} \frac{b^{\alpha}\beta_{V}^{H} \mu_{V}^{1-\alpha}}{ N_{H}} I_{V}^{*}+\left(\mu_{H} + \frac{b^{\alpha}\beta_{V}^{H} \mu_{V}^{1-\alpha}}{ N_{H}} I_{V}^{*}\right) \frac{b^{p} \beta_{H}^{V} \mu_{H}^{1-p}}{N_{H} C^{p}} I_{H}^{*} \right]>0.\\
		\end{array}
	\end{eqnarray}
	and
	\begin{eqnarray}
		\begin{array}{llll}
			\displaystyle	A_{1}B_{1}-C_{1}&=& \displaystyle\left[\left(\mu_{H} + \frac{b^{\alpha}\beta_{V}^{H} \mu_{V}^{1-\alpha}}{ N_{H}} I_{V}^{*}\right) \left(\mu_{V}+ \frac{b^{p} \beta_{H}^{V} \mu_{H}^{1-p}}{N_{H} C^{p}} I_{H}^{*}\right)\right.\\&& \displaystyle\left. +\frac{b^{p} \beta_{H}^{V} \mu_{H}^{1-p}}{N_{H} C^{p}} I_{H}^{*} \left(\mu_{H}+\frac{\mu_{H}^{1-\alpha}}{C^{\beta}}\right)\right]\left[\mu_{H}+\frac{\mu_{H}^{1-\alpha}}{C^{\beta}}+\mu_{V}+ \frac{b^{p} \beta_{H}^{V} \mu_{H}^{1-p}}{N_{H} C^{p}} I_{H}^{*}\right]\\&+& \displaystyle\left(\mu_{H}+\frac{\mu_{H}^{1-\alpha}}{C^{\beta}}\right)^{2} \left(\mu_{H} + \frac{b^{\alpha}\beta_{V}^{H} \mu_{V}^{1-\alpha}}{ N_{H}} I_{V}^{*}\right)\\&+&\displaystyle \mu_{H}\left(\mu_{H}+\frac{\mu_{H}^{1-\alpha}}{C^{\beta}}\right)\left(\mu_{V}+ \frac{b^{p} \beta_{H}^{V} \mu_{H}^{1-p}}{N_{H} C^{p}} I_{H}^{*}\right)\\ &+& \displaystyle \left[2\mu_{H}+\frac{\mu_{H}^{1-\alpha}}{C^{\beta}}+\mu_{V}+ \frac{b^{p} \beta_{H}^{V} \mu_{H}^{1-p}}{N_{H} C^{p}} I_{H}^{*}\right]\left(\mu_{H} + \frac{b^{\alpha}\beta_{V}^{H} \mu_{V}^{1-\alpha}}{ N_{H}} I_{V}^{*}\right)^{2}\\&+& \displaystyle \left(\mu_{H}+\frac{\mu_{H}^{1-\alpha}}{C^{\beta}}\right)\frac{b^{\alpha+p} \beta_{V}^{H} \beta_{H}^{V}}{ N_{H}^{2} C^{p} \mu_{V}^{\alpha-1} \mu_{H}^{p-1} } I_{H}^{*} I_{V}^{*},\\ \implies A_{1}B_{1}-C_{1}&>&0.
		\end{array}
	\end{eqnarray} 
From this, we can infer that the endemic equilibrium point $\Psi_{*}$ in $\mathbb {\mathring{T}}$ shows local asymptotic stability, according to the Routh-Hurwitz criterion.
\end{proof}
In the context of model~\eqref{EQ:Fractional-order-vector-borne-disease-converted_model}, under the condition $R_{0}> 1$ we have confirmed global aysmptotic stability of the endemic equilibrium point $\Psi_{*}$ in $\mathbb {\mathring{T}}$, according to~\cite{li1996geometric,li1999global,wu2023global}.

Now, the equation~\eqref{EQ:Fractional-order-vector-borne-disease-converted_model} can be modified using the Theorem~\eqref{Theorem-1} in the following form:
\begin{eqnarray}
	\centering
	\begin{array}{llll}
		\displaystyle e^{-\mu_{V} t}~_{RL}{D_{t}}^{1-\alpha}\left[I_{V}(t) e^{\mu_{V} t}\right] &\approx& \displaystyle \mu_{V}^{1-\alpha} I_{V}(t),\\
		\displaystyle  e^{-\mu_{H} t}~_{RL}{D_{t}}^{1-\beta}\left[I_{H}(t) e^{\mu_{H} t}\right] &\approx& \displaystyle \mu_{H}^{1-\beta} I_{H}(t),\\
		\displaystyle e^{-\mu_{H} t}~_{RL}{D_{t}}^{1-p}\left[I_{H}(t) e^{\mu_{H} t}\right] &\approx & \displaystyle \mu_{H}^{1-p} I_{H}(t).
		\label{EQ:Relation-fractional_derivative-1}
	\end{array}
\end{eqnarray}
Using the relation in~\eqref{EQ:Relation-fractional_derivative-1} in the system~\eqref{EQ:Fractional-order-vector-borne-disease-converted_model}, we have the modified system as follows:

\begin{eqnarray}~\label{EQ:Modified_Fractional-order-vector-borne-disease-model}
	\centering
	\begin{array}{llll}
		\displaystyle \frac{dS_{H}}{dt} &=& \displaystyle \mu_{H} N_{H} - \frac{b^{\alpha}\beta_{V}^{H}  \mu_{V}^{1-\alpha}}{N_{H}}  S_{H}(t) I_{V}(t)  -\mu_{H} S_{H}(t),\\ 
		\displaystyle \frac{dI_{H}}{dt}&= & \displaystyle \frac{b^{\alpha}\mu_{V}^{1-\alpha}\beta_{V}^{H}}{N_{H}}     S_{H}(t) I_{V}(t)- \frac{\mu_{H}^{1-\beta} I_{H}(t)}{C^{\beta}} -\mu_{H} I_{H}(t),\\
		\displaystyle  \frac{dI_{V}}{dt} &= & \displaystyle \frac{b^{p}\beta_{H}^{V} \mu_{H}^{1-p}}{N_{H} C^{p}}  \left(\frac{\Pi_{V}}{\mu_{V}}-I_{V}(t)\right) I_{H}(t) -\mu_{V} I_{V}(t).
	\end{array}
\end{eqnarray}
The system of equation can be formulated as following form:
\begin{eqnarray}~\label{general-equation}
	\frac{dp_j(t)}{dt}=\mathbb {f}_{j}(p_{j}(t)),\hspace{0.3cm} t\geq \sigma,\hspace{0.2cm} j=1,2,3,...,n.
\end{eqnarray}where $\mathbb{R}_{+0}^n=\lbrace (x_1,....x_n): p_j\geq 0,\hspace{0.2cm} j=1,....,n \rbrace$.

Let us define the general principal of global stability using~\cite{li2002global,li1996geometric,wu2023global}.

\begin{lemma}[Theorem-2.2 in~\cite{li2002global}]~\label{lemma-5}
\begin{eqnarray}
	\left(H1\right)&:&  \exists~\hspace{0.1cm} \mathbb{K}\subset \mathbb{D} ~\hspace{0.1cm} \text{such that}~\hspace{0.1cm} \mathbb{K}~\hspace{0.1cm} \text{is compact absorbing set}.\\\nonumber
	\left(H2\right)&:&  \text{In a system}~\eqref{general-equation}~\hspace{0.1cm} \Delta_{*}~\hspace{0.1cm} \text{ denotes unique equilibrium}~\hspace{0.1cm} \text{exhibiting local asymptotic stability}.\\\nonumber
	\left(H3\right)&:&  \text{The Poincare-Bendixson property}~\hspace{0.1cm}\text{is satisfied by the equation }~\eqref{general-equation} . \\\nonumber
	\left(H4\right)&:& \text{In} ~\hspace{0.1cm}\mathbb{D}~\hspace{0.1cm} \text{every periodic orbit of}~\eqref{general-equation}~\hspace{0.1cm}  \text{is orbitally stable in asymptotic sense}.
\end{eqnarray} Hence, $\Delta_{*}$ is globally asymptotically stable in $\mathbb{D}$.
\end{lemma}
\begin{proposition}~\label{prop-6}
If $R_{0}> 1$ the endemic equilibrium point $\Psi_{*}$ of~\eqref{EQ:Fractional-order-vector-borne-disease-converted_model} indicates global asymptotic stablity in $\mathbb {\mathring{T}}$.
\end{proposition}
\begin{proof}If the model~\eqref{EQ:Modified_Fractional-order-vector-borne-disease-model} is uniformly persistent then there exist absorbing compact set~\cite{wu2023global,freedman1994uniform}. The positive invariant boundary $\mathbb{\partial T}$ of $\mathbb{T}$ contain unique disease-free equilibrium $\Psi_{0}$ of the model~\eqref{EQ:Modified_Fractional-order-vector-borne-disease-model}. Furthermore, set $\lbrace\Psi_{0}\rbrace$ is maximal  and isolated on the boundary $\mathbb{\partial T}$. Now, following \citep{wu2023global} and the \textbf{Theorem-4.3} of~ \cite{freedman1994uniform} we see that instability of $\Psi_{0}$ is identical to the uniform persistence of the system of equation~\eqref{EQ:Modified_Fractional-order-vector-borne-disease-model}. Now using the \textbf{proposition}~\eqref{disease_free_global_prop-4} for instability of $\Psi_{0}$ in the circumstance of $R_{0}>1$ and the \textbf{Theorem-4.3} of~ \cite{freedman1994uniform}. Hence, we established that the model~\eqref{EQ:Modified_Fractional-order-vector-borne-disease-model} is uniformly persistent if $R_{0}>1$. Therefore $\mathbb{\partial T} \subset \mathbb{T}$ is a compact absorbing set. Hence, the condition $(H1)$ is fulfilled. The condition $(H2)$ also satisfied, since, $\Psi_{*}$ the endemic equilibrium point is unique and locally asymptotic stable by~\textbf{proposition}~\eqref{prop-5}.
	
	Now, we verify the condition $(H3)$.
	Let the jacobian matrix of the system~\eqref{EQ:Modified_Fractional-order-vector-borne-disease-model} be defined as:
	\begin{eqnarray}
		\centering
		\small
		\begin{array}{llll}
			J &=& \left(
			\begin{array}{c c c}
				- \frac{b^{\alpha}\beta_{V}^{H} \mu_{V}^{1-\alpha}}{ N_{H}} I_{V}^{*} -\mu_{H} & 0 & -\frac{b^{\alpha}\beta_{V}^{H} \mu_{V}^{1-\alpha}}{ N_{H}} S_{H}^{*} \\\\ \frac{b^{\alpha}\beta_{V}^{H} \mu_{V}^{1-\alpha}}{ N_{H}} I_{V}^{*} & -\mu_{H}-\frac{\mu_{H}^{1-\beta}}{C^{\beta}} & \frac{b^{\alpha}\beta_{V}^{H} \mu_{V}^{1-\alpha}}{ N_{H}} S_{H}^{*}\\\\
				0 & \frac{b^{p} \beta_{H}^{V} \mu_{H}^{1-p}}{N_{H} C^{p}}  \left(\frac{\Pi_{V}}{\mu_{V}}-I_{V}^{*}\right) & -\mu_{V}-\frac{b^{p} \beta_{H}^{V} \mu_{H}^{1-p}}{N_{H} C^{p}}  I_{H}^{*}\\\\
			\end{array}
			\right),\\
		\end{array}
		\hspace{0.3cm}\label{EQ:Jacobian-3}
	\end{eqnarray}
	choosing the matrix
	\begin{eqnarray}
		\centering
		\begin{array}{llll}
			H &=& \left(
			\begin{array}{c c c}
				-1 & 0 & 0 \\\\ 0 & 1 & 0\\\\
				0 & 0 & 1\\\\
			\end{array}
			\right),\\
		\end{array}
		\label{EQ:Jacobian-4}
	\end{eqnarray}
	Then,
	\begin{eqnarray}
		\centering
		\small
		\begin{array}{llll} 
			\displaystyle	H J H &=&  \displaystyle \left(
			\begin{array}{c c c}
				- \frac{b^{\alpha}\beta_{V}^{H} \mu_{V}^{1-\alpha}}{ N_{H}} I_{V} -\mu_{H} & 0 & -\frac{b^{\alpha}\beta_{V}^{H} \mu_{V}^{1-\alpha}}{ N_{H}} S_{H} \\\\
				-\frac{b^{\alpha}\beta_{V}^{H} \mu_{V}^{1-\alpha}}{ N_{H}} I_{V} & -\mu_{H}-\frac{\mu_{H}^{1-\beta}}{C^{\beta}} & -\frac{b^{\alpha}\beta_{V}^{H} \mu_{V}^{1-\alpha}}{ N_{H}} S_{H}\\\\
				0 & -\frac{b^{p} \beta_{H}^{V} \mu_{H}^{1-p}}{N_{H} C^{p}}  \left(\frac{\Pi_{V}}{\mu_{V}}-I_{V}\right) & -\mu_{V}-\frac{b^{p} \beta_{H}^{V} \mu_{H}^{1-p}}{N_{H} C^{p}}  I_{H}\\\\
			\end{array}
			\right),
		\end{array}	
		\hspace{0.4cm}\label{EQ:Jacobian-5}
	\end{eqnarray}
	Hence, the system of equation~\eqref{EQ:Modified_Fractional-order-vector-borne-disease-model} is said to be competitive in  $\mathbb {\mathring{T}}$, since, $H J H$ has non-positive off diagonal elements $\forall$ $\left(S_{H}, I_{H}, I_{V}\right)$$\in$   $\mathbb{\mathring{T}}$. Since,  $\mathbb{\mathring{T}}\subset \mathbb R^{3}$ is convex and the system~\eqref{EQ:Modified_Fractional-order-vector-borne-disease-model} is competitive then using the result of~\cite{hirsch2006monotone} and the theorem$(2.1)$ of~\cite{li2002global} we have the condition of $(H3)$ holds.
	
	To prove condition $(H4)$ we follow~\cite{li1996geometric}. 
	Let the second additive compound matrix of~\eqref{EQ:Modified_Fractional-order-vector-borne-disease-model} calculated as follows:
	\begin{eqnarray}
		\centering
		\footnotesize
		\begin{array}{llll}
			J^{[2]} &=& \left(
			\begin{array}{c c c}
				-  \Lambda_{1} I_{V} -2\mu_{H} -\Lambda_{2}& \Lambda_{1} S_{H} & \Lambda_{1} S_{H} \\\\ \Lambda_{3} \left(\frac{\Pi_{V}}{\mu_{V}}-I_{V}\right) & - \Lambda_{1} I_{V}-\mu_{H}-\mu_{V}-\Lambda_{3}  I_{H} & 0\\\\
				0 & \Lambda_{1} I_{V} & -\mu_{H}-\Lambda_{2}-\mu_{V}-  \Lambda_{3}I_{H}\\\\
			\end{array}
			\right),\\
		\end{array}
		\hspace{0.4cm}	\label{EQ:Jacobian-6}
	\end{eqnarray} where, 
	\begin{eqnarray}
		\displaystyle	\Lambda_{1} &=&  \displaystyle \frac{b^{\alpha}\beta_{V}^{H} \mu_{V}^{1-\alpha}}{ N_{H}},\\\nonumber
		\displaystyle	\Lambda_{2}& =&\frac{\mu_{H}^{1-\alpha}}{C^{\beta}},\\\nonumber
		\displaystyle	\Lambda_{3}& =& \displaystyle \frac{b^{p} \beta_{H}^{V} \mu_{H}^{1-p}}{N_{H} C^{p}}.\\\nonumber
	\end{eqnarray}
	\begin{eqnarray}
		\centering
		\begin{array}{llll}
			A &=& \left(
			\begin{array}{c c c}
				1 & 0 & 0 \\\\ 0 & \frac{I_{H}}{I_{V}} & 0\\\\
				0 & 0 & \frac{I_{H}}{I_{V}}\\\\
			\end{array}
			\right),\\
		\end{array}
		\label{EQ:Jacobian-7}
	\end{eqnarray}
	\begin{eqnarray}
		\centering
		\begin{array}{llll}
			A_{f} &=& \left(
			\begin{array}{c c c}
				0 & 0 & 0 \\\\ 0 & \frac{I_{H}^{'}}{I_{V}}-\frac{I_{H} I_{V}^{'}}{I_{V}^{2}} & 0\\\\
				0 & 0 & \frac{I_{H}^{'}}{I_{V}}-\frac{I_{H} I_{V}^{'}}{I_{V}^{2}}\\\\
			\end{array}
			\right),\\
		\end{array}
		\label{EQ:Jacobian-8}
	\end{eqnarray}
	\begin{eqnarray}
		\centering
		\small
		\begin{array}{llll}
			A_{f} A^{-1} &=& \left(
			\begin{array}{c c c}
				0 & 0 & 0 \\\\ 0 & \frac{I_{H}^{'}}{I_{H}}-\frac{ I_{V}^{'}}{I_{V}} & 0\\\\
				0 & 0 & \frac{I_{H}^{'}}{I_{H}}-\frac{ I_{V}^{'}}{I_{V}}\\\\
			\end{array}
			\right),\\
		\end{array}
		\label{EQ:Jacobian-9}
	\end{eqnarray}
	\begin{eqnarray}
		\centering
		\footnotesize
		\begin{array}{llll}
			\displaystyle	B&=&  \displaystyle A_{f} A^{-1}+A_{f}  J^{[2]} A^{-1}\\ &=&  \displaystyle \left(
			\begin{array}{c c c}
				- \Lambda_{1} I_{V} -2\mu_{H} -\Lambda_{2} & \Lambda_{1} S_{H} & \Lambda_{1} S_{H} \\\\ \Lambda_{3} \left(\frac{\Pi_{V}}{\mu_{V}}-I_{V}\right) & - \Lambda_{1} I_{V}-\mu_{H}-\mu_{V}-\Lambda_{3}  I_{H} & 0\\\\
				0 & \Lambda_{1} I_{V} & -\mu_{H}-\Lambda_{1}-\mu_{V}-\Lambda_{3} I_{H}\\\\
			\end{array}
			\right),\\&=&\left(
			\begin{array}{c c}
				B_{11}& B_{12}\\\\ B_{21} & B_{22}\\\\
			\end{array}
			\right),\\
		\end{array}
		\label{EQ:Jacobian-10}
	\end{eqnarray} where,
	\begin{eqnarray}
		B_{11}&=& -\Lambda_{1} I_{V} - 2 \mu_{H} -\Lambda_{2} \\\nonumber
		B_{12}&=&\begin{pmatrix}
			\frac{\Lambda_{1} S_{H} I_{V}}{I_{H}} && \frac{\Lambda_{1} S_{H} I_{V}}{I_{H}}\\\nonumber
		\end{pmatrix},\\\nonumber
		B_{21}&=&\begin{pmatrix}
			\Lambda_{3} \left( \frac{ \Pi{V}}{\mu_{V}}-I_{V}\right) \frac{ I_{H}}{I_{V}}\\\nonumber 0 \\\nonumber
		\end{pmatrix},\\\nonumber
		B_{22}&=&\begin{pmatrix}
			\frac{ I_{H}^{'}}{I_{H}}-\frac{ I_{V}^{'}}{I_{V}}-\Lambda_{1} I_{V} -\mu_{H} -\mu_{V}-\Lambda_{3} I_{H} && 0\\\nonumber \Lambda_{1} I_{V}& &\frac{ I_{H}^{'}}{I_{H}}-\frac{ I_{V}^{'}}{I_{V}}-\Lambda_{1} I_{V} -\mu_{H} -\mu_{V}-\Lambda_{3} I_{H}\\\nonumber 
		\end{pmatrix}.\\\nonumber
	\end{eqnarray}
	Now referring to~\cite{li1996geometric}, we define, $\Bar{q_{2}}$ as follows:
	\begin{eqnarray}~\label{EQ:Lozinski_Measure}
		\Bar{q_{2}}&=&\lim_{t\to\infty}sup \sup_{\left(S_{H}(0),I_{H}(0),I_{V}(0)\right)\in \mathbb{\mathring T}} \frac{1}{t}\int_{0}^{t} \mu(B) ds.\\\nonumber
	\end{eqnarray} where,
	\begin{eqnarray}~\label{EQ:Lozinski_Measure-1}
		\mu(B)\leq \sup \lbrace g_{1}, g_{2} \rbrace
	\end{eqnarray} 
	and, 
	\begin{eqnarray}~\label{EQ:Lozinski_Measure-2}
		\displaystyle	g_{1} &= &  \displaystyle \mu_{1}(B_{11}) + |B_{12}| ,\\\nonumber&=&  \displaystyle -\Lambda_{1} I_{V} - 2 \mu_{H} -\Lambda_{2} +\frac{\Lambda_{1} S_{H} I_{V}}{I_{H}},\\\nonumber
		\displaystyle	g_{2}&=&   \displaystyle |B_{21}|+\mu_{1}(B_{22}) ,\\\nonumber & =&  \displaystyle \Lambda_{3} \left( \frac{ \Pi_{V}}{\mu_{V}}-I_{V}\right) \frac{ I_{H}}{I_{V}}+\frac{ I_{H}^{'}}{I_{H}}-\frac{ I_{V}^{'}}{I_{V}} -\mu_{H} -\mu_{V}-\Lambda_{3} I_{H}.\\\nonumber
	\end{eqnarray}
	Now, from~\eqref{EQ:Modified_Fractional-order-vector-borne-disease-model} we have,
	\begin{eqnarray}~\label{EQ:Lozinski_Measure-3}
		\displaystyle	\frac{ I_{H}^{'}}{I_{H}}&=&  \displaystyle \frac{\Lambda_{1} S_{H} I_{V}}{I_{H}}-\mu_{H}-\Lambda_{2},\\\nonumber
		\displaystyle	\frac{ I_{V}^{'}}{I_{V}}&=&  \displaystyle \Lambda_{3} \left( \frac{ \Pi_{V}}{\mu_{V}}-I_{V}\right) \frac{ I_{H}}{I_{V}}-\mu_{V}.\\\nonumber
	\end{eqnarray}
	Now using the result~\eqref{EQ:Lozinski_Measure-3} in~\eqref{EQ:Lozinski_Measure-2} we have,
	\begin{eqnarray}~\label{EQ:Lozinski_Measure-4}
		g_{1} &= \frac{ I_{H}^{'}}{I_{H}}-  \mu_{H} -\Lambda_{1} I_{V} \\\nonumber&\leq \frac{ I_{H}^{'}}{I_{H}}-  \mu_{H} ,\\\nonumber
		g_{2}&= \frac{ I_{H}^{'}}{I_{H}} -\mu_{H} -\Lambda_{3} I_{H}\\\nonumber&\leq \frac{ I_{H}^{'}}{I_{H}}-  \mu_{H} .\\\nonumber
	\end{eqnarray} Therefore, using~\eqref{EQ:Lozinski_Measure-4} in~\eqref{EQ:Lozinski_Measure-1} we have,\begin{eqnarray}~\label{EQ:Lozinski_Measure-5}
		\mu(B)\leq \frac{ I_{H}^{'}}{I_{H}}- \mu_{H}.
	\end{eqnarray}  Hence,
	\begin{eqnarray}
		\lim_{t\to\infty} sup \frac{1}{t}\int_{0}^{t} \mu(B) ds &\leq \frac{1}{t} \log\left(\frac{I_{H}(t)}{I_{H}(0)}\right)- \mu_{H},\\\nonumber& <-\mu_{H} .\\\nonumber
	\end{eqnarray} Hence, using~\eqref{EQ:Lozinski_Measure} we have,
	\begin{eqnarray}
		\Bar{q_{2}} <0.
	\end{eqnarray}
	Since, $\mathbb{T}$ is connected, and $\Bar{q_{2}}<0$ then the condition $(H4)$ is satisfied~\cite{li1996geometric,li1999global,li2002global}. Hence, if $R_{0}>1$, the endemic equilibrium point $\Psi_{*}$ in $\mathbb {\mathring{T}}$ is globally asymptotic stable.
\end{proof}

\section{Optimal Control Analysis}~\label{Optimal Control}

To control dengue outbreak in an affected region, we considered the following three control policies in the proposed fractional order dengue model~\eqref{EQ:Fractional-order-vector-borne-disease-model}:
\begin{enumerate}
	\item Individual precaution which reduce mosquito biting rate by using mosquito net, repellent, \textit{etc.} Using this control, the effective transmission rates become: $\displaystyle \frac{b^{\alpha} \beta_{V}^{H} (1- \psi(t))}{N_{H}}$, and $\displaystyle \frac{b^{p} \beta_{H}^{V} (1- \psi(t))}{N_{H}}$, where, $0 \leq \psi(t) \leq 1$.  
	\vspace{0.3cm}
	\item The second control is the reduction in the adult mosquito recruitment rate by reducing aquatic transition (\textit{e.g.} killing mosquito egg, larvae, and pupae by using some aquatic insecticides). Thus by applying this control the modified recruitment rate is $\Pi_{V} (1- \zeta(t))$, where, $0 \leq \zeta(t) \leq 1$.
	\vspace{0.3cm}
	\item The final control we considered in the dengue model~~\eqref{EQ:Fractional-order-vector-borne-disease-model} is the adult vector control by spraying ultra-low-volume insecticide applications that kill adult mosquitoes. The modified mosquito death rate become $\mu_{V} + c_{m} \kappa(t)$, where $0 \leq \kappa(t) \leq 1$, and $0.2 \leq c_{m} \leq 0.8$~\cite{abboubakar2021mathematical}.	   
\end{enumerate}

Therefore, based on these assumptions the dengue control model with fractional order transmission and recovery become:

\begin{eqnarray}
	\begin{array}{llll}
		\displaystyle \frac{dS_{H}}{dt} &=& \displaystyle \mu_{H} N_{H} - \frac{b^{\alpha} \beta_{V}^{H}}{N_{H}}  S_{H}(t)\left( 1-\psi(t)\right)  e^{- \mu_{V}t}~ _{RL}{D_{t}}^{1-\alpha} \left[I_{V}(t) e^{\mu_{V} t} \right] -\mu_{H} S_{H}(t),\\ \displaystyle \frac{dI_{H}}{dt}&=& \displaystyle \frac{b^{\alpha} \beta_{V}^{H}}{N_{H}}  S_{H}(t)\left( 1-\psi(t)\right)  e^{- \mu_{V} t}~ _{RL}{D_{t}}^{1-\alpha}\left[I_{V}(t) e^{\mu_{V} t} \right]- \frac{e^{-\mu_{H} t}}{C^{\beta}}~_{RL}D_{t}^{1-\beta}\left[I_{H}(t) e^{\mu_{H} t} \right] -\mu_{H} I_{H}(t),\\ \frac{dR_{H}}{dt}&=& \displaystyle \frac{e^{-\mu_{H} t}}{C^{\beta}}~_{RL}D_{t}^{1-\beta}\left[I_{H}(t) e^{\mu_{H} t} \right] - \mu_{H} R_{H}(t),\\ 
		\displaystyle \frac{dS_{V}}{dt}&= &\displaystyle \Pi_{V} (1-\zeta(t))- \frac{b^{p} \beta_{H}^{V}}{N_{H} C^{p}}  S_{V}(t) \left( 1-\psi(t)\right)  e^{-\mu_{H} t}~_{RL}D_{t}^{1-p}\left[I_{H}(t) e^{\mu_{H} t}\right]  -\left( \mu_{V}+c_m \kappa(t)\right) S_{V}(t),\\ 
		\displaystyle \frac{dI_{V}}{dt}&= &\displaystyle \frac{b^{p} \beta_{H}^{V}}{N_{H} C^{p}}  S_{V}(t) \left( 1-\psi(t)\right)  e^{-\mu_{H} t} ~_{RL}D_{t}^{1-p}\left[I_{H}(t) e^{\mu_{H} t}\right] -\left( \mu_{V}+c_m \kappa(t)\right) I_{V}(t),
	\end{array}
	\label{EQ:Fractional-order-vector-borne-disease-model_with_control}
\end{eqnarray}where, $0 < \alpha \leq 1$, $0 <\beta \leq p \leq 1$, $S_{H}(0) = S^{H}_{0} \geq 0$, $I_{H}(0) = I^{H}_{0}\geq 0$, $R_{H}(0) = R^{H}_{0}\geq 0$, $S_{V}(0) = S^{V}_{0}\geq 0$, $I_{V}(0) = I^{V}_{0}\geq 0$.

Our target is to minimize the following cost function:
\begin{eqnarray}
	\begin{array}{llll}
		\displaystyle J(\psi, \zeta, \kappa) &= &\displaystyle \int_{0}^{T_{max}} \left[ A_{1} N_{V}  + A_{2} I_H + B_{1} S_{V}\zeta(t) + \frac{1}{2} B_{2} \zeta(t)^{2} + B_{3} N_{V} \kappa(t) + \frac{1}{2} B_{4} \kappa(t)^{2} \right. \\\\ & &  \displaystyle \left. + B_{5} N_{H} \psi(t) + \frac{1}{2} B_{6} \psi(t)^{2}\right] dt,
	\end{array}
	\label{EQ:Cost-functional-optimal-control}
\end{eqnarray} 
subject to the model with control~\eqref{EQ:Fractional-order-vector-borne-disease-model_with_control}.

Following~\cite{abboubakar2021mathematical}, $A_{1}$, $A_{2}$, $B_{1}$, $B_{2}$, $B_{3}$, $B_{4}$, $B_{4}$, and $B_{6}$ are weight constants related to the cost-functional~\eqref{EQ:Cost-functional-optimal-control} and their values are provided below:
\begin{table}[h]
	\tabcolsep 4pt
	\centering
	\caption{\bf{Weight coefficients in the cost functional~(\ref{EQ:Cost-functional-optimal-control})}}
	\begin{tabular}{|p{5cm} p{5cm}   p{2cm}|}
		\hline \footnotesize{{Weight parameter}} &
		\footnotesize{{Value}} & \footnotesize{{Reference}}\\
		\hline
		\footnotesize{$A_{1}$} & \footnotesize{$1$}  &
		\footnotesize{\cite{abboubakar2021mathematical}}\\
		\footnotesize{$A_{2}$} & \footnotesize{$10$} &
		\footnotesize{\cite{abboubakar2021mathematical}}\\			
		\footnotesize{$B_{1}$} & \footnotesize{$5\hspace{0.1cm}E-10$} &
		\footnotesize{\cite{abboubakar2021mathematical}} \\
		\footnotesize{$B_{3}$} & \footnotesize{$5\hspace{0.1cm} E-10$}&\footnotesize{\cite{abboubakar2021mathematical}}\\
		\footnotesize{$B_{5}$} & \footnotesize{$5\hspace{0.1cm}E-10$} &
		\footnotesize{\cite{abboubakar2021mathematical}}\\			
		\footnotesize{$B_{2}$} & \footnotesize{$1$} &
		\footnotesize{\cite{abboubakar2021mathematical}} \\
		\footnotesize{$B_{4}$} & \footnotesize{$1$}  &
		\footnotesize{\cite{abboubakar2021mathematical}}\\
		\footnotesize{$B_{6}$} & \footnotesize{$5$} &
		\footnotesize{\cite{abboubakar2021mathematical}}\\
		[0.2ex]
		\hline
	\end{tabular}
	\label{Tab:Control-weight-Table}
\end{table}
Here, $A_{2} > A_{1}$ implies that the control mainly target in reducing infected human population~\cite{abboubakar2021mathematical}. Quadratic term in the control functional represents high levels of intervention during epidemic~\cite{sardar2013optimal}.   

Aim is to determine an optimal solution $\displaystyle \xi^{*} = \displaystyle \Large(\psi^{*}, \zeta^{*}, \kappa^{*} \Large) \in \Sigma$ such that 

\begin{eqnarray}\label{EQ:Optimal-control-Objective}
	\displaystyle J(\xi^{*}) &= \displaystyle \min \limits_{\xi \in \Sigma} J(\psi, \zeta, \kappa),  
\end{eqnarray}with constraint provided by the system~\eqref{EQ:Fractional-order-vector-borne-disease-model_with_control} and 
\begin{align}~\label{EQ:Optimal-control-boundary}
	\displaystyle \Sigma &= \displaystyle \Bigg\lbrace \Bigl(\psi, \zeta, \kappa \Bigr)~|~{}  \psi, \zeta, \kappa \in L^{1}(0, T_{max}), \Bigl(\psi(t), \zeta(t), \kappa(t) \Bigr) \in [0,1]\times[0,1]\times[0,1]~{} \forall~ t \in [0, T_{max}] \Bigg\rbrace.
\end{align}

\subsection{\textbf{Optimal control problem \& Existence of solution }}
Next we will identify sufficient conditions that ensure a solution exists to the optimal control problem~\eqref{EQ:Cost-functional-optimal-control}. To achieve this, we validate the conditions provided in \textbf{Theorem~4.1 (Chapter-III)}~\cite{fleming2012deterministic} and \textbf{Theorem~3.1}~\cite{gaff2009optimal} in light of the optimal control model~\eqref{EQ:Fractional-order-vector-borne-disease-model_with_control}. Hence, we summarize this as follows:

\begin{theorem}~\label{Theorem-optimal-control} Assume the listed requirements are met:
	\begin{enumerate}
		\item[\textbf{(H1)}] There exists a solution for the control system~\eqref{EQ:Fractional-order-vector-borne-disease-model_with_control} with control variables $\xi \in \Sigma$.\\ 
		\item[\textbf{(H2)}] The control state system can be written as a linear function of the control variables $\psi(t)$, $\zeta(t)$, $\kappa(t)$  with coefficients dependent on time and the state variables. \\
		\item[\textbf{(H3)}] Let the integrand in the equation~\eqref{EQ:Cost-functional-optimal-control} be written as 
		\begin{eqnarray}
			\begin{array}{llll}
				\displaystyle L(t, S_{H}, I_{H}, R_{H}, S_{V}, I_{V}, \psi, \zeta, \kappa) &=&\displaystyle A_{1} N_{V}  + A_{2} I_H + B_{1} S_{V}\zeta(t) + \frac{1}{2} B_{2} \zeta(t)^{2} + B_{3} N_{V} \kappa(t) + \frac{1}{2} B_{4} \kappa(t)^{2} \\ & +& \displaystyle B_{5} N_{H} \psi(t) + \frac{1}{2} B_{6} \psi(t)^{2},\\
			\end{array}
			\hspace{0.3cm}	\label{EQ:Integrand-in-Optimal-control}
		\end{eqnarray}  then the integrand $L$ in equation~\eqref{EQ:Integrand-in-Optimal-control} is convex on $\Sigma$ and also satisfies 
		
		$\displaystyle L(t, S_{H}, I_{H}, R_{H}, S_{V}, I_{V}, \psi, \zeta, \kappa) \geq c_{1} ||(\psi, \zeta, \kappa)||^{\beta} - c_{2}$, where, $c_{1} > 0$, $\beta > 1$.
	\end{enumerate}
	then, an optimal control pair $\Bigg(\psi^{*},\zeta^{*}, \kappa^{*}\Bigg) \in \displaystyle\Sigma$ exists, and corresponding solution $S_{H}^{*}$, $I_{H}^{*}$, $R_{H}^{*}$, $S_{V}^{*}$, $I_{V}^{*}$ to the control state system~\eqref{EQ:Fractional-order-vector-borne-disease-model_with_control} that minimizes $J(\psi, \zeta, \kappa)$ over $\Sigma$. 
\end{theorem}

\begin{proof}
	Condition (b) in \textbf{Theorem~4.1} of~\cite{fleming2012deterministic} is clearly met as $\Sigma$ is compact. Similarly, the set of end conditions of the control problem~\eqref{EQ:Cost-functional-optimal-control} also satisfies the condition (c) in \textbf{Theorem~4.1} of~\cite{fleming2012deterministic}. Therefore, nontrivial requirements in the \textbf{Theorem~4.1 (associated corollary 4.1)} of~\cite{fleming2012deterministic} are provided by the conditions (H1), (H2), and (H3).   
	
	Using the relation~\eqref{EQ:Relation-fractional_derivative-2},
	%\begin{align}\label{EQ:Relation-fractional-derivative}
	%e^{-\mu_{V} t}~_{RL}{D_{t}}^{1-\alpha}\left[I_{V}(t) e^{\mu_{V} t}\right] &\approx \mu_{V}^{1-\alpha} I_{V}(t),\\\nonumber\\\nonumber e^{-\mu_{H} t}~_{RL}{D_{t}}^{1-\beta}\left[I_{H}(t) e^{\mu_{H} t}\right] &\approx \mu_{H}^{1-\beta} I_{H}(t),\\\nonumber\\\nonumber e^{-\mu_{H} t}~_{RL}{D_{t}}^{1-p}\left[I_{H}(t) e^{\mu_{H} t}\right] &\approx \mu_{H}^{1-p} I_{H}(t).
	%\end{align}
	in the control system~\eqref{EQ:Fractional-order-vector-borne-disease-model_with_control}, we have the modified control system as follows:
	\begin{eqnarray}
		\begin{array}{llll}
			\displaystyle	\frac{dS_{H}}{dt} &= & \displaystyle\mu_{H} N_{H} - \frac{b^{\alpha}\beta_{V}^{H} \mu_{V}^{1-\alpha}}{N_{H}}  \frac{\Gamma(1-\alpha)}{\gamma(\mu_{V}\xi_{1},1-\alpha)}S_{H}  I_{V}\left( 1-\psi(t)\right)\\&+& \displaystyle\frac{b^{\alpha}\beta_{V}^{H} \mu_{V}^{1-\alpha}}{N_{H}}\frac{e^{-\mu_{V}\xi_{1}} (\mu_{V}\xi_{1})^{-\alpha}\gamma(\mu_{V}t,\alpha)}{\gamma(\mu_{V}\xi_{1},1-\alpha)\Gamma(\alpha)}S_{H}  I_{V} \left( 1-\psi(t)\right)  -\mu_{H} S_{H},\\ 
			\displaystyle \frac{dI_{H}}{dt}&= &\displaystyle\frac{b^{\alpha}\beta_{V}^{H} \mu_{V}^{1-\alpha}}{N_{H}}  \frac{\Gamma(1-\alpha)}{\gamma(\mu_{V}\xi_{1},1-\alpha)}S_{H}  I_{V}\left( 1-\psi(t)\right)-\frac{b^{\alpha}\beta_{V}^{H} \mu_{V}^{1-\alpha}}{N_{H}}\frac{e^{-\mu_{V}\xi_{1}} (\mu_{V}\xi_{1})^{-\alpha}\gamma(\mu_{V}t,\alpha)}{\gamma(\mu_{V}\xi_{1},1-\alpha)\Gamma(\alpha)}S_{H}  I_{V} \left( 1-\psi(t)\right)\\&-&\displaystyle \frac{\Gamma(1-\beta)}{\gamma(\mu_{H}\xi_{1},1-\beta)}\frac{\mu_{H}^{1-\beta}}{C^{\beta}}~ I_{H}+\frac{e^{-\mu_{H}\xi_{1}} (\mu_{H}\xi_{1})^{-\beta}\gamma(\mu_{H}t,\alpha)}{\gamma(\mu_{H}\xi_{1},1-\beta)\Gamma(\beta)}\frac{\mu_{H}^{1-\beta}}{C^{\beta}}~ I_{H} -\mu_{H} I_{H},\\
			\displaystyle \frac{dR_{H}}{dt}&= &\displaystyle \frac{\Gamma(1-\beta)}{\gamma(\mu_{H}\xi_{1},1-\beta)}\frac{\mu_{H}^{1-\beta}}{C^{\beta}}~ I_{H}-\frac{e^{-\mu_{H}\xi_{1}} (\mu_{H}\xi_{1})^{-\beta}\gamma(\mu_{H}t,\alpha)}{\gamma(\mu_{H}\xi_{1},1-\beta)\Gamma(\beta)}\frac{\mu_{H}^{1-\beta}}{C^{\beta}}~ I_{H} - \mu_{H} R_{H},\\
			\displaystyle \frac{dS_{V}}{dt}&= &  \displaystyle \Pi_{V} (1-\zeta(t))- \frac{b^{p}\beta_{H}^{V} \mu_{H}^{1-p}}{N_{H} C^{p}}\frac{\Gamma(1-p)}{\gamma(\mu_{H}\xi_{1},1-p)}  \left( 1-\psi(t)\right) S_{V} I_{H}\\&+&  \displaystyle \frac{b^{p}\beta_{H}^{V} \mu_{H}^{1-p}}{N_{H} C^{p}} \frac{e^{-\mu_{H}\xi_{1}} (\mu_{H}\xi_{1})^{-p}\gamma(\mu_{H}t,p)}{\gamma(\mu_{H}\xi_{1},1-p)\Gamma(p)}\left( 1-\psi(t)\right)S_{V} I_{H}   -\left( \mu_{V}+c_m \kappa(t)\right) S_{V},\\ 
			\displaystyle	\frac{dI_{V}}{dt}&= &\displaystyle \frac{b^{p}\beta_{H}^{V} \mu_{H}^{1-p}}{N_{H} C^{p}}\frac{\Gamma(1-p)}{\gamma(\mu_{H}\xi_{1},1-p)}  \left( 1-\psi(t)\right) S_{V} I_{H}-\frac{b^{p}\beta_{H}^{V} \mu_{H}^{1-p}}{N_{H} C^{p}} \frac{e^{-\mu_{H}\xi_{1}} (\mu_{H}\xi_{1})^{-p}\gamma(\mu_{H}t,p)}{\gamma(\mu_{H}\xi_{1},1-p)\Gamma(p)}\left( 1-\psi(t)\right)S_{V} I_{H} \\&-& \displaystyle\left( \mu_{V}+c_m \kappa(t)\right) I_{V}, 
		\end{array}
		\label{EQ:Modified-Fractional-order-vector-borne-disease-model_with_control}
	\end{eqnarray} where, $0 < \alpha \leq 1$, $0 <\beta \leq p \leq 1$, $S_{H}(0) = S^{H}_{0}$, $I_{H}(0) = I^{H}_{0}$, $R_{H}(0) = R^{H}_{0}$, $S_{V}(0) = S^{V}_{0}$, $I_{V}(0) = I^{V}_{0}$.
	Adding first three equation of the modified control system~\eqref{EQ:Modified-Fractional-order-vector-borne-disease-model_with_control}, we have:
	\begin{eqnarray}
		\displaystyle \frac{d\bigg(S_{H}+I_{H}+R_{H}\bigg)}{dt} &=& \displaystyle 0,\\\nonumber  \displaystyle \implies S_{H}(t)+I_{H}(t)+R_{H}(t) &=&  \displaystyle N_{H}~\hspace{1cm} \forall t,\\\nonumber  \implies S_{H}(t) &\leq&  \displaystyle N_{H},~{} I_{H}(t) \leq N_{H},~{} R_{H}(t) \leq N_{H},~\hspace{1cm} \forall t.  
	\end{eqnarray}Thus for a finite end time $T_{max}$, $S_{H}$, $I_{H}$, and $R_{H}$ are bounded above by $N_{H} > 0$.
	
	Similarly, the modified control system~\eqref{EQ:Modified-Fractional-order-vector-borne-disease-model_with_control},along with addition of the last two equations generate the following:
	\begin{eqnarray}
		\displaystyle	\frac{d N_{V}}{dt} &= &\displaystyle \Pi_{V} (1-\zeta(t)) - (\mu_{V} + c_{m} \kappa(t)) N_{V},\\\nonumber \implies \frac{d N_{V}}{dt} &\leq &  \displaystyle \Pi_{V} - \mu_{V} N_{V},~\hspace{1cm} \text{as}~{}~{}~{} \zeta(t),~{} \kappa(t) \in[0,1],~{}\text{and}~{}~{} c_{m} \geq 0.  
	\end{eqnarray}Therefore, using the standard comparison theorem~\cite{lakshmikantham1989stability}, we have:
	\begin{eqnarray}\label{EQ:Mosquito-inequality}
		\displaystyle	N_{V}(t) &\leq & \displaystyle  N_{V}(0) e^{-\mu_{V}t}+ \frac{\Pi_{V}}{\mu_{V}} \Bigg(1- e^{-\mu_{V}t} \Bigg).    
	\end{eqnarray}Thus, for a finite end time $T_{max}$, $\exists$ $M>0$, the total vector population $N_{V}(t)$ is bounded above by the quantity $M$. This upper bound of $N_{V}$ is also an upper bound for $S_{V}$, and $I_{V}$.     
	
	Again from the system~\eqref{EQ:Modified-Fractional-order-vector-borne-disease-model_with_control}, and considering the fact that $\xi \in \Sigma$, and since $\lim_{t \rightarrow \infty} \frac{\Gamma(1-\alpha)}{\gamma(\mu_{V}\xi_{1},1-\alpha)}=1$. Also, since, $\lim_{t \rightarrow \infty}$ $e^{-\mu_{V}\xi_{1}} (\mu_{V}\xi_{1})^{-\alpha}=0$, we have,
	$ \frac{\Gamma(1-\alpha)}{\gamma(\mu_{V}\xi_{1},1-\alpha)}-\frac{e^{-\mu_{V}\xi_{1}} (\mu_{V}\xi_{1})^{-\alpha}\gamma(\mu_{V}t,\alpha)}{\gamma(\mu_{V}\xi_{1},1-\alpha)\Gamma(\alpha)}\geq 0$. Then for $\epsilon>0$  we have
	\begin{eqnarray}
		\begin{array}{llll}
			\displaystyle	-\frac{\Gamma(1-\alpha)}{\gamma(\mu_{V}\xi_{1},1-\alpha)}>-(1+\epsilon)~\hspace{2cm} \forall t\geq t_{1},\\\\
			\displaystyle	- e^{-\mu_{V}\xi_{1}} (\mu_{V}\xi_{1})^{-\alpha}>-\epsilon~\hspace{2cm} \forall t\geq t_{2} ,\\
			\displaystyle	-\frac{\gamma(\mu_{V}t,\alpha)}{\Gamma(\alpha)}>-(1+\epsilon)~\hspace{2cm} \forall t\geq t_{3}. \\
		\end{array}
		\label{EQ:Epsilon_Relation_control}
	\end{eqnarray} Substituting~\eqref{EQ:Epsilon_Relation_control} in~\eqref{EQ:Modified-Fractional-order-vector-borne-disease-model_with_control} we have:
	\begin{eqnarray}
		\begin{array}{llll}
			\displaystyle	\frac{dS_{H}}{dt} &\geq & \displaystyle - \frac{b^{\alpha} \mu_{V}^{1-\alpha} \beta_{V}^{H}}{N_{H}}(1+\epsilon) S_{H} I_{V} -\mu_{H} S_{H}\\
			\displaystyle \frac{dI_{H}}{dt} &\geq& \displaystyle -\bigg(\frac{\mu_{H}^{1-\beta}}{C^{\beta}}(1+\epsilon) +\mu_{H}\bigg)I_{H}\\
			\displaystyle \frac{dR_{H}}{dt} &\geq& \displaystyle - \mu_{H} R_{H}\\
			\displaystyle \frac{d S_{V}}{dt} &\geq& \displaystyle - \frac{b^{p} \mu_{H}^{1-p} \beta_{H}^{V}}{N_{H} C^{p}}(1+\epsilon) S_{V} I_{H} -(\mu_{V} +c_{m}) S_{V}\\
			\displaystyle  \frac{d I_{V}}{dt} &\geq& \displaystyle -(\mu_{V} +c_{m}) I_{V}.
		\end{array}
		\label{EQ:inequality-control}
	\end{eqnarray}Therefore, using the standard comparison theorem~\cite{lakshmikantham1989stability}, we have from~\eqref{EQ:inequality-control}:
	\begin{eqnarray}
		\begin{array}{llll}
			\displaystyle	S_{H}(t) &\geq& 	\displaystyle S_{H}(0) \exp{\left[- \int_{0}^{t} \Bigg(\frac{b^{\alpha} \mu_{V}^{1-\alpha} \beta_{V}^{H}}{N_{H}}(1+\epsilon)  I_{V}(\tau) + \mu_{H} \Bigg) d\tau \right]} \geq 0, \\
			\displaystyle I_{H}(t) &\geq& 	\displaystyle I_{H}(0) \exp{\left[- \Bigg(\frac{\mu_{H}^{1-\beta}}{C^{\beta}}(1+\epsilon) + \mu_{H} \Bigg) t 
				\right] } \geq 0, \\ 	
			\displaystyle  R_{H}(t) &\geq& 	\displaystyle R_{H}(0) \exp{\left[- \mu_{H} t \right] } \geq 0,~\hspace{2 cm} \forall t \in [0, T_{max}] \\
			\displaystyle S_{V}(t) &\geq& \displaystyle S_{V}(0) \exp{\left[- \int_{0}^{t} \Bigg(\frac{b^{p} \mu_{H}^{1-p} \beta_{H}^{V}}{N_{H} C^{p}} (1+\epsilon) I_{H}(\tau) + \mu_{V}+c_{m} \Bigg) d\tau \right]} \geq 0, \\
			\displaystyle I_{V}(t) &\geq & 	\displaystyle I_{V}(0) \exp{\left[- (\mu_{V}+c_{m}) t \right]} \geq 0.
		\end{array}
		\label{EQ:inequality-control-new}
	\end{eqnarray}
	Hence, solution to the control system~\eqref{EQ:Modified-Fractional-order-vector-borne-disease-model_with_control} is bounded below for some finite end time $T_{max}$. Consequently, for some fixed end time $T_{max}$ every solution to the system~\eqref{EQ:Modified-Fractional-order-vector-borne-disease-model_with_control} with control $\xi \in \Sigma$ is bounded. Furthermore, the state control system remains continuous in terms of state variables. Moreover, it is clear that in the control state equation~\eqref{EQ:Modified-Fractional-order-vector-borne-disease-model_with_control}, partial derivatives in connection with the state variables are also bounded. Therefore, we prove that the state control system fulfilled the Lipschitz condition regarding the state variables using \textbf{Theorem-1 (page-248)} in~\cite{coddington2012introduction}. Finally, $\forall$$\xi \in \Sigma$, there is a unique solution of the state control system, according to Picard-Lindel$\ddot{o}$f theorem in~\cite{coddingtontheory}. Thus, condition (H1) holds for the dengue control problem~\eqref{EQ:Cost-functional-optimal-control}.
	As the state equations in the system~\eqref{EQ:Fractional-order-vector-borne-disease-model_with_control} are linearly dependent on $\psi(t)$, $\zeta(t)$, and $\kappa(t)$, therefore, the condition (H2) is satisfied for the control problem~\eqref{EQ:Cost-functional-optimal-control}.
	
	Finally, the integrand $L$ defined in equation~\eqref{EQ:Integrand-in-Optimal-control} is a quadratic function of $\psi$, $\zeta$, and $\kappa$, therefore, it is clearly a convex function with respect to the control $\xi \in \Sigma$. 
	
	Since $\xi \in \Sigma$, therefore, $B_{6} \psi^{2}(t) \leq B_{6}$ $\implies$ $B_{6} \psi^{2}(t) - B_{6} \leq 0$. Now, from the equation~\eqref{EQ:Integrand-in-Optimal-control}, we have:
	\begin{eqnarray}
		\begin{array}{llll}
			\displaystyle L(t, S_{H}, I_{H}, R_{H}, S_{V}, I_{V}, \psi, \zeta, \kappa) &=&\displaystyle A_{1} N_{V}  + A_{2} I_H + B_{1} S_{V}\zeta(t) + \frac{ B_{2}}{2} \zeta(t)^{2}\\\\ & +& \displaystyle B_{3} N_{V} \kappa(t) + \frac{ B_{4}}{2} \kappa(t)^{2}  + B_{5} N_{H} \psi(t) + \frac{B_{6}}{2}  \psi(t)^{2} \\\\ &\geq& \displaystyle \frac{B_{2}}{2}  \zeta^{2}(t) + \frac{B_{4}}{2}  \kappa^{2}(t) + \frac{B_{6}}{2}  \psi^{2}(t) \\\\ &\geq& \displaystyle \frac{ B_{2}}{2} \zeta^{2}(t) + \frac{ B_{4}}{2} \kappa^{2}(t) + \frac{ B_{6}}{2} \psi^{2}(t) - B_{6}\\\\ &\geq& \displaystyle \frac{\min{\bigg( B_{2}, B_{4}, B_{6}\bigg)}}{2} \bigg(\zeta^{2} + \kappa^{2}(t) + \psi^{2}(t) \bigg) - B_{6}\\ &=& \displaystyle c_{1} ||(\zeta, \kappa, \psi)||^{2} - c_{2},
		\end{array}
	\end{eqnarray}
	where, $c_{1} = \frac{\min{\large( B_{2}, B_{4}, B_{6}\large)}}{2} > 0$, and $c_{2} = B_{6}$.

	Therefore, (H3) is satisfied for the optimal control problem~\eqref{EQ:Cost-functional-optimal-control}.
	
\end{proof}
\subsection{\textbf{Derivation of adjoint system corresponding to the optimal control problem~\eqref{EQ:Cost-functional-optimal-control}}}
Through the application of Pontryagin's maximum principle~\cite{kirk2004optimal}, we constructed the Hamiltonian in the following manner:
\begin{eqnarray}
	\begin{array}{llll}
		\displaystyle	Z(t) &= &	\displaystyle A_{1} N_{V}  + A_{2} I_H + B_{1} S_{V}\zeta(t) + \frac{ B_{2}}{2} \zeta(t)^{2} + B_{3} N_{V} \kappa(t) + \frac{B_{4}}{2}  \kappa(t)^{2}+	\displaystyle B_{5} N_{H} \psi(t) + \frac{ B_{6}}{2} \psi(t)^{2}\\ &+& \displaystyle \lambda_{1}\left\lbrace  \mu_{H} N_{H} - \frac{b^{\alpha}  \beta_{V}^{H}}{N_{H}} S_{H}(t) \left( 1-\psi(t)\right) e^{- \mu_{V} t}~ _{RL}{D_{t}}^{1-\alpha}\left[I_{V}(t) e^{\mu_{V} t} \right] -\mu_{H} S_{H}\right\rbrace \\ &+ & 	\displaystyle \lambda_{2} \left\lbrace \frac{b^{\alpha} \beta_{V}^{H}}{N_{H}}  S_{H}(t)\left( 1-\psi(t)\right)  e^{- \mu_{V} t}~ _{RL}{D_{t}}^{1-\alpha}\left[I_{V}(t) e^{ \mu_{V} t} \right]- \frac{e^{-\mu_{H} t}}{C^{\beta}}~_{RL}D_{t}^{1-\beta}\left[I_{H}(t) e^{\mu_{H} t} \right] -\mu_{H} I_{H} \right\rbrace  \\ &+&	\displaystyle \lambda_{3}\left\lbrace  \frac{e^{-\mu_{H} t}}{C^{\beta}}~_{RL}D_{t}^{1-\beta}\left[I_{H}(t) e^{\mu_{H} t} \right] - \mu_{H} R_{H}\right\rbrace \\ &+& 	\displaystyle \lambda_{4}\left\lbrace \Pi_{V} (1-\zeta(t))- \frac{b^{p} \beta_{H}^{V}}{N_{H} C^{p}}  S_{V}(t) \left( 1-\psi(t)\right)  e^{-\mu_{H} t}~_{RL}D_{t}^{1-p}\left[I_{H}(t) e^{\mu_{H} t}\right]  -\left( \mu_{V}+c_m \kappa(t)\right) S_{V} \right\rbrace \\ &+ & 	\displaystyle \lambda_{5} \left\lbrace \frac{b^{p} \beta_{H}^{V}}{N_{H} C^{p}}  S_{V}(t) \left( 1-\psi(t)\right)  e^{-\mu_{H} t} ~_{RL}D_{t}^{1-p}\left[I_{H}(t) e^{\mu_{H} t}\right] -\left( \mu_{V}+c_m \kappa(t)\right) I_{V}    \right\rbrace,
	\end{array}
	\label{EQ:Hamiltonian-fractional-dengue-control}
\end{eqnarray} where, $\lambda_{1}$, $\lambda_{2}$, $\lambda_{3}$, $\lambda_{4}$, and $\lambda_{5}$ are adjoint variables associted with the states $S_{H}$, $I_{H}$, $R_{H}$, $S_{V}$, and $I_{V}$. Also, at time $t$, $N_{V}(t)$ represent the total vector population.

Theorem given below corresponds to the derivation of optimal control problem:

\begin{theorem}~\label{Theorem1} Given an optimal control solution $\left(\psi^{*}, \zeta^{*}, \kappa^{*}\right)$, and $S_{H}^{*},I_{H}^{*}, R_{H}^{*}, S_{V}^{*}$, $I_{V}^{*}$ are solution of the corresponding state equation~(\ref{EQ:Fractional-order-vector-borne-disease-model_with_control}) that minimizes the cost-functional $J(\psi, \zeta, \kappa)$ in $\Sigma$. Consequently, there exist adjoint variables $\lambda_{1}$, $\lambda_{2}$, $\lambda_{3}$, $\lambda_{4}$, and $\lambda_{5}$, satisfying the equation
	\begin{equation}
		\begin{split}
	\displaystyle	\frac{d \lambda_{1}}{dt} =- \frac{\partial Z}{\partial S_{H}},\frac{d \lambda_{2}}{dt} =- \frac{\partial Z}{\partial I_{H}}, \frac{d \lambda_{3}}{dt} =- \frac{\partial Z}{\partial R_{H}},\\\nonumber
	\displaystyle	\frac{d \lambda_{4}}{dt} =- \frac{\partial Z}{\partial S_{V}}, \frac{d \lambda_{5}}{dt} =- \frac{\partial Z}{\partial I_{V}},\nonumber
\end{split}
	\end{equation}
 with transversality condition $\lambda_{i} (T_{max}) = 0$, $i=1, 2, 3, 4, 5$. The optimality conditions are provided as:
 \begin{equation}
	\frac{\partial Z}{\partial \psi} = 0, \frac{\partial Z}{\partial \zeta} = 0, \frac{\partial Z}{\partial \kappa} = 0.\\\nonumber
\end{equation}
	 Furthermore, $(\psi^{*}, \zeta^{*}, \kappa^{*})$ is given as:
	\begin{eqnarray}
		\begin{array}{llll}
			\displaystyle \psi^{*} &=& \displaystyle \min\left\lbrace 1, \max \left (0, \frac{1}{B_{6}} \left[ (\lambda_{5}-\lambda_{4}) \frac{b^{p}}{N_{H} C^{p}} \beta_{H}^{V} S_{V} e^{-\mu_{H} t}~_{RL}D_{t}^{1-p}\left[I_{H} e^{\mu_{H} t}\right] \right. \right. \right. \\ & +&  \displaystyle \left. \left. \left.  (\lambda_{2}-\lambda_{1}) \frac{b^{\alpha}}{N_{H}} \beta_{V}^{H} S_{H} e^{-\mu_{V} t}~_{RL}D_{t}^{1-\alpha}\left[I_{V}(t) e^{\mu_{V} t}\right] - B_{5} N_{H} \right] \right) \right\rbrace \\
			\displaystyle \zeta^{*} & = &\displaystyle \min\left\lbrace 1, \max \left( 0, \frac{\lambda_{4} \Pi_{V} - B_{1} S_{V}}{ B_{2}} \right) \right\rbrace \\
			\displaystyle \kappa^{*} & = &\displaystyle \min\left\lbrace 1, \max \left(0, \frac{c_{m} \left(\lambda_{4} S_{V} + \lambda_{5} I_{V} \right) - B_{3} N_{V}}{B_{4}} \right) \right\rbrace
		\end{array}
	\end{eqnarray}
\end{theorem}
\begin{proof}
	\begin{eqnarray}
		\begin{array}{llll}
			\displaystyle	\frac{d \lambda_{1}}{d t} & = &\displaystyle (\lambda_{1}-\lambda_{2}) \frac{b^{\alpha}}{N_{H}} \beta_{V}^{H} (1-\psi(t)) e^{-\mu_{V}t}~ _{RL}{D_{t}}^{1-\alpha}\left[I_{V}(t) e^{\mu_{V}t} \right] + \lambda_{1} \mu_{H} \\
			\displaystyle \frac{d\lambda_{2}}{dt} &= & \displaystyle - A_{2} +\lambda_{2}~\mu_{H}  + (\lambda_{2}-\lambda_{3}) \frac{\mu_{H}^{1-\beta}}{C^{\beta}} + \displaystyle (\lambda_{4}-\lambda_{5}) \frac{b^{p} \mu_{H}^{1-p} \beta_{H}^{V} S_{V} (1-\psi(t))}{ N_{H} C^{p}}\\
			\displaystyle \frac{d\lambda_{3}}{dt} &= &\displaystyle \lambda_{3} \mu_{H} \\
			\displaystyle\frac{d \lambda_{4}}{d t} & =& \displaystyle   -A_{1} - B_{1} \zeta(t) - B_{3} \kappa(t) + (\lambda_{4}-\lambda_{5}) \frac{b^{p}}{N_{H} C^{p}} \beta_{H}^{V}  (1-\psi(t)) e^{-\mu_{H} t}~_{RL}D_{t}^{1-p}\left[I_{H}(t) e^{\mu_{H} t}\right]  +\lambda_{4} \left( \mu_{V} + c_m \kappa(t)\right)\\
			\displaystyle \frac{d \lambda_{5}}{d t} & = &\displaystyle - A_{1} - B_{3}~ \kappa(t) +  (\lambda_{1}-\lambda_{2}) \frac{b^{\alpha} \mu_{V}^{1-\alpha}\beta_{V}^{H} (1-\psi(t)) S_{H}}{ N_{H}}  +\lambda_{5} \left(\mu_{V}+c_m \kappa(t)\right)
		\end{array}
		\hspace{-0.3cm}\label{EQ:Costate-Fractional-Dengue-model}
	\end{eqnarray}
	
	with transversal conditions  $\lambda_{1}(T_{max}) =\lambda_{2}(T_{max}) =\lambda_{3}(T_{max})=\lambda_{4}(T_{max})=\lambda_{5}(T_{max})=0$,
	with optimal condition we have $\frac{\partial Z}{\partial \psi} =\frac{\partial Z}{\partial \zeta} = \frac{\partial Z}{\partial \kappa} = 0 $, at $\psi = \psi^{*}$, $\zeta = \zeta^{*}$, and $\kappa = \kappa^{*}$ respectively.  
	
	Let $(\psi^{*}, \zeta^{*}, \kappa^{*})$ act as optimal solution to the control problem~(\ref{EQ:Cost-functional-optimal-control}) in accordance with control system~\eqref{EQ:Fractional-order-vector-borne-disease-model_with_control}. Hence, by means of Pontryagin's Maximum principle, there exist $\lambda_{1},\lambda_{2},\lambda_{3},\lambda_{4}$ and $\lambda_{5}$ which satisfy the following equation $\frac{d\lambda_{1}}{dt} = -\frac{\partial Z}{\partial S_{H}}$, $\frac{d\lambda_{2}}{dt} = -\frac{\partial Z}{\partial I_{H}}$, $\frac{d\lambda_{3}}{dt} = -\frac{\partial Z}{\partial R_{H}}$, $\frac{d\lambda_{4}}{dt} = -\frac{\partial Z}{\partial S_{V}}$ and $\frac{d\lambda_{5}}{dt} = -\frac{\partial Z}{\partial I_{V}}$, here $Z$ is defined in~(\ref{EQ:Hamiltonian-fractional-dengue-control}) with the transversal conditions  $\lambda_{i}(T_{max}) =0,i=1,2,3,4,5$.
\end{proof}

\section{Estimation procedure}~\label{estimation}
We calibrated the fractional-order model~\eqref{EQ:Fractional-order-vector-borne-disease-model} with the weekly collected dengue data in San Juan, Puerto Rico, between April 9, 2010, and April 2, 2011~\cite{sardar2017mathematical}. During the model calibration process, the unknown parameters to be estimated for the model~\eqref{EQ:Fractional-order-vector-borne-disease-model} are
\begin{itemize}
	\item The fractional-order derivatives $\alpha$ and $ \beta$ (we assume $\beta=p$).\vspace{0.1cm}
	\item Transmission probability of susceptible human from infected vector $\beta_{V}^{H}$.\vspace{0.1cm}
	\item $\beta_{H}^{V}$ is transmission probability from infected human to susceptible vector .\vspace{0.1cm}
	\item Average number of mosquito bite $b$.\vspace{0.1cm}
	\item Ratio between total vector and human population $\delta$.\vspace{0.1cm}
	\item Mortality rate of mosquito $\mu_{V}$.\vspace{0.1cm}
\end{itemize}
Due to uncertainty of the initial state variable condition, we prefer to estimated them during model calibration. Furthermore, we take into account the entire human individual to be constant, togather with $R_{H}(t)=N_{H}-S_{H}(t)-I_{H}(t)$, hence, the compartment $R_{H}(t)$ becomes redundant.  
The new infected cases from the dengue model~\eqref{EQ:Fractional-order-vector-borne-disease-converted_model} at the $t_{j}$ week is expressed as:
\begin{eqnarray}~\label{EQ:Sumsquare}
	N(t_{j},\theta) = N(0) + \int_{0}^{t} \frac{b^{\alpha}}{N_{H}} \beta_{V}^{H} S_{H}(u) e^{-\mu_{V} t}~ _{RL}{D_{u}}^{1-\alpha}\left[I_{V}(u) e^{\mu_{V} u} \right] du
\end{eqnarray}
where, all the unknown model's parameters of~\eqref{EQ:Fractional-order-vector-borne-disease-model} are contained in $\theta$.

We construct the sum of squares function using~\eqref{EQ:Sumsquare}, defined in the following way:
\begin{eqnarray}~\label{EQ:sumsquare}
	SS(\theta)= \sum_{j=1}^{i}\left(N_{t_{j}}-N(t_{j}, \theta)\right)^{2}
\end{eqnarray}
where, $N_{t_{j}}$ is the experimental data at the $t_{j}$ week with complete set of data point $i$. Initially, the Latin-Hypercube-Sampling method applied to draw a sample of $100$ model unknown parameters $\theta$ within the specified range (see Table~\ref{Table:Model-parameters}). Additionally, implementing nonlinear least square on each sample $\theta$ we obtain $100$ estimated $\hat{\theta}$. The lsqcurveft in the optimization toolbox included with MATLAB (Mathworks, R2018a) was used to estimate the model~\eqref{EQ:Fractional-order-vector-borne-disease-model} parameter. We select the sample whose $SS(\theta)$, as defined in equation~\eqref{EQ:sumsquare}, is minimum among the $1000$ obtained samples of $\hat{\theta}$. Next, the MCMC sampling method is performed for $10,000$ samples, with this minimum $\hat{\theta}$ used as the initial guess. The Delayed Rejection Adaptive Metropolis-Hastings (DRAM) algorithm~\cite{haario2001adaptive,haario2006dram} technique is adopted to sample the unknown parameters of the model~\eqref{EQ:Fractional-order-vector-borne-disease-model}. In order to make sure the chain converged, Gewekes Z-scores were checked. Additionally, we have developed an implicit-$\theta$ scheme (see \textbf{Appendix A}) to solve the fractional order dengue model~\eqref{EQ:Fractional-order-vector-borne-disease-model}. In Table~\ref{Table:Model-parameters}, we have provided the parameter values that were utilized during the model simulation. Fig~\ref{Fig:Model_fitting} displays result. Evaluated parameter values, encompassing human and vector demographic parameters for model~\eqref{EQ:Fractional-order-vector-borne-disease-model} at San Juan province in Puerto Rico, are provided as  [Mean $(95\% CI)$] (see Table~\ref{Tab:estimated-parameters-Table}).

\begin{figure}[H]
	\begin{center}
		\includegraphics[width=1.0\textwidth]{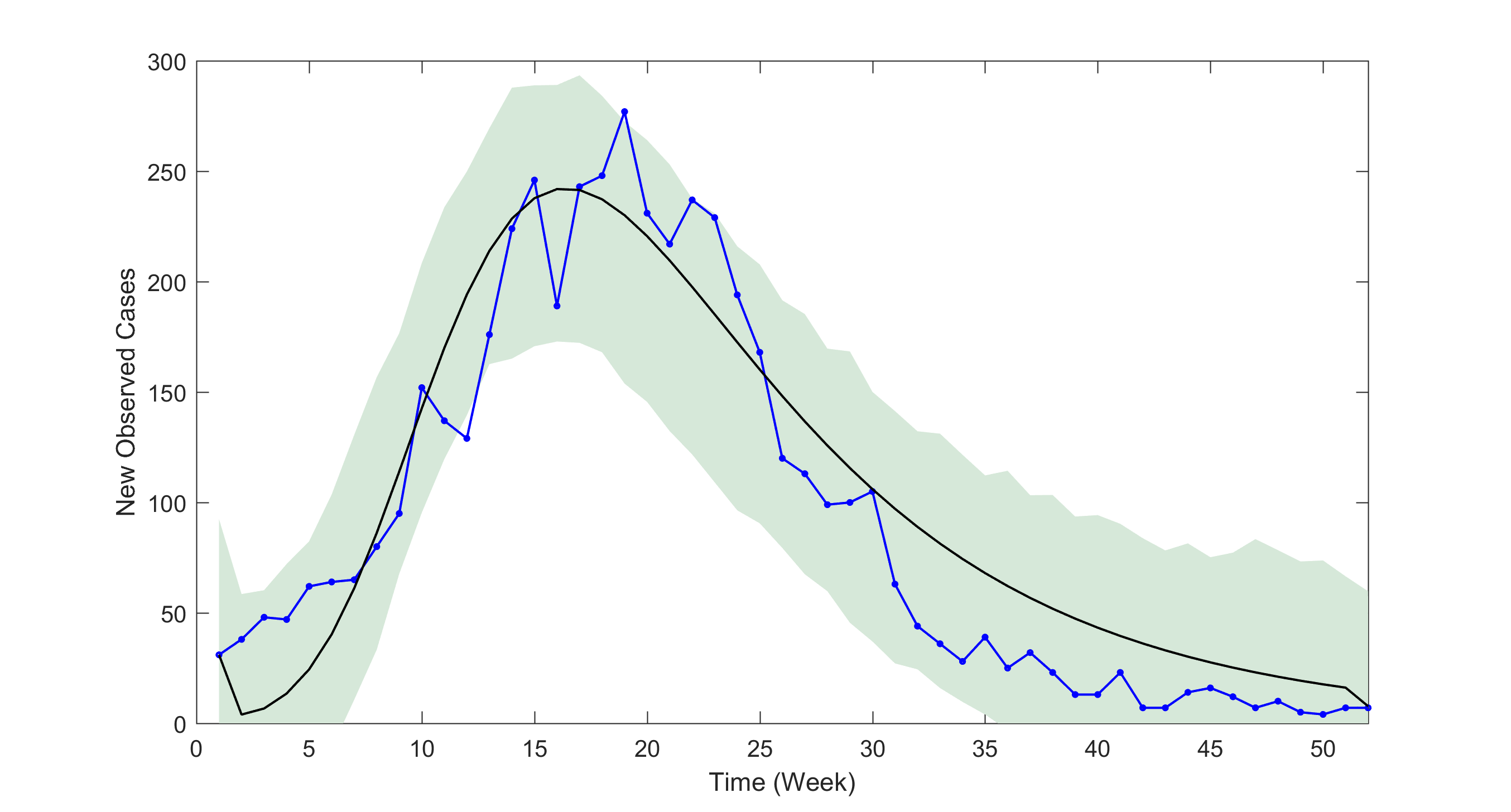}
	\end{center}
	\caption{A fractional-order dengue model~\eqref{EQ:Fractional-order-vector-borne-disease-model} was used to fit weekly recorded cases of dengue at San Juan province of Puerto Rico taken from  April 9, 2010, to  April 2 2011. The blue dotted line represents the weekly observed cases from the data, while the black line shows the new dengue cases obtained from model~\eqref{EQ:Fractional-order-vector-borne-disease-model} solution. Shaded region represents the $95\%$ confidence interval. To solve the fractional-order dengue model~\eqref{EQ:Fractional-order-vector-borne-disease-model} the $\theta-Method$ (see \textbf{Appendix A} was employed with a time step and $\theta$ are taken as $h= \frac{1}{5}$, and $\theta = 0$, respectively.}
	\label{Fig:Model_fitting}
\end{figure}

\section{Sensitivity testing}~\label{sensitivity}
In order to developed an effective control policy, it is necessary to establish a correlation between set of epidemiologically measurable parameters of the model~\eqref{EQ:Fractional-order-vector-borne-disease-model} with new dengue cases and the basic reproduction number $R_{0}$. The developed mathematical model includes several important parameters, as listed in Table~\ref{Table:Model-parameters}. Among these parameters, seven are epidemiologically important: $\alpha$ (the fractional-order derivatives on the disease transmission process), $\beta$ (the fractional-order derivatives in recovery process), $b$ (average number of mosquito bite), $\beta_{H}^{V}$ (probability of transmission from infected human to susceptible vector), $\beta_{V}^{H}$( transmission probability from infected vector to susceptible human), $\delta$( ratio between total vector and human population), and $\mu_{V}$( death rate of mosquito). We applied global sensitivity analysis~\cite{wu2013sensitivity} to measure how the seven model parameters influence total dengue cases and the basic reproduction ratio ($R_{0}$). The relationship between the seven mentioned parameters and the specified responses displays a non-linear and monotonic nature. Therefore, priority placed on computing the partial rank correlation coefficient (PRCC) for these seven parameters against the responses. Using Monte Carlo stratified simulations known as Latin hypercube sampling (LHS)~\cite{stein1987large}, $1000$ samples for  $\alpha$, $\beta$, $b$, $\beta_{H}^{V}$, $\beta_{V}^{H}$,  $\delta$, and $\mu_{V}$ have been drawn from their respective range (see Table~\ref{Tab:estimated-parameters-Table}). Hence, using PRCC~\cite{wu2013sensitivity}, we calculate and its p-value to establish an association between the mention parameters with total dengue cases and the basic reproduction number ($R_{0}$).

\begin{figure}[t]
	\begin{center}
		\includegraphics[width=1.0\textwidth]{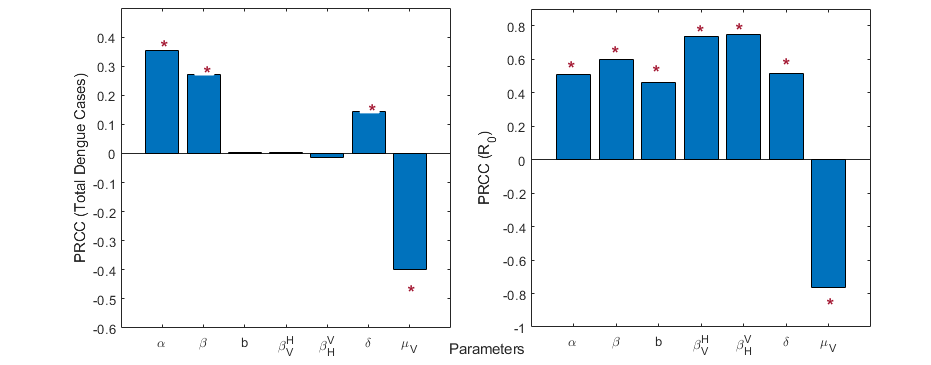}
	\end{center}
	\caption{Global sensitivity analysis of seven modelparameters~(\ref{EQ:Fractional-order-vector-borne-disease-model}) on the responses 1. total number of dengue cases and 2. basic reproduction number ($R_0$), respectively, in San Juan province of Puerto Rico. The total number of cases measured from $9^{th}$ April, $2010 $ to $2^{nd}$ May, $2011$. PRCC measure the impact of parameter uncertainty on the two previously mentioned responses. With Latin hypercube sampling technique ($LHS$), 1000 samples were taken for each parameter from the ranges in Table~\eqref{Table:Model-parameters}. The mean values of other parameter, shown in Table~\eqref{Tab:estimated-parameters-Table}, are fixed when computing the two aforementioned responses. $'*'$ represent the most influential parameters.}
	\label{Fig:Sensitivity}
\end{figure}

\section{Optimal Control strategies}~\label{Result}
We conduct numerical simulation to simulate the control problem~\eqref{EQ:Fractional-order-vector-borne-disease-model_with_control} using the parameter values of Tables~\ref{Tab:estimated-parameters-Table} with $\mu_{H}=2.4456E-04$, $N_{H}=2347833$, $C=1$ to establish an effective control policy. To perform optimal estimation we used forward-backward-sweep method~\cite{lenhart2007optimal} using a implicit-$\theta$ scheme given in \textbf{Appendix~A}.
To establish an effective control policy in San Juan province, Puerto Rico we consider different control structures to reduce dengue epidemic in this location. Different control combination are provided below:
\begin{enumerate}
	\item[$\bullet$] \textbf{Individual precaution only ($S_{1}$)}: In this scenario, only individual protection (using mosquito net, mosquito repellent, \textit{etc.}) is considered \textit{i.e.} $\psi \neq 0$ and $\zeta = \kappa = 0$. We found that applying this control may significantly reduce dengue cases in a location (see Table.~\ref{Tab:estimated-Control-total_cases-Table}).
	
	\item[$\bullet$] \textbf{Reducing aquatic transition only ($S_{2}$)}: In this control strategy, we considered reducing mosquito recruitment rate only (\textit{i.e.} killing mosquito egg, larvae, and pupae by spraying aquatic insecticides, and therefore, it reduces aquatic transition rate) \textit{i.e.} $\zeta \neq 0$, and $\psi = \kappa = 0$. We found that applying this control, dengue cases are reduced up to a certain time point but again, cases increased when the effect of $\zeta$ is reduced (Table.~\ref{Tab:estimated-Control-total_cases-Table}). 
	
	\item[$\bullet$] \textbf{Adult vector control only ($S_{3}$)}: In this control scenario, we considered adult vector control only (spraying adult insecticides) \textit{i.e.} $\kappa \neq 0$, and $\psi = \zeta = 0$. Applying this control, new dengue cases along with the total mosquito population reduced significantly (see Table~\ref{Tab:estimated-Control-total_cases-Table}). 
	
	\item[$\bullet$] \textbf{Individual precaution along with reduced aquatic transition ($S_{4}$)}: In this case, we considered a combination of individual protection and aquatic transition control \textit{i.e.} $\psi \neq 0$, $\zeta \neq 0$, and $\kappa = 0$. In this scenario, we found that optimal effect of individual protection dominates aquatic transition; as a result number of dengue cases reduce significantly but total mosquito population remain unchanged (see Table.~\ref{Tab:estimated-Control-total_cases-Table}). 
	
	\item[$\bullet$] \textbf{Individual precaution along with adult mosquito control ($S_{5}$)}: In this scenario, we considered combination of individual protection and adult vector control \textit{i.e.} $\psi \neq 0$, $\kappa \neq 0$, and $\zeta =0$. As similar to previous case, in this control scenario also optimal effect $\psi$ dominates $\kappa$. Therefore, the number of dengue cases reduce significantly, the total mosquito population remain unchanged (see Table.~\ref{Tab:estimated-Control-total_cases-Table}).
	
	\item[$\bullet$] \textbf{Adult mosquito control along with reduction in aquatic transition ($S_{6}$)}: In this scenario, we considered combination of adult vector control and reduction in mosquito recruitment rate \textit{i.e.} $\kappa \neq 0$, $\zeta \neq 0$, and $\psi = 0$. We found that in this optimal control scenario the incidence of total dengue cases and overall mosquito population are reduced significantly (see Table.~\ref{Tab:estimated-Control-total_cases-Table}). 
	
	\item[$\bullet$] \textbf{Combination of all control ($S_{7}$)}: In this scenario, we studied the optimal effect of all control ($\kappa \neq 0$, $\zeta \neq 0$, $\psi \neq 0$). In this case, optimal effect of individual protection dominates other two controls (adult vector control and reduction in mosquito recruitment rate). Therefore, optimal effect of all control reduce dengue cases significantly but total vector population remain unchanged (see Table.~\ref{Tab:estimated-Control-total_cases-Table}).  
\end{enumerate}
\begin{table}[H]
	\tabcolsep 4pt
	\centering
	\caption{\bf{Estimated optimal application rate and number of total dengue cases projected between 2009-2011 for the fractional-order dengue transmission model~(\ref{EQ:Fractional-order-vector-borne-disease-model}). Estimated parameter values are provided in the subsequent format [\textbf{Mean (95\% CI)}].}}
	\begin{tabular}{ p{6cm} p{6cm} p{4cm}}
		\hline\\ \footnotesize{{Control policy}} &
		\footnotesize{{Mean weekly optimal application rate}} & \footnotesize{{Total cases}}  \\
		\hline\\
		\footnotesize{$\psi(t)= \eta(t)=\kappa(t)=0$} & \footnotesize{$--$}  &
		\footnotesize{$\substack{ 5089 \\\\(4694 - 5559)}$} \\\\
		\footnotesize{$\psi(t)\neq 0, \eta(t)=\kappa(t)=0$} & \footnotesize{$\substack{0.3671\\\\ (0.3424 - 0.3850)}$}  &
		\footnotesize{$\substack{50 \\\\ (47 - 53)}$} \\\\
		\footnotesize{$\psi(t)= \kappa(t)=0, \eta(t)\neq0$} & \footnotesize{$\substack{0.0191\\\\ (5.96E-20 - 0.1157)}$} &
		\footnotesize{$\substack{47  \\\\(46 - 48)}$} \\\\
		\footnotesize{$\psi(t)=\eta(t)=0, \kappa(t)\neq0$} & \footnotesize{$\substack{ 0.2395\\\\ (0.0962 - 0.3076)}$} &
		\footnotesize{$\substack{46  \\\\(46 - 47)}$} \\\\
		\footnotesize{$\psi(t)\neq 0, \eta(t)\neq0, \kappa(t)=0$} & \footnotesize{$\substack{0.3681 (0.3530 - 0.3828)\\\\ 0.0230 (7.08E-20-0.1154 )}$} &
		\footnotesize{$\substack{50  \\\\ (47 - 54)}$}  \\\\
		\footnotesize{$\psi(t)=0, \eta(t)\neq0, \kappa(t)\neq 0$} & \footnotesize{$\substack{0.0750 (5.43E-20 -  0.2498)\\\\ 0.2110 ( 2.44E-18-0.3077 )}$} &
		\footnotesize{$\substack{47  \\\\ (46 - 54)}$}  \\\\
		\footnotesize{$\psi(t)\neq 0, \eta(t)=0, \kappa(t)\neq0$} & \footnotesize{$\substack{0.3532 (0.2020 - 0.3872)\\\\ 0.2526 (0.1336-0.3267 )}$} &
		\footnotesize{$\substack{47  \\\\ (46 - 49)}$}  \\\\
		\footnotesize{$\psi(t)\neq 0, \eta(t)\neq0, \kappa(t)\neq 0$} & \footnotesize{$\substack{0.1481 (5.10E-19 - 0.2883)\\\\ 0.3311 (0.1442-0.3920 )\\\\0.2545(0.1154 -0.3459)}$} &
		\footnotesize{$\substack{47  \\\\ (46 - 50)}$}  \\\\
		[0.2ex]
		\hline\\
	\end{tabular}
	\label{Tab:estimated-Control-total_cases-Table}
\end{table}

\begin{table}[H]
	\tabcolsep 4 pt
	\centering
	\caption{\bf{The estimate parameters of the fractional-order dengue transmission model~(\ref{EQ:Fractional-order-vector-borne-disease-model}). Estimated parameter values are in the format [\textbf{Mean (95\% CI)}].}}
	\begin{tabular}{ p{6cm} p{6cm} }
		\hline\\
		\footnotesize{{Parameter}}  & \footnotesize{{Estimated values }}  \\
		\hline\\
		\footnotesize{$\alpha$}   &
		\footnotesize{$\substack{0.2352 \\\\ (0.1472 - 0.2858 )}$} \\\\
		\footnotesize{$\beta$} &	\footnotesize{$\substack{0.9918 \\\\ (0.9159 - 0.9993 )}$} \\\\
		\footnotesize{$b$} &
		\footnotesize{$\substack{3.7578\\\\ (2.8098  - 4.1421 )}$ } \\\\
		\footnotesize{$\beta_{V}^{H}$} & \footnotesize{$\substack{0.0135 \\\\ (0.0115 - 0.0190)}$}  \\\\
		\footnotesize{$\beta_{H}^{V}$} &
		\footnotesize{$ \substack{0.9405 \\\\ (0.8243 - 0.9935)}$ }\\\\
		\footnotesize{$\delta$} & \footnotesize{$\substack{1.0470 \\\\ (1.0045 - 1.6046)}$}\\\\
		\footnotesize{$\mu_{V}$} & \footnotesize{$\substack{0.1428 \\\\ (0.1401 - 0.1586)}$}\\\\
		[0.2ex]
		\hline\\
	\end{tabular}
	\label{Tab:estimated-parameters-Table}
\end{table}
\vspace{5cm}

\section{Results and Discussion}~\label{conclusion}
Recent studies suggest that certain epidemic transmission and recovery processes may exhibit complex characteristics, deviating from the traditional Markovian distribution~\cite{angstmann2016fractional,sardar2017mathematical}. This implies that the dynamics of these processes may be influenced by non-Markovian factors, which could have significant implications for our understanding and management of epidemics~\cite{sardar2015mathematical,starnini2017equivalence,sardar2017mathematical, lin2020non, sherborne2018mean}. To study more on this fact, we developed a new single-serotype model~\eqref{EQ:Fractional-order-vector-borne-disease-model} that incorporates transmission and recovery processes governed by power-law distribution. We showed that these power-law distributions in transmission and recovery processes lead to a fractional-order dengue system~\eqref{EQ:Fractional-order-vector-borne-disease-model}. In this new dengue model~\eqref{EQ:Fractional-order-vector-borne-disease-model}, fractional derivative appeared as tempered fractional derivative~\cite{sabzikar2015tempered}. Mathematically, we derived an approximation (see Theorem~\ref{Theorem-1}) related to the tempered fractional derivative $\bigg( e^{-\beta t}\ _{RL}D_{t}^{\alpha}\left[e^{\beta t} F(t)\right] \approx \beta^{\alpha} F(t) \bigg)$ that appeared on the right-hand side of the system~\eqref{EQ:Fractional-order-vector-borne-disease-model}. We provide a mathematical proof regarding the positive invariance and the boundedness (see Proposition~\ref{prop-2} and Proposition~\ref{prop-3}) for every forward solution connected to the fractional-order dengue system~\eqref{EQ:Fractional-order-vector-borne-disease-model}. The basic reproduction number ($R_{0}$) for the fractional-order system~\eqref{EQ:Fractional-order-vector-borne-disease-model} was derived, which is dependent on both non-Markovian transmission and recovery process parameters ($\alpha$, $\beta$, and $p$). Moreover, we found that the fractional order dengue system~\eqref{EQ:Fractional-order-vector-borne-disease-model} possesses a unique disease free equilibrium point ($\Psi_{0}$) and using the approximation provided in Theorem~\ref{Theorem-1}, we studied the local and global stability of $\Psi_{0}$ from the perspective of the basic reproduction ratio $R_{0}$ (see Proposition~\ref{prop-4} and Proposition~\ref{disease_free_global_prop-4}). Furthermore, we found that the fractional order dengue system~\eqref{EQ:Fractional-order-vector-borne-disease-model} shows a unique endemic steady-state unique point ($\Psi_{*}$) under the condition $R_{0} > 1$. Using the Poincar\'e-Bendixson theorem along with a geometric framework~\cite{li2002global}, we validate the global stability of $\Psi_{*}$ in terms of $R_{0}$.

To explore the outcome of different interventions, we updated the equation~\eqref{EQ:Fractional-order-vector-borne-disease-model} into a fractional-order dengue model with control~\eqref{EQ:Fractional-order-vector-borne-disease-model_with_control}. Furthermore, we constructed an optimal control problem [see equation~\eqref{EQ:Cost-functional-optimal-control}] related to the fractional-order system~\eqref{EQ:Fractional-order-vector-borne-disease-model_with_control} to study the optimal effect of different dengue interventions to reduce infected individual in a location. Furthermore, we state and prove (see Theorem~\ref{Theorem-optimal-control}) the existence of a solution for optimal control problem related to the fractional order dengue model~\eqref{EQ:Fractional-order-vector-borne-disease-model_with_control}. Additionally, we simultaneously solve the dengue state system~\eqref{EQ:Fractional-order-vector-borne-disease-model_with_control} along with the co-state equation~\eqref{EQ:Costate-Fractional-Dengue-model} applying forward-backward-sweep computational technique~\cite{lenhart2007optimal}. Numerical discretization of the fractional derivatives appeared in both state system~\eqref{EQ:Fractional-order-vector-borne-disease-model_with_control} and co-state system~\eqref{EQ:Costate-Fractional-Dengue-model} are done using a implicit-$\theta$ scheme provided in \textbf{Appendix~A}. The parameters value used to obtaining the numerical solution for the control problem~\eqref{EQ:Cost-functional-optimal-control} are extracted from Table~\ref{Table:Model-parameters} and Table~\ref{Tab:estimated-parameters-Table}.

Since the analytical solution of equation~\eqref{EQ:Fractional-order-vector-borne-disease-model} is almost impossible to be determined, hence we developed a numerical scheme [see \textbf{Appendix~A} in~\eqref{sec:numerical-scheme}] to study the dynamics of the dengue system~\eqref{EQ:Fractional-order-vector-borne-disease-model}. Furthermore, using this scheme, we also determined the solution for the newly constructed optimal control problem [see Section~\ref{Optimal Control}] related to the fractional-order control model~\eqref{EQ:Fractional-order-vector-borne-disease-model_with_control}.

To establish our objective, we initially estimate several crucial parameters associated with the dengue system~\eqref{EQ:Fractional-order-vector-borne-disease-model}. We aim to validate the fractional order system~\eqref{EQ:Fractional-order-vector-borne-disease-model} using the data collected from San Juan area in Puerto Rico, between April 9, 2010, to April 2, 2011, to better understanding the dynamics involved. The model, as depicted in Figure~\ref{Fig:Model_fitting}, suggests that our model effectively captured the data, indicating a successful fit to the aforementioned data. Table~\ref{Tab:Control-weight-Table} provides estimated parameters for the fractional order system~\eqref{EQ:Fractional-order-vector-borne-disease-model}. The power-law coefficient ($\alpha$), representing transmission memory~\cite{sardar2015mathematical} in the model~\eqref{EQ:Fractional-order-vector-borne-disease-model} were observed to be less than unity (see Table~\ref{Tab:estimated-parameters-Table}). Based on the results, the memory effect in disease dynamics ($\alpha$: transmission memory) has significant impact on the transmission process in the particular region. This reflects a significant shift in the dynamics of transmission process, indicating a heightened role played by memory in disease process ($\alpha$) in this specific context. In addition, the data presented in Table~\eqref{Tab:estimated-parameters-Table}  illustrates that a increasing in memory during the recovery process has the potential to regulate the disease process. This clearly highlights the importance of performing a sensitivity analysis to investigate the correlation between a specific parameter and two key outcomes: total number of dengue cases, and basic reproduction ratio $(R_{0})$.

To more accurately asses the impact of parameters on dynamics of the model, global sensitivity analysis performed on the basic reproduction index $R_{0}$ and the total number of dengue cases, respectively. Thus, it follows that the parameters that impact the total dengue cases are the fractional derivative, which is known as memory effect in transmission of infection process $\alpha$, $\beta$ the memory impact in recovery, $\delta$ the ratio between total vector and human population and $\mu_{V}$ mosquito date rate~(see Figure~\ref{Fig:Sensitivity}). Factors influencing basic reproduction number $R_{0}$ are average number of mosquito bites $b$, rate of transmission from infected human to susceptible vector $\beta_{H}^{V}$, rate of transmission of susceptible human from infected vector $\beta_{V}^{H}$, followed by $\alpha$, $\beta$, $\delta$, and $\mu_{V}$~(see Figure~\ref{Fig:Sensitivity}). As biologically, $\alpha$ measures some memory effect in disease transmission~\cite{du2013measuring, saeedian2017memory, sardar2017mathematical} and therefore, increasing memory impact ($\alpha \to 0$ ) in transmission process leads to longer epidemic duration and fewer dengue incidences in a region. This result is in agreement with some previous studies on fractional-order transmission models~\cite{sardar2015mathematical,sardar2017mathematical}. The power-law recovery process parameter $\beta$ also has a positive correlation to the $R_{0}$ (see Figure~\ref{Fig:Sensitivity}). However, it may significantly change the cumulative dengue cases in a location (see Figure~\ref{Fig:Sensitivity}). Biologically, $\beta$ may represent a measure of recovery rate, and thus $\beta \to 0$ leads to faster recovery and subsequently, produce lesser number of dengue cases in a location (see Figure~\ref{Fig:Sensitivity}). This observation suggests that a decline in memory during the recovery process ($\beta to 1$) correlates with an increase in both the reproduction number and the total cases. In addition, Fig.~\ref{Fig:Sensitivity} indicates that raising the mortality rate of mosquitoe can reduce the basic reproduction number $R_{0}$ and consequently the disease burden.

Our study brings up an intriguing discussion about potential control strategies for managing the new dengue case and threshold quantity $R_{0}$. This exploration presents an opportunity to comprehensively assess and contemplate the various strategies available. According to Table~\ref{Tab:estimated-Control-total_cases-Table}, the enforcement of different control strategies leads a noticeable change in the number of total dengue cases. Existing intervention strategies to control dengue transmission encompass mosquito control through the use of insecticides, individual precautionary measures, and heightened human awareness regarding the utilization of mosquito nets or alternative repellents~\cite{sardar2015mathematical,pinheiro1997global}. Increasing memory in disease process ($\alpha \to 0$) in transmission process in effective way (using mosquito net, mosquito repellent, etc.) to reduce biting rate $b$ which decrease the basic reproduction number (see Fig~\ref{Fig:Sensitivity}). Out of the seven intervention setting, utilizing all three control strategies, including non-pharmaceutical intervention, behavioral intervention, and environmental intervention, either individually or in combination, has been demonstrated to naturally decrease the number of maximum cases compared to scenario without any intervention. This comprehensive approach has shown promising results in mitigating the disease spread and minimizing the overall burden on public health. Also, comparing the optimal effect of three dengue interventions and all their possible combinations suggests that adult mosquito control may provide the best result in terms of dengue case reduction (see Table~\ref{Tab:estimated-Control-total_cases-Table}). Thus, policymakers may focus on fractional-order transmission and recovery parameters ($\alpha$, $\beta$) along with effective vector control (adult) to decrease the incidence of dengue transmission in a region.

Lastly, a few shortcomings of the current study can be noted, and they might be expanded upon in the future. There is evidence that asymptomatic and exposed cases of vector-borne diseases can transmit the disease to others~\cite{duong2015asymptomatic, gumel2012causes, eshita2007vector, wilder2004seroepidemiology}. Incorporating these disease stages may lead to more complex results. We left this for future projects.

\bibliographystyle{ieeetr}
\biboptions{square}
\bibliography{Manuscript_bib}

\section{Appendix $A$. Numerical-Scheme}~\label{sec:numerical-scheme}
\subsection{An implicit-$\theta$ method}
To develop a numerical algorithm for the initial value problem~(\ref{EQ:Fractional-order-vector-borne-disease-model}), we used a one-step Euler method. We begin with the rectangle scheme to discretize the RL-fractional derivatives that appeared in the dengue model~(\ref{EQ:Fractional-order-vector-borne-disease-model}). 
\subsection{Rectangle method:} Let us assume that $[0, t_{max}]$ be the interval for which the problems~(\ref{EQ:Fractional-order-vector-borne-disease-model}) are well defined. We subdivide this interval into $(T+1)$ uniform grid by $t_{i} = i h, i= 0, 1,....T$, where, $h > 0$, and $T = \left \lceil \frac{t_{max}}{h} \right \rceil $ is a positive integer. Furthermore, we assume that $S_{H}^{i} =S_{H}(t_{i}) ,I_{H}^{i} =I_{H}(t_{i}), R_{H}^{i} =R_{H}(t_{i}), S_{V}^{i} =S_{V}(t_{i}), I_{V}^{i} =I_{V}(t_{i}),$ for $i=0,1....T$ are the approximation of $S_{H}$, $I_{H}$ and $I_{V} $ respectively at the point $t_{i}$.
We approximate the system~(\ref{EQ:Fractional-order-vector-borne-disease-model}) as:

\begin{eqnarray}
	\begin{array}{llll}
		\displaystyle	\frac{S_{H}^{i+1}-S_{H}^{i}}{h}&=& 	\displaystyle \mu_{H}N_{H}- \frac{b^{\alpha}}{N_{H}} \beta_{V}^{H} S_{H}^{i+1} e^{-\mu_{V} (i+1)h}~ _{RL}{D_{t}}^{1-\alpha}\left[I_{V}(t) e^{\mu_{V} t} \right]\Big|_{t=t_{i}}-\mu_{H}  S_{H}^{i+1}\\
		\displaystyle	\frac{I_{H}^{i+1}-I_{H}^{i+}}{h}& =& 	\displaystyle \frac{b^{\alpha}}{N_{H}} \beta_{V}^{H} S_{H}^{i+1} e^{-\mu_{V} (i+1)h}~ _{RL}{D_{t}}^{1-\alpha}\left[I_{V}(t) e^{\mu_{V} t} \right]\Big|_{t=t_{i}}\\ &-& 	\displaystyle \frac{e^{-\mu_{H} (i+1)h}}{C^{\beta}}~_{RL}D_{t}^{1-\beta}\left[I_{H}(t) e^{\mu_{H} t} \right]\Big|_{t=t_{i}} -\mu_{H}  I_{H}^{i+1},\\
		\displaystyle	\frac{R_{H}^{i+1}-R_{H}^{i}}{h} &=& 	\displaystyle\frac{e^{-\mu_{H} (i+1)h}}{C^{\beta}}~_{RL}D_{t}^{1-\beta}\left[I_{H}(t) e^{\mu_{H} t} \right]\Big|_{t=t_{i}} -\mu_{H}  R_{H}^{i+1},\\
		\displaystyle	\frac{S_{V}^{i+1}-S_{V}^{i}}{h}&=& 	\displaystyle \Pi_{V}-\frac{b^{p}}{N_{H} C^{p}} \beta_{H}^{V} S_{V}^{i+1} e^{-\mu_{H} (i+1)h} ~_{RL}D_{t}^{1-p}\left[I_{H}(t) e^{\mu_{H} t}\right]\Big|_{t=t_{i}} -\mu_{V}  S_{V}^{i+1},\\
		\displaystyle	\frac{I_{V}^{i+1}-I_{V}^{i}}{h}&=& 	\displaystyle \frac{b^{p}}{N_{H} C^{p}} \beta_{H}^{V} S_{V}^{i+1} e^{-\mu_{H} (i+1)h} ~_{RL}D_{t}^{1-p}\left[I_{H}(t) e^{\mu_{H}t}\right]\Big|_{t=t_{i}} -\mu_{V}  I_{V}^{i+1}.
	\end{array}
	\label{EQ:Rectangle-method-0}
\end{eqnarray}

Now, approximate fractional derivatives of the system~\eqref{EQ:Fractional-order-vector-borne-disease-model} in right-hand side. From definition~\eqref{Definition-1}, we have:
\begin{eqnarray}
	\begin{array}{llll}
		\displaystyle	_{RL}{D_{t}}^{1-\alpha}\Big[I_{V}(t) e^{\mu_{V} t} \Big]& = & 	\displaystyle \frac{1}{\Gamma(\alpha)} \frac{d}{dt} \int_{0}^{t} (t-s)^{\alpha -1} e^{\mu_{V} s} I_{V}(s) ds,\\\\&=& 	\displaystyle \frac{1}{\Gamma(\alpha)} \frac{dH(t)}{dt}, 
	\end{array}
	\label{EQ:Rectangle-method-1}
\end{eqnarray}
where, $\displaystyle H(t) = \int_{0}^{t} (t-s)^{\alpha -1} e^{\mu_{V} s} I_{V}(s) ds$.

Hence, using the vale $H(t)$ in~(\ref{EQ:Rectangle-method-1}) produced: 
\begin{eqnarray}
	\begin{array}{llll}
		\displaystyle	_{RL}{D_{t}}^{1-\alpha}\left[I_{V}(t) e^{\mu_{V} t} \right]\Big|_{t=t_{i}}& = & \displaystyle \frac{1}{\Gamma(\alpha)} \left[ \frac{H(t_{i+1})-H(t_{i})}{h}\right]\\\\& = &  \displaystyle \frac{e^{\mu_{V} (i+1)h}}{h}\left[  \sum_{j=0}^{i} \omega_{j,i}^{I_{V},R} I_{V}(t_{j+1}) + e^{-\mu_{V}h}\sum_{j=0}^{i}\tilde{ \omega}_{j,i}^{I_{V},R} I_{V}(t_{j+1})\right] 
	\end{array}
	\hspace{0.6cm}\label{EQ:Rectangle-method-2}
\end{eqnarray}

where,

$\omega_{j,i}^{I_{V},R}=\frac{1}{\mu_{V}^{\alpha} \Gamma(\alpha)} \left\{\begin{array}{ll}
	\gamma(\mu_{V}h,\alpha), & \mathrm{if}\text{ }(j=i) \\
	\left[ \gamma(\mu_{V}(i+1-j)h,\alpha)-\gamma(\mu_{V}(i-j)h,\alpha)\right]  & \mathrm{if}\text{ }(j=0,1...(i-1)) \\
\end{array}\right.\\ $
$\tilde{ \omega}_{j,i}^{I_{V},R}=\frac{1}{\mu_{V}^{\alpha} \Gamma(\alpha)}  \left\{\begin{array}{ll}
	0, & \mathrm{if}\text{ }(j=i) \\
	\left[ \gamma(\mu_{V}(i-1-j)h,\alpha)-\gamma(\mu_{V}(i-j)h,\alpha)\right]  & \mathrm{if}\text{ }(j=0,1...(i-1)) .\\
\end{array}\right. \\$

\subsection{Trapezoidal method:}
We use the piece-wise linear trapezoidal quadrature interpolation formula for the function given below:

$S_{H}(t) = \frac{t-t_{i+1}}{t_{i}-t_{i+1}} S_{H}(t_{i}) +\frac{t-t_{i}}{t_{i+1}-t_{i}} S_{H}(t_{i+1})$,\\
$I_{H}(t) = \frac{t-t_{i+1}}{t_{i}-t_{i+1}} I_{H}(t_{i}) +\frac{t-t_{i}}{t_{i+1}-t_{i}} I_{H}(t_{i+1})$,\\
$R_{H}(t) = \frac{t-t_{i+1}}{t_{i}-t_{i+1}} R_{H}(t_{i}) +\frac{t-t_{i}}{t_{i+1}-t_{i}} R_{H}(t_{i+1})$,\\
$S_{V}(t) = \frac{t-t_{i+1}}{t_{i}-t_{i+1}} S_{V}(t_{i}) +\frac{t-t_{i}}{t_{i+1}-t_{i}} S_{V}(t_{i+1})$,\\
$I_{V}(t) = \frac{t-t_{i+1}}{t_{i}-t_{i+1}} I_{V}(t_{i}) +\frac{t-t_{i}}{t_{i+1}-t_{i}} I_{V}(t_{i+1}),$\\
for $i=0,1....T$ are the approximation of $S_{H}, I_{H}$, and $I_{V}$ respectively at the point $t_{i}$ and $t_{i+1} $. 
Then approximate the term $_{RL}{D_{t}}^{1-\alpha}\left[
I_{V}(t) e^{\mu_{V} t} \right]$ at $t=t_{i}$ as follows:

\begin{eqnarray}
	\begin{array}{llll}
		\displaystyle	_{RL}{D_{t}}^{1-\alpha}\left[
		I_{V}(t) e^{\mu_{V} t} \right]\Big|_{t=t_{i}} &=& 	\displaystyle \frac{1}{\Gamma(\alpha)} \left[ \frac{H(t_{i+1})-H(t_{i})}{h}\right]\\ &=&	\displaystyle \frac{e^{\mu_{V} (i+1)h}}{h^{2}}\left[  \sum_{j=0}^{i} \left( \omega_{j,i}^{I_{V},T} -e^{-\mu_{V}h} U_{j,i}^{I_{V},T}  \right) I_{V}(t_{j})\right. \\ &&	\displaystyle\left. \quad + \sum_{j=0}^{i}\left( \tilde{\omega}_{j,i}^{I_{V},T} -e^{-\mu_{V}h} \tilde{U}_{j,i}^{I_{V},T}  \right) I_{V}(t_{j+1})\right],
	\end{array}
	\hspace{0.6cm}\label{EQ:Trapizoidal-method-1}
\end{eqnarray}

where,

$\omega_{j,i}^{I_{V},T}=\frac{1}{\mu_{V}^{\alpha+1} \Gamma(\alpha)}  \left\{\begin{array}{ll}
	\gamma(\mu_{V}h,\alpha+1), & \mathrm{if}\text{ }(j=i) \\
	\left[ \gamma(\mu_{V}(i+1-j)h,\alpha+1)-\gamma(\mu_{V}(i-j)h,\alpha+1)\right]\\ 
	-\mu_{V} (i-j)h\left[ \gamma(\mu_{V}(i+1-j)h,\alpha)-\gamma(\mu_{V}(i-j)h,\alpha)\right]  & \mathrm{if}\text{ }(j=0,1...(i-1)) \\
\end{array}\right.\\ $
$U_{j,i}^{I_{V},T}=\frac{1}{\mu_{V}^{\alpha+1} \Gamma(\alpha)}  \left\{\begin{array}{ll}
	0, & \mathrm{if}\text{ }(j=i) \\
	\left[ \gamma(\mu_{V}(i-j)h,\alpha+1)-\gamma(\mu_{V}(i-j-1)h,\alpha+1)\right]\\ 
	-\mu_{V} (i-j-1)h\left[ \gamma(\mu_{V}(i-j)h,\alpha)-\gamma(\mu_{V}(i-j-1)h,\alpha)\right]  & \mathrm{if}\text{ }(j=0,1...(i-1)) \\
\end{array}\right.\\ $
$\tilde{ \omega}_{j,i}^{I_{V}T}=\frac{1}{\mu_{V}^{\alpha+1} \Gamma(\alpha)}  \left\{\begin{array}{ll}
	\mu_{V}h\gamma(\mu_{V}h,\alpha)-\gamma(\mu_{V}h,\alpha+1), & \mathrm{if}\text{ }(j=i) \\
	\mu_{V} (i+1-j)h\left[ \gamma(\mu_{V}(i+1-j)h,\alpha)-\gamma(\mu_{V}(i-j)h,\alpha)\right]\\ 
	-\left[ \gamma(\mu_{V}(i+1-j)h,\alpha+1)-\gamma(\mu_{V}(i-j)h,\alpha+1)\right]  & \mathrm{if}\text{ }(j=0,1...(i-1))\\
\end{array}\right. \\$
$\tilde{ U}_{j,i}^{I_{V}T}=\frac{1}{\mu_{V}^{\alpha+1} \Gamma(\alpha)}  \left\{\begin{array}{ll}
	0, & \mathrm{if}\text{ }(j=i) \\
	\mu_{V} (i-j)h\left[ \gamma(\mu_{V}(i-j)h,\alpha)-\gamma(\mu_{V}(i-j-1)h,\alpha)\right]\\ 
	-\left[ \gamma(\mu_{V}(i-j)h,\alpha+1)-\gamma(\mu_{V}(i-j-1)h,\alpha+1)\right]  & \mathrm{if}\text{ }(j=0,1...(i-1)) \\
\end{array}\right. \\$
\subsection{Implicit $\theta$-Method} Let the integral  $_{RL}{D_{t}}^{1-\alpha}\left[
y(t) e^{\mu_{m} t} \right]$ at $t=t_{i}$ can be discritize using $ \theta-method $, which is defined as follows:
\begin{eqnarray}~\label{EQ:Theta-method}
	\begin{array}{llll}
		\displaystyle	_{RL}{D_{t}}^{1-\alpha}\left[
		y(t) e^{\mu_{m} t} \right]\Big|_{t=t_{i}}  &=& \displaystyle \frac{e^{\mu_{m} (i+1)h}}{h}\left\lbrace  \theta \left[  \sum_{j=0}^{i} \omega_{j,i}^{y,R} y(t_{j+1}) + e^{-\mu_{m}h}\sum_{j=0}^{i}\tilde{ \omega}_{j,i}^{y,R} y(t_{j+1})\right]\right. \\& & \displaystyle\left. \quad+  \frac{(1-\theta)}{h}\left[  \sum_{j=0}^{i} \left( \omega_{j,i}^{y,T} -e^{-\mu_{m}h} U_{j,i}^{y,T}  \right) y(t_{j})\right.\right. \\& & \displaystyle\left.\left. \quad + \sum_{j=0}^{i}\left( \tilde{\omega}_{j,i}^{y,T} -e^{-\mu_{m}h} \tilde{U}_{j,i}^{y,T}  \right) y(t_{j+1})\right]\right\rbrace,
	\end{array}
\end{eqnarray}
where, we consider the $\theta-method$  as the combination of Rectangle and Trapezoidal method defined above~(\ref{EQ:Rectangle-method-2}) and~(\ref{EQ:Trapizoidal-method-1}). Then the differential equation
$\frac{dy}{dt} =   _{RL}{D_{t}}^{1-\alpha}\left[
y(t) e^{\mu_{m} t} \right]$, can be discritize using Euler method as follows

\begin{eqnarray}~\label{EQ:Theta-method-1}
	\begin{array}{llll}
		\displaystyle	y_{i+1} &=& \displaystyle y_{i}+\frac{e^{\mu_{m} (i+1)h}}{h}\left\lbrace  \theta \left[  \sum_{j=0}^{i} \omega_{j,i}^{y,R} y(t_{j+1}) + e^{-\mu_{m}h}\sum_{j=0}^{i}\tilde{ \omega}_{j,i}^{y,R} y(t_{j+1})\right]\right. \\ && \displaystyle \left. \quad+  \frac{(1-\theta)}{h}\left[  \sum_{j=0}^{i} \left( \omega_{j,i}^{y,T} -e^{-\mu_{m}h} U_{j,i}^{y,T}  \right) y(t_{j})\right.\right. \\& & \displaystyle \left.\left. \quad + \sum_{j=0}^{i}\left( \tilde{\omega}_{j,i}^{y,T} -e^{-\mu_{m}h} \tilde{U}_{j,i}^{y,T}  \right) y(t_{j+1})\right]\right\rbrace,
	\end{array}
\end{eqnarray}

Now, $ _{RL}{D_{t}}^{1-\alpha}\left[
I_{V}(t) e^{\mu_{V} t}\right] $ descritize using~\eqref{EQ:Theta-method} as:
\begin{eqnarray}
	\begin{array}{llll}
		\displaystyle	_{RL}{D_{t}}^{1-\alpha}\left[
		I_{V}(t) e^{\mu_{V} t} \right]\Big|_{t=t_{i}}  &=& \displaystyle \frac{e^{\mu_{V} (i+1)h}}{h}\left\lbrace  \theta \left[  \sum_{j=0}^{i} \omega_{j,i}^{I_{V},R} I_{V}(t_{j+1}) + e^{-\mu_{V}h}\sum_{j=0}^{i}\tilde{ \omega}_{j,i}^{I_{V},R} I_{V}(t_{j+1})\right]\right. \\&& \displaystyle \left. \quad+  \frac{(1-\theta)}{h}\left[  \sum_{j=0}^{i} \left( \omega_{j,i}^{I_{V},T} -e^{-\mu_{V}h} U_{j,i}^{I_{V},T}  \right) I_{V}(t_{j})\right.\right. \\ && \displaystyle\left.\left. \quad + \sum_{j=0}^{i}\left( \tilde{\omega}_{j,i}^{I_{V},T} -e^{-\mu_{V}h} \tilde{U}_{j,i}^{I_{V},T}  \right) I_{V}(t_{j+1})\right]\right\rbrace. \\
	\end{array}
\end{eqnarray}
Then, use~\eqref{EQ:Theta-method-1} and~\eqref{EQ:Theta-method-2} the equation~\eqref{EQ:Rectangle-method-0} can be transformed as system of equation:
\begin{eqnarray}
	\begin{array}{llll}
		\displaystyle A_{S_{H}}S_{H}^{i+1,\theta}+B_{S_{H}}S_{H}^{i+1,\theta}I_{V}^{i+1,\theta}+C_{S_{H}}&=0,\\\\
		\displaystyle A_{I_{H}}S_{H}^{i+1,\theta}+B_{I_{H}}I_{H}^{i+1,\theta}+C_{I_{H}}S_{H}^{i+1,\theta}I_{V}^{i+1,\theta}+D_{I_{H}}&=0,\\\\
		\displaystyle	A_{R_{H}}R_{H}^{i+1,\theta}+B_{R_{H}}I_{H}^{i+1,\theta}+C_{R_{H}}&=0,\\\\
		\displaystyle A_{S_{V}}S_{V}^{i+1,\theta}+B_{S_{V}}S_{V}^{i+1,\theta}I_{H}^{i+1,\theta}+C_{S_{V}}&=0,\\\\
		\displaystyle A_{I_{V}}S_{V}^{i+1,\theta}+B_{I_{V}}S_{V}^{i+1,\theta}I_{H}^{i+1,\theta}+C_{I_{V}}I_{V}^{i+1,\theta}+D_{I_{V}}&=0.
	\end{array}
	\label{EQ:Theta-method-2}
\end{eqnarray}
where,
\begin{eqnarray}
	\begin{array}{llll}
		\displaystyle A_{S_{H}} &=& 1+\mu_{H} h\\\\&
		+& 	\displaystyle \frac{b^{\alpha}}{N_{H}} \beta_{V}^{H} \left[ \sum_{j=0}^{i-1} \left( \theta \left\lbrace\omega_{j,i}^{I_{V},R}  + e^{- \mu_{V}h}\tilde{ \omega}_{j,i}^{I_{V},R} \right\rbrace+\frac{(1-\theta)}{h} \sum_{j=0}^{i-1}\left( \tilde{\omega}_{j,i}^{I_{V},T} -e^{-\mu_{V}h} \tilde{U}_{j,i}^{I_{V},T} \right)\right)  I_{V}(t_{j+1}) \right.\\\\&& 	\displaystyle \left. \quad+\frac{(1-\theta)}{h}\sum_{j=0}^{i} \left( \omega_{j,i}^{I_{V},T} -e^{-\mu_{V} h} U_{j,i}^{I_{V},T}  \right) I_{V}(t_{j})\right], \\\\
		\displaystyle	B_{S_{H}}&=& 	\displaystyle \frac{b^{\alpha}}{N_{H}} \beta_{V}^{H} \left( \theta \left\lbrace\omega_{i,i}^{I_{V},R}  + e^{- \mu_{V}h}\tilde{ \omega}_{i,i}^{I_{V},R} \right\rbrace+\frac{(1-\theta)}{h}\left\lbrace  \tilde{\omega}_{i,i}^{I_{V},T} -e^{-\mu_{V}h} \tilde{U}_{i,i}^{I_{V},T} \right\rbrace\right),  \\\\
		\displaystyle	C_{S_{H}}&=& 	\displaystyle-S_{H}^{i,\theta}-h \mu_{H} N_{H}.
	\end{array}
\end{eqnarray}

\begin{eqnarray}
	\begin{array}{llll}
		\displaystyle	A_{I_{H}} &=& \displaystyle1+h\mu_{H}-A_{S_{H}},\\\\
		\displaystyle	B_{I_{H}}&=& \displaystyle 1+h\mu_{H}+\frac{1}{C^{\beta}} \left( \theta \left\lbrace\omega_{i,i}^{I_{H},R}  + e^{-\mu_{H}h}\tilde{ \omega}_{i,i}^{I_{H},R} \right\rbrace+\frac{(1-\theta)}{h} \left( \tilde{\omega}_{i,i}^{I_{H},T} -e^{- \mu_{H}h} \tilde{U}_{i,i}^{I_{H},T} \right)\right), \\\\
		\displaystyle	C_{I_{H}}&=& \displaystyle -B_{S_{H}}, \\\\
		\displaystyle	D_{I_{H}}&=& \displaystyle -I_{H}^{i,\theta}+\frac{1}{C^{\beta}}\left[ \sum_{j=0}^{i-1} \left( \theta \left\lbrace\omega_{j,i}^{I_{H},R}  + e^{-\mu_{H}h}\tilde{ \omega}_{j,i}^{I_{H},R} \right\rbrace+\frac{(1-\theta)}{h} \sum_{j=0}^{i-1}\left( \tilde{\omega}_{j,i}^{I_{H},T} -e^{- \mu_{H}h} \tilde{U}_{j,i}^{I_{H},T} \right)\right)  I_{H}(t_{j+1}) \right. \\ && \displaystyle\left. \quad+\frac{(1-\theta)}{h}\sum_{j=0}^{i} \left( \omega_{j,i}^{I_{H},T} -e^{- \mu_{H} h} U_{j,i}^{I_{H},T}  \right) I_{H}(t_{j})\right],\\\\
		\displaystyle	A_{R_{H}} &=&  \displaystyle 1+h \mu_{H},\\\\
		\displaystyle	B_{R_{H}}&=& \displaystyle1+h \mu_{H}-B_{I_{H}}, \\\\
		\displaystyle	C_{R_{H}}&=& \displaystyle -R_{H}^{i,\theta}-I_{H}^{i,\theta}-D_{I_{H}}, \\\\
		\displaystyle	A_{S_{V}} &=& \displaystyle 1+\mu_{V} h\\&& \displaystyle
		+ \frac{b^{\alpha} \beta_{H}^{V}}{N_{H}C^{p}}\left[ \sum_{j=0}^{i-1} \left( \theta \left\lbrace\omega_{j,i}^{I_{H},R}  + e^{-\mu_{H}h}\tilde{ \omega}_{j,i}^{I_{H},R} \right\rbrace+\frac{(1-\theta)}{h} \sum_{j=0}^{i-1}\left( \tilde{\omega}_{j,i}^{I_{H},T} -e^{- \mu_{H}h} \tilde{U}_{j,i}^{I_{H},T} \right)\right)  I_{H}(t_{j+1}) \right. \\&& \displaystyle \left. \quad+\frac{(1-\theta)}{h}\sum_{j=0}^{i} \left( \omega_{j,i}^{I_{H},T} -e^{- \mu_{H} h} U_{j,i}^{I_{H},T}  \right) I_{H}(t_{j})\right],\\\\
		\displaystyle	B_{S_{V}}&=& \displaystyle  \frac{b^{\alpha} \beta_{H}^{V}}{N_{H}C^{p}} \left( \theta \left\lbrace\omega_{i,i}^{I_{H},R}  + e^{- \mu_{H}h}\tilde{ \omega}_{i,i}^{I_{H},R} \right\rbrace+\frac{(1-\theta)}{h}\left\lbrace  \tilde{\omega}_{i,i}^{I_{H},T} -e^{- \mu_{H}h} \tilde{U}_{i,i}^{I_{H},T} \right\rbrace\right),\\\\
		\displaystyle	C_{S_{V}}&=& \displaystyle -S_{V}^{i,\theta}-h \Pi_{V},\\\\
		\displaystyle	A_{I_{V}} &=& \displaystyle 1+ h \mu_{V}-A_{S_{V}}\\
		\displaystyle	B_{I_{V}}&=& \displaystyle-B_{S_{V}},\\\\
		\displaystyle	C_{I_{V}}&=& \displaystyle 1+h  \mu_{V},\\\\
		\displaystyle	D_{I_{V}}&=& \displaystyle -I_{V}^{i,\theta}.\\\\
	\end{array}
\end{eqnarray}

\end{document}